\documentclass{article}
\usepackage[T1]{fontenc}
\usepackage[utf8]{inputenc}
\usepackage[english]{babel}
\usepackage{geometry}

\usepackage{csquotes}
\usepackage{changepage}

\usepackage{palatino}
\usepackage{amssymb}
\usepackage{amsthm}
\usepackage{amsmath}
\usepackage{mathtools}
\usepackage{euler}
\usepackage{hyperref}
\usepackage{todonotes}
\usepackage{tikz-cd}
\usetikzlibrary{arrows}
\usepackage{multicol}

\usepackage{ragged2e}

\newtheorem{thm}{Theorem}[subsection]
\newtheorem{lemma}[thm]{Lemma}
\newtheorem{corollary}[thm]{Corollary}
\newtheorem{prop}[thm]{Proposition}

\theoremstyle{definition}
\newtheorem{defi}[thm]{Definition}

\newtheorem*{notation}{Notation}
\newtheorem{example}{Example}
\newtheorem{rmk}[thm]{Remark}

\DeclareMathOperator{\A}{\mathbb{A}}
\DeclareMathOperator{\C}{\mathbb{C}}
\DeclareMathOperator{\Z}{\mathbb{Z}}

\DeclareMathOperator{\F}{\mathcal{F}}
\DeclareMathOperator{\Y}{\mathcal{Y}}
\DeclareMathOperator{\Pg}{\mathcal{P}}

\DeclareMathOperator{\g}{\mathfrak{g}}
\DeclareMathOperator{\h}{\mathfrak{h}}
\DeclareMathOperator{\bg}{\mathfrak{b}}
\DeclareMathOperator{\ad}{ad\,}

\DeclareMathOperator{\spec}{Spec\,}
\DeclareMathOperator{\res}{Res\,}
\DeclareMathOperator{\Der}{Der\,}
\DeclareMathOperator{\Hom}{Hom\,}
\DeclareMathOperator{\End}{End}
\DeclareMathOperator{\Sym}{Sym}
\DeclareMathOperator{\Alg}{Alg}
\DeclareMathOperator{\Lie}{Lie}
\DeclareMathOperator{\spn}{Span}
\DeclareMathOperator{\id}{id\,}
\DeclareMathOperator{\ev}{ev}
\DeclareMathOperator{\rank}{rank}
\DeclareMathOperator{\gr}{gr\,}
\DeclareMathOperator{\Aut}{Aut\,}
\DeclareMathOperator{\vac}{|0\rangle}
\DeclareMathOperator{\restr}{Rest}
\DeclareMathOperator{\restrp}{rest}

\title{A Feigin-Frenkel theorem with $n$ singularities}
\author{Luca Casarin}
\date{}

\begin{document}


\maketitle

\begin{center}
    \textbf{Abstract}
    
    \smallskip
    \justifying
    For a simple Lie algebra $\g$ we consider an analogue of the affine algebra $\hat{\g}_k$ with $n$ singularities, defined starting from the ring of functions on the $n$-pointed disk. We study the center of its completed enveloping algebra and prove an analogue of the Feigin-Frenkel theorem in this setting. In particular, we first give an algebraic description of the center by providing explicit topological generators; we then characterize the center geometrically as the ring of functions on the space of $^LG$-Opers over the $n$-pointed disk. Finally, we prove some factorization properties of our isomorphism thus establishing a relation between our isomorphism and the usual isomorphism of Feigin-Frenkel.
\end{center}

\tableofcontents

\section{Introduction}

This paper was inspired by the work of Fortuna, Lombardo, Maffei and Melani \cite{fortuna2020local}, where they constructed a canonical isomorphism between two commutative algebras: 
\begin{enumerate}
\item the center of the completed enveloping algebra of $\widehat{(\mathfrak{sl}_2)}_{k_c,2}$, the $2$-singularities analogue of the affine algebra of $\mathfrak{sl}_2$ at the critical level;
\item the algebra of functions on the space of Opers over the $2$-pointed disk.
\end{enumerate}
The main goal of these pages is to generalize their result to an arbitrary simple Lie algebra $\g$ and to an arbitrary number of singularities.

\subsection{Main characters and goals}

Let $\g$ be a simple finite dimensional Lie algebra over $\C$. We briefly recall the construction of the affine algebra $\hat{\g}_{\kappa}$. Let $\kappa$ be an invariant symmetric form on $\g$, which must be a scalar multiple of the Killing form $\kappa_{\g}$. The Lie algebra $\hat{\g}_{\kappa}$ is obtained as a $1$ dimensional central extension of the loop algebra $\g((t)) = \g\otimes\C((t))$. So as a vector space $\hat{\g}_{\kappa} = \g((t)) \oplus \C\mathbf{1}$, where $\mathbf{1}$ is a central element and the bracket is defined by
\[
    [Xf,Yg] = [X,Y]fg + \kappa(X,Y)\mathbf{1}\int (gdf) \quad \text{for }\; X,Y \in \mathfrak{g} \text{ and } f,g \in \C((t)),
\]
where $\int : \Omega^{1,\text{cont}}_{\C((t))/\C} = \C((t))dt \to \C$ is the residue map. The natural topology on $\C((t))$ induces a structure of topological Lie algebra on $\hat{\g}_{\kappa}$. One can use this topology to construct a completion of the enveloping algebra of $\hat{\g}_{\kappa}$, which after having factored out the ideal $(\mathbf{1} - 1)$ is denoted by $\tilde{U}_{\kappa}(\hat{\mathfrak{g}})$.

The Feigin-Frenkel theorem \cite{feigin1992affine} deals with the description of the center of the completed enveloping algebra $\tilde{U}_{\kappa_c}(\hat{\mathfrak{g}})$ at the \textbf{critical level} $\kappa_c = -\frac{1}{2}\kappa_{\g}$.

\begin{thm}[Feigin-Frenkel 1992]
    Let  $\tilde{U}_{\kappa}(\hat{\mathfrak{g}})$ be as above. Then if $\kappa \neq \kappa_c$ the center of $\tilde{U}_{\kappa}(\hat{\mathfrak{g}})$ is spanned by $1$. For $\kappa = \kappa_c$ there is a canonical, $\Der \C((t))$-equivariant isomorphism
    \[
            Z(\tilde{U}_{\kappa_c}(\hat{\mathfrak{g}})) = \C[Op_{^L\mathfrak{g}}(D^*)]
    \]
    of the center of $\tilde{U}_{\kappa_c}(\hat{\mathfrak{g}})$ with the algebra of functions on the space of Opers, relative to the Langlands dual Lie algebra $^L\mathfrak{g}$, on the pointed disk $D^* = \spec \C((t))$.
 
\end{thm}

The complete ring $\C((t))$ is the algebra of functions on the formal disk which admit singularity only in $0 \in D$. It is then natural to try to look for an analogue construction allowing singularities in multiple points $a_1,\dots,a_n$, which we would like to "move freely".

This is formalized as follows. In place of $\C((t))$ we consider the complete topological ring
\[
    K_n = A_n[[t]]\bigg[\prod(t-a_i)^{-1}\bigg] \quad \text{where } \quad A_n = \C[[a_1,\dots,a_n]] 
\]
with the topology given by the $A_n$-linear subspaces $J_N = \big(\prod (t-a_i)\big)^NA_n[[t]]$. This is the most natural proposal for a ring of functions on a formal $n$-pointed disk. Working in an $A_n$-linear setting allows us in addition to "move the points freely", for instance declaring two of them to be equal by factoring by the ideal $(a_i - a_j)$, or on the contrary declaring them to be distinct by inverting $a_i - a_j$. 

The $n$-singularities analogue of the residue map 
\[  
    \int_n : K_ndt \to A_n
\]
is defined on the dense subring $A_n[t]\big[\prod(t-a_i)^{-1}\big]$ as the sum of the residues at the points $a_i$, and then extended by continuity. With these two key ingredients we can build an analogue of $\hat{\g}_{\kappa}$ in the setting of $n$ singularities:
\[
    \hat{\g}_{\kappa,n} = \g\otimes_{\C} K_n \oplus A_n\mathbf{1},
\]
where again $\mathbf{1}$ is a central element and the bracket is defined as before:
\[
    [Xf,Yg] = [X,Y]fg + \kappa(X,Y)\mathbf{1}\int_n(gdf) \quad \text{with } X,Y \in \g\; \text{ and } f,g \in K_n.
\]

The topology of $K_n$ induces a structure of topological Lie algebra on $\hat{\g}_{\kappa,n}$. We may construct its completed enveloping algebra as follows. Let $U'_{\kappa}(\hat{\g}_n) = U_{A_n}(\hat{\g}_{\kappa,n}) / (\mathbf{1} - 1)$, and define $\tilde{U}_{\kappa}(\hat{\g}_n)$ to be its completion:
\[
    \tilde{U}_{\kappa}(\hat{\g}_n) = \varprojlim_N \frac{U'_{\kappa}(\hat{\g}_n)}{U'_{\kappa}(\hat{\g}_n)J_N}.
\]
Our goal is to describe the center of this complete associative algebra at the critical level, which we denote 
\[
Z_n(\hat{\mathfrak{g}}) = Z(\tilde{U}_{\kappa_c}(\hat{\g}_n)).
\]

\medskip

The perfect candidate to describe geometrically the center, in analogy to the Feigin-Frenkel theorem, is the algebra of functions on the space of Opers, relative to a Lie algebra $^L\mathfrak{g}$, over the $n$-pointed formal disk $D_n = \spec K_n$. \\
The space of Opers, relative to any simple Lie algebra $\g$ and over $D_n$, denoted by $Op_{\g}(D_n)$ is a classifying space for certain connections on trivial $G$-bundles on $D_n$, where $G$ is the adjoint algebraic group of $\mathfrak{g}$. The general definition may be found in \cite{beilinson2005opers}. \\ 
It can be shown to be an ind-affine scheme over $A_n$. We denote by $A_n[Op_{\g}(D_n)]$ its algebra of functions, which is naturally a complete topological ring.

Our main goal is to prove the following theorem.

\begin{thm}
    For every simple finite Lie algebra $\g$ over $\C$ and every positive integer $n$ there exists a canonical isomorphism
    \[
        Z_n(\hat{\mathfrak{g}}) = A_n[Op_{^L\mathfrak{g}}(D_n)],
    \]
    where $^L\mathfrak{g}$ is the Langlands dual Lie algebra.
    This isomorphism is $\Der K_n$-equivariant and compatible with the structure of factorization algebras on both spaces.
\end{thm}
This is the content of Theorem \ref{theoremmain} together with Theorem \ref{thmassociativity}. We refer to \cite{beilinson2004chiral} for the notion of factorization algebras.
\subsection{Sketch of the proof}

After introducing in chapter 2 the main algebraic tools that we will use, we proceed following the original proof of the Feigin-Frenkel theorem with some slight, but in our opinion crucial, changes. Indeed, as in the proof of the aforementioned theorem, the starting point is to construct central elements in the completed enveloping algebra coming from central elements in the affine vertex algebra $V^{\kappa_c}(\g)$. In the classical setting of $1$ singularity this is achieved regarding the formal series
\begin{equation}
    X(z) = \sum_{n \in \Z} X_nz^{-n-1} \quad \text{for } X \in \g \text{ ,where } X_n = X\otimes t^n \in \hat{\mathfrak{g}}_{\kappa},
\end{equation}
as series with coefficients in the completed enveloping algebra $\tilde{U}_{\kappa}(\hat{\mathfrak{g}})$. One then checks that the formulas of concatenated normally ordered products and derivatives converge in $\tilde{U}_{\kappa}(\hat{\mathfrak{g}})$. Hence, to each element $A$ in the vertex algebra $V^\kappa(\mathfrak{g})$ is associated a formal series
\[
    Y[A,z] \in \tilde{U}_{\kappa}(\hat{\mathfrak{g}})[[z^{\pm 1}]].
\]
These are shown to satisfy the usual commutation relation:
\[
    \big[Y[A,z],Y[B,w]\big] = \sum_{k \geq 0} \frac{1}{n!}Y[A_{(k)}B,w]\partial_w^k\delta(z-w).
\]
As a consequence of this identity, if $A$ is central in $V^{\kappa}(\g)$ then the Fourier coefficients of $Y[A,z]$ are all central in the completed enveloping algebra $\tilde{U}_{\kappa}(\hat{\mathfrak{g}})[[z^{\pm 1}]]$. The details of this argument may be found in \cite[Ch.4]{frenkel2004vertex}.

\smallskip

If we want to apply a similar reasoning to the case of $n$ singularities, the first and most important step is to find the correct analogue of the series $X(z)$. We exploit the following remark:

\begin{rmk}
    Let $V$ be a $\C$ vector space. There is a canonical isomorphism of $\C[z,\partial_z]$-modules
    \begin{align*}
        \End(V)[[z^{\pm 1}]] &= \Hom_{\C}\big(\C[z^{\pm 1}],\End(V)\big) \\
        \varphi(z) &\mapsto \bigg( p(z) \mapsto \int p(z)\varphi(z)dz \bigg)
    \end{align*}
    Under this isomorphism the subspace of fields is mapped isomorphically to the subspace of continuous linear morphisms $\Hom_{\C}^{\text{cont}}\big(\C[z^{\pm 1}],\End(V)\big)$, where $\C[z^{\pm 1}]$ has the topology given by the subspaces $z^N\C[z]$, $V$ is given the discrete topology, and $\End(V)$ is equipped with the topology of pointwise convergence.
    Since $\End(V)$ is complete we have $\Hom^{\text{cont}}_{\C}\big( \C[z^{\pm 1}] , \End(V) \big) = \Hom^{\text{cont}}_{\C}\big( \C((z)) , \End(V) \big)$. Therefore:
    \[
        \big\{ \text{Fields on } V \big\} = \Hom^{\text{cont}}_{\C}\big( \C((z)) , \End(V) \big).
    \]
\end{rmk} 
\noindent We call the space appearing in the right hand side the \textbf{space of distributions} on $\C((z))$ with values in $\End(V)$.

In section 3 we rebuild the basic theory of vertex algebras and of fields using the language of distributions. We shall work in the generality of $U$-valued distributions on $K$, where $K$ is an (localized) $(\varphi)$-adic ring over an arbitrary noetherian ring $A$ with a global coordinate (c.f. Definitions \ref{deffadicrings} and \ref{coordinate})
and $U$ is a complete topological algebra over $A$ with topology generated by left ideals. We remark here that the main difference between $K$ as above and $K_n$ is that the latter is equipped with a residue, in particular a crucial property is Lemma \ref{intformula}. Essentially all the results in this paper which are stated only for $K_n$ may be extended to any localized $(\varphi)$-adic ring equipped with a global coordinate $z$, if we endow $K$ with a residue which satisfies Lemma \ref{intformula} and $\int \circ \partial_z = 0$.

\smallskip

The notion of fields in vertex algebras is recovered in the special case when $K = \C((z))$. In particular we may consider fields with values in any complete topological algebra with topology generated by left ideals, as $\tilde{U}_{\kappa}(\hat{\mathfrak{g}})$. We hope that this point of view may also shed some new light in the usual theory of vertex algebras, simplifying, for example, the arguments used in Proposition 4.2.2 of \cite{frenkel2004vertex} and related results.

\smallskip

In place of the series $X(z)$ in $(1)$ we consider the $\tilde{U}_{\kappa}(\hat{\mathfrak{g}}_n)$-valued distributions on $K_n$ given by 
\[
    \hat{X} \in \Hom^{\text{cont}}_{A_n}\big( K_n, \tilde{U}_{\kappa}(\hat{\mathfrak{g}}_n)\big), \qquad \hat{X}(f) = X\otimes f, \text{ for } X\in\mathfrak{g}, f \in K_n.
\]
Setting $a_i = 0$ we recover the distribution associated to the generating series $X(z)$.

We can perform derivatives and $n$-products of fields-distributions, mimicking the usual $n$-products of fields. We thus prove the analogue of Dong's Lemma in \ref{dong}. Finally we prove in Theorem \ref{vertexalgebras} that a subspace of mutually local continuous distributions, containing a unity, which is closed with respect to derivatives and $n$-products is naturally an honest vertex algebra. This should be thought as an analogue of Proposition 3.2 of \cite{kac1998vertex} and its consequences.

In conclusion, the canonical distributions $\hat{X}$, for $X \in \mathfrak{g}$, introduced above, generate a vertex algebra which we call $V_{K_n}^{\kappa}(\mathfrak{g})$.
In fact, as stated in Proposition \ref{propnewvertexisoldvertex}, this vertex algebra is canonically isomorphic to the universal affine vertex algebra of $\mathfrak{g}$ at level $\kappa$:
\begin{equation}\label{eqvertexalgebras}
    V^{\kappa}(\g) = V_{K_n}^{\kappa}(\mathfrak{g}).
\end{equation}
This is for us a crucial point: it allows us to associate to every central element $z \in \zeta\big(V^{\kappa_c}(\mathfrak{g})\big)$ and any function $f \in K_n$ a central element in the completed enveloping algebra $z(f) \in \tilde{U}_{\kappa_c}(\hat{\mathfrak{g}}_n)$. 

In order to make better use of the above results, we introduce a new object: the complete associative algebra $\tilde{U}_{K_n}( V^{\kappa}(\mathfrak{g}))$, which we now describe. In section 4 we define a functor from the category of quasi-conformal vertex algebras finitely generated (as vertex algebras) by homogeneous vectors to the category of complete associative algebras over $A_n$
\[
    V \rightsquigarrow \tilde{U}_{K_n}(V).
\]
The construction of $\tilde{U}_{K_n}(V)$ is performed by first constructing the $K_n$-analogue of the Lie algebra of Fourier coefficients $U(V)$ (with the notation of \cite{frenkel2004vertex}), which we denote here $Lie_{K_n}(V)$, then taking a natural completion of its enveloping algebra and finally factoring out some natural relations which come from the axioms of vertex algebra. 

The isomorphism in \eqref{eqvertexalgebras} yields a Lie algebra morphism $Lie_{K_n}(V^{\kappa}(\g)) \to \tilde{U}_{\kappa}(\hat{\mathfrak{g}}_n)$ which naturally extends to a map
\[
    \tilde{U}_{K_n}\big( V^{\kappa}(\mathfrak{g})\big) \to \tilde{U}_{\kappa}(\hat{\mathfrak{g}}_n),
\]
which we show to be an isomorphism. Finally, simple remarks show that the composition
\[
    \tilde{U}_{K_n}\big( \zeta(V^{\kappa_c}(\mathfrak{g}))\big) \to \tilde{U}_{K_n}\big( V^{\kappa_c}(\mathfrak{g})\big) \to \tilde{U}_{\kappa_c}(\hat{\mathfrak{g}}_n)
\]
factors through the center $Z_n(\hat{\g})$:
\begin{equation}\label{eqmorphismcenter}
    \tilde{U}_{K_n}\big( \zeta(V^{\kappa_c}(\mathfrak{g}))\big) \to Z_n(\hat{\mathfrak{g}}).
\end{equation}

    We complete the proof of Theorem $1.2$ in two further steps: first, we prove that the map in \eqref{eqmorphismcenter} is an isomorphism. This is done in section 5, in complete analogy with the proof of the Feigin-Frenkel Theorem presented in \cite{frenkel2007langlands}. It is crucial here to develop an analogue of Jet schemes in our context of multiple singularities. We define them and prove some basic properties.
    
Second, we exploit the intermediate object $\tilde{U}_{K_n}\big( \zeta(V^{\kappa_c}(\mathfrak{g}))\big)$. We know from the Feigin-Frenkel Theorem that $\zeta(V^{\kappa_c}(\mathfrak{g}))$ is canonically isomorphic to $\C[Op_{^L\mathfrak{g}}(D)]$, the algebra of functions on the space of Opers for $^L\mathfrak{g}$ on $D  =\spec \C[[t]]$. To complete the proof it is enough to prove that there is a canonical isomorphism 
\[  
    \tilde{U}_{K_n}(\C[Op_{\mathfrak{g}}(D)]) = A_n[Op_{\g}(D_n)]
\]
for an arbitrary simple Lie algebra $\g$. This is done in section 6, using explicit descriptions of both spaces.

\smallskip

Finally in section 7 we deal with the so called factorization properties. These are collections of algebras $\mathcal{A}_I$ indexed by finite sets with the following additional structure: to every surjection $I \twoheadrightarrow J$ and every decomposition $I = I_1 \cup I_2$ there are prescribed ways to relate $\mathcal{A}_I$ with $\mathcal{A}_J$ and $\mathcal{A}_{I_1},\mathcal{A}_{I_2}$  respectively. For the general definition we refer to \cite{beilinson2004chiral}. To work with factorization algebras we make a little change of notation, instead of considering an integer $n$ we consider a finite set $I$ and attach to it the center of the completed enveloping algebra $Z_I(\hat{\g})$.

It is pretty clear that the algebras $A_I[Op_{^L\mathfrak{g}}(D_I)]$ and $\tilde{U}_{\kappa}(\hat{\g}_I)$ satisfy the factorization properties.
The latter induce natural factorization maps
\begin{align*}
    Z_I(\hat{\g}) &\to Z_J(\hat{\g}), \\
    Z_I(\hat{\g}) &\to Z_{I_1}(\hat{\g}) \hat{\boxtimes} Z_{I_2}(\hat{\g}).
\end{align*}
We don't know a priori that these are isomorphisms though. The goal of the last section is to show that the isomorphisms we constructed 
\[
    A_I[Op_{^L\mathfrak{g}}(D_I)] \to Z_I(\g)
\] are compatible with the factorization maps, which easily implies that $Z_I(\hat{\g})$ is a factorization algebra as well.

The main difficulty is to show factorization properties for distributions first. We see that the distributions on $K_n$ we constructed starting from $V^{\kappa_c}(\g)$ satisfy natural factorization properties, and from these, in addition with the explicit description of our isomorphism, we prove Theorem \ref{thmassociativity}, which concludes the proof of our main Theorem $1.2$.

\subsection{On the necessity of our assumptions}

We would like to comment here some of the choices we made in writing this article, as well as to give some hints for possible further developments.

\begin{itemize}
    \item Our definition of a global coordinate (c.f. Definition \ref{coordinate}) is way stronger than needed, but the rings we are interested in this article (i.e. $R_n$, $K_n$) admit a global coordinate in this sense, so to simplify exposition we kept this strong hypothesis. Our definition of a global coordinate allows little generality, indeed a ring which admits a global coordinate is not too far from being the ring of functions on an open set of $\A^1_A$, so we are really dealing with a "flat" case. In order to work over more general curves (or families of curves) we believe that a more suitable set of assumptions would be the following:
    \begin{itemize}
        \item There exists $z \in K$ such that the module of differentials $\Omega^1_{K/A} = Kdz$.
        \item The diagonal ideal $I \subset K\otimes K$ is principal.
    \end{itemize}
    With these hypothesis one of our main results, Theorem \ref{vertexalgebras}, should be formulated in a different way. We believe that in this generality a closed space of distributions should have a natural structure of chiral algebra, instead of a vertex algebra. This will be the object of future works.
    \item The spaces of distributions, which are a central object in this paper (c.f. Definition \ref{defdistributions}), are not strictly essential until Section 7. Indeed, instead of reinterpreting the spaces of formal series $V[[z^{\pm 1}]]$ as spaces of distributions and then consider distributions on $K_n$, one may directly define suitable spaces of formal series in the setting of $n$ singularities. This approach would have been definitely more concrete but we believe it would have been also unnecessarily unnatural and cumbersome in notation.  We stress again that using the language of formal power series would have made very hard working with factorization properties. Finally, we think that the language of distributions sheds some light on the continuity condition of fields in the vertex algebras theory.
    
    As a final remark we would like the say that the language of distributions may also apply to rings whose topologies is less algebraic, such as the ring of germs of meromorphic functions around a point of the complex plane.
\end{itemize}

\noindent\textit{Acknowledgements.} I would like to thank my advisors Alberto De Sole and Andrea Maffei, for the huge support in writing this article and the numerous discussions which made this work possible.

\section{Algebraic preliminaries}

We give here a brief overview of the basic algebraic setting we will need throughout the rest of the paper. We will mainly deal with complete topological algebraic objects and define here some related constructions such as complete tensor products and continuous modules of differentials. Starting from section 2.4 we are going to define the class of rings we are mainly concerned with: we call them \textbf{rings with a global coordinate}. Finally at the end of the section we are going to define the rings of functions on formal $n$-pointed disks and to construct natural generalizations of the affine Kac-Moody algebra $\hat{\g}_k$ related to them.

\subsection{Complete topological rings}

We will work over a fixed commutative algebra $A$ over $\C$ which we assume to be noetherian. All modules, algebras and rings (commutative algebras) are to be considered over $A$, in the case of non commutative $A$-algebras we will require $A$ to be central. All associative algebras and rings we will consider are unital. 

When considering topologies on these algebraic objects we are always going to assume that  the topologies are $A$-linear and that the base ring $A$ inherits the discrete topology, in the case of topological $A$-algebras. In other words we assume that there is a fundamental system of neighborhoods of $0$ composed by open $A$-submodules $I_\alpha$ while a topological basis is given by translating the $I_\alpha$. When we write "let $I_\alpha$ be a fundamental system of neighborhoods for $R$" we implicitly assume these to be open $A$-submodules.

We say that a topology is \emph{generated by two sided} (resp. \emph{left}, resp. \emph{right}) \emph{ideals} if there is a fundamental system of open neighborhoods of $0$ consisting of two sided (resp. left, resp. right) ideals.
We will also consider topologies which are not induced by ideals. Our prototypical example for the latter will be the ring of Laurent polynomials 

\[
    \C[t^{\pm1}] \;\text{ ,with the topology given by }\; t^n\C[t]\; n\geq 0,
\]
and its completion, the ring of formal Laurent series $\C((t))$.



Given a topological algebra $(R,\tau)$ we call $\hat{R}^\tau$ (or just $\hat{R}$, when the topology is clear from the context) its completion, which has a natural structure of complete topological algebra. Note that, if $R$ is an algebra with two different topologies $\tau_1$, $\tau_2$ and if $\tau_2$ is finer than $\tau_1$, then there is a canonical morphism of algebras $\hat{R}^{\tau_2} \to \hat{R}^{\tau_1}$.

The following construction will be useful later.

\begin{defi}
    Let $V$ be a topological $A$-module, $I_{\alpha}$ a fundamental system of neighborhoods of $0$ for $V$. We define the \textbf{completed symmetric algebra} relative to $V$, denoted by $\widetilde{\Sym}_A(V)$, the completion of $\Sym_A(V)$ by the ideals generated by $I_\alpha$.
    \[
        \widetilde{\Sym}_A(V) = \varprojlim \frac{\Sym_A(V)}{(I_\alpha)}.
    \]
    It satisfies the following universal property. For any continuous $A$-linear morphism $V \to S$, where $S$ is a complete topological ring with topology generated by ideals there exists a unique continuous morphism of commutative rings
    \[
        \widetilde{\Sym}_A(V) \to S
    \]
    extending the morphism $V \to S$.
\end{defi}

\subsection{Tensor products}

There is no canonical way to assign a topology on the tensor product of two topological algebras. \\ 
We are going to consider here various topologies on the tensor product $S_1\otimes S_2$ of two complete topological algebras over $A$.  Suppose that we are given $U$, a complete topological $A$-algebra, with two continuous $A$-linear morphisms $\varphi_i: S_i \to U$ and consider the linear map
\[
    \varphi_1\varphi_2 : S_1 \otimes S_2 \to U, \qquad f_1\otimes f_2 \mapsto \varphi_1(f_1)\varphi_2(f_2).
\]
We will consider different topologies on $S_1\otimes S_2$ depending whether the topology on $U$ is generated by left, right or two sided ideals, in order to make the map $\varphi_1\varphi_2$ continuous. In what follows $I_\alpha$ (resp. $J_\beta$) will denote a fundamental system of neighborhoods for $S_1$ (resp. $S_2$).
\begin{itemize}
    
    \item \textbf{"The middle topology"} Define a system of open neighborhoods of $0$ for $S_1\otimes S_2$ to be given by the $A$-submodules
    \[
        I_\alpha\otimes J_\beta.
    \]
    This topology makes the product in $S_1\otimes S_2$ continuous, therefore its completion along this topology, which we denote by $\hat{S_1\otimes S_2}$, is naturally a complete topological $A$-algebra. The map $\varphi_1\varphi_2$ is always continuous with respect to the middle topology and therefore induces a continuous map 
    \[
        \varphi_1\varphi_2 : \hat{S_1\otimes S_2} \to U.
    \]
    
    \item \textbf{"The left (resp. right) or $1$ (resp. $2$) topology"} On the tensor product $S_1\otimes S_2$ consider the topology given by declaring a fundamental system of open neighborhoods of $0$ to be given by the $A$-submodules
    \[
         I_\alpha \otimes S_2 \quad \big( \text{resp. }  S_1\otimes J_\beta\big).
    \]
    These topologies make the product continuous, hence the completions have natural structures of complete topological $A$-algebras. They are both coarser then the middle topology.
    If $U$ has a topology generated by left (resp. right) ideals the map $\varphi_1\varphi_2$ is continuous with respect to the right (resp. left) topology.
    
    We are also going to consider analogues of these topologies on multiple tensor products. We define the $j$ topology on $S_1\otimes \dots \otimes S_n$ to be the topology on which a fundamental system of open neighborhoods is given by $S_1\otimes \dots\otimes I_\alpha \otimes\dots\otimes S_n$ where $I_\alpha$ is a fundamental system of neighborhoods of $0$ for $S_j$. Again, in this topology the product is continuous and we may consider the completion
    \[
        \widehat{S_1\otimes \dots \otimes S_n}^j
    \]
    along this topology, which is naturally a complete topological $A$-algebra.
    
    Note that this topology makes sense even when $S_1$ has no preassigned topology. In particular, starting from any ring $B$ over $A$ we may construct the completed tensor product
    \[
        B(S) \stackrel{\text{def}}{=} \widehat{B\otimes S}^2.
    \]
    We call $B(S)$ the \textbf{base change} of $S$ along the map $A \to B$.
    \item \textbf{"The ideal sum topology"} Suppose that the topology on $S_1$ (resp. $S_2$) is the $I$-adic (resp. $J$-adic) topology, for some two-sided ideals $I \subset S_1,J \subset S_2$.
     We call the completion of $S_1\otimes S_2$ by the $(I\otimes S_2 + S_1\otimes J)$-adic topology
    \[
        \widehat{S_1\otimes S_2}.
    \]
    This topology is coarser than the ones listed before and if the topology on $U$ is given by two-sided ideals then the map $\varphi_1\varphi_2$ is continuous with respect to the ideal sum topology.
    In addition notice that the two canonical morphisms
    \[
        \iota_i : S_i \to \widehat{S_1\otimes S_2}
    \]
    are continuous and that $\widehat{S_1\otimes S_2}$ is the coproduct in the category of complete topological $A$-algebras with topology generated by two-sided ideals.
    \item \textbf{"The naive sum topology"} Define a system of open neighborhoods of $0$ for $S_1\otimes S_2$ to be given by the $A$ submodules
    \[
        S_1\otimes I_\alpha + J_\beta\otimes S_2.
    \]
    A map which is continuous for the left and right topology is also continuous for the naive sum topology. The product on $S_1\otimes S_2$ does not need to be continuous for this topology, so the completion $\widehat{S_1\otimes S_2}^{1+2}$ does not have a natural structure of an algebra.
\end{itemize}

\begin{lemma}\label{continuousextension}
    The sum topology (resp. left, right) makes the linear map
    \[
        \varphi_1\varphi_2 : S_1\otimes S_2 \to U 
    \]
    continuous if $U$ has a topology generated by two sided (resp. right, left) ideals, while without any additional hypothesis on the topology on $U$ the morphism $\varphi_1\varphi_2$ is always continuous with respect to the middle topology.
    In addition there exist natural maps which make the following diagram commutative:
    \[\begin{tikzcd}
	&&& {\widehat{S_1\otimes S_2}^1} \\
	{} & {S_1\otimes S_2} & {} & {\hat{S_1\otimes S_2}} && {\widehat{S_1\otimes S_2}^{1+2}} \\
	&&& {\widehat{S_1\otimes S_2}^2}
	\arrow[from=1-4, to=2-6]
	\arrow[from=3-4, to=2-6]
	\arrow[from=2-4, to=3-4]
	\arrow[from=2-4, to=1-4]
	\arrow[from=2-2, to=1-4]
	\arrow[from=2-2, to=3-4]
	\arrow[from=2-2, to=2-4]
	\arrow[from=2-4, to=2-6]
\end{tikzcd}\]
\end{lemma}

\begin{proof}
    Clear.
\end{proof}

\begin{rmk}\label{notaring}
    Consider the case in which we take an element $x\in S_1\otimes S_2$, invertible in both $\widehat{S_1\otimes S_2}^i$, but not invertible in $\hat{S_1\otimes S_2}$. Then there is no reason to think that the following two compositions coincide:
    \[\begin{tikzcd}
    && {\widehat{S_1\otimes S_2}^1} && {\widehat{S_1\otimes S_2}^1} \\
    {S_1\otimes S_2} &&&&&& {\widehat{S_1\otimes S_2}^{1+2}} \\
    && {\widehat{S_1\otimes S_2}^2} && {\widehat{S_1\otimes S_2}^2}
    \arrow[from=2-1, to=3-3]
    \arrow[from=2-1, to=1-3]
    \arrow["{\text{mult}(x^{-1})}", from=1-3, to=1-5]
    \arrow["{\text{mult}(x^{-1})}", from=3-3, to=3-5]
    \arrow[from=3-5, to=2-7]
    \arrow[from=1-5, to=2-7]
\end{tikzcd}\]
\end{rmk}

\subsection{\texorpdfstring{$(\varphi)$}{(phi)}-adic rings and rings with a global coordinate}

We will consider rings which arise as function rings of principal formal neighborhoods.

\begin{defi}\label{deffadicrings}
    Let $R$ be a ring over $A$ and let $\varphi \in R$ be a non invertible element. $R$ is called \textbf{$(\varphi)$-adic} if its complete topological with respect to topology is induced by the principal proper ideal $I = (\varphi)$. We also require that $\varphi$ is not a 0 divisor and that $R/I$ is finite and free as an $A$-module.
    
    A complete topological ring $K$ over a base ring $A$ is called a \textbf{localized $(\varphi)$-adic} ring if it is the localization $K = R_\varphi$ of an $(\varphi)$-adic ring $R$.
\end{defi}

\begin{notation}
    From now on when we write $R$ we will denote a $(\varphi)$-adic ring. While $K = R_\varphi$ will denote its attached localized $(\varphi)$-adic ring.
\end{notation}

The main examples we are interested in are the following:

\begin{itemize}
    \item Let $A = \C$ and consider $R^{\text{pol}} = \C[t]$. We can take the completion $R = \C[[t]]$ along the ideal generated by $t \in \C[t]$ and then consider also $K = \C((t))$, the localization of $R$ by $t$. 
    \item Let $A_n = \C[[a_1,\dots,a_n]]$ and $R_n^{\text{pol}} = A_n[t]$ and consider the polynomial $\varphi_n = \prod (t - a_i)$. 
    
    Then consider $R$, the completion of $R^{\text{pol}}$ by the topology generated by the ideal $I = (\varphi_n)$. Note that since $\varphi_n$ is monic of degree $n$ we have $R^{\text{pol}}/(\varphi) =\oplus_{k=0}^{n-1} At^k$.  It can be shown that $R_n \simeq A_n[[t]]$ which is a complete topological ring with the topology given by the ideals $\big( \prod (t - a_i)\big)^NA_n[[t]]$. Finally we consider $K_n$, the localization of $R_n$ in which we invert the element $\varphi_n$.
\end{itemize}

\begin{rmk}\label{generality}
The rings we are interested in are essentially the ones listed above and their base change $B(R_n), B(K_n)$ for a fixed $A_n$-algebra $B$. Until section \ref{knstarts} everything we will prove holds in the generality of (localized) $(\varphi)$-adic rings equipped with a global coordinate (c.f. Definition \ref{coordinate}).
\end{rmk}

We list here some simple results on $(\varphi)$-adic rings.

\begin{prop}\label{propgeneralitiesfadic}
    Let $R,R_1,R_2$ be $(\varphi)$-adic rings over the same base ring $A$. The following hold:
\end{prop}

\begin{enumerate}
    \item Let $e_1, \dots, e_n \in R$ elements whose classes form an $A$-linear basis of $R/I$ (which we recall by assumption is free and finite as an $A$-module). Then the elements $e_i\varphi^k$ for $k\geq 0$ form a topological basis of $R$, meaning that every element in $x \in R$ can be written in a unique way as an infinite sum
    \[
        x = \sum_{k\geq 0}\sum_{i=1}^n a_{i,k}e_i\varphi^k \quad \text{ with } a_{i,k} \in A
    \]
    and every such sum defines an element in $R$.
    \item The previous statement may be generalized to the localization $K = R_\varphi$. The elements $e_i\varphi^k$ with $k\in\Z$ form a topological basis for $K$, meaning that every element $x \in K$ can be written in an unique way as an infinite sum
    \[
        x = \sum_{k \geq - N}\sum_{i=1}^n a_{i,k}e_i\varphi^k \quad \text{ with } a_{i,k} \in A
    \]
    and every such sum defines an element in $K$.
    \item Given $B$ a ring over $A$, consider the base change $B(R)$ and $\varphi = 1\otimes\varphi \in B(S)$. Then $B(S)$ is a $(\varphi)$-adic ring and its localization $B(S)_\varphi$ is canonically isomorphic to $B(K)$, where as always $K = R_\varphi$.
    \item There is a natural map 
    \[
        \hat{R_1 \otimes R_2} \to \widehat{R_1\otimes R_2}
    \]
    which turns out to be bijective, so it is an algebraic isomorphism but it may be not bicontinuous. This can be proved explicitly using the topological basis introduced above.
    \item Setting $K_1 = (R_1)_{\varphi_1}$ and $K_2 = (R_2)_{\varphi_2}$ it can be checked that the natural map
    \[
        (\hat{R_1\otimes R_2})_{\varphi_1\otimes \varphi_2} \to \hat{K_1 \otimes K_2}
    \]
    is an isomorphism.
\end{enumerate}

\begin{defi}\label{coordinate}
    Let $S$ be a ring over $A$ (resp. a complete topological ring) for which $\Omega^1_{S/A}$ (resp. $\Omega^{1,\text{cont}}_{S/A}$) is finitely generated and projective. We call a function $z \in S$ a \textbf{global coordinate} on $S$  if the multiplication morphism $S\otimes S \to S$ induces an isomorphism
    \[
        \frac{{S\otimes S}}{(z-w)} \simeq S \quad \bigg(\text{resp. } \frac{{\hat{S\otimes S}}}{(z-w)} \simeq S \bigg),
    \]
    where by a little abuse of notation (which we will use throughout the paper) $z = z\otimes 1$ and $w = 1\otimes z$. In the case of $S$ being a complete topological ring $(z-w)$ turns out to be a closed ideal, being the kernel of a continuous map.
\end{defi}

 When speaking of rings with a global coordinate, we will always assume the condition of finite generation and projectivity of the modules of differentials (resp. of continuous differentials). This is mainly because of the following proposition.

\begin{lemma}\label{lemmacoordinateforr}
    Let $R^{\text{pol}}$ be a finitely generated ring over $A$ and let $\varphi \in R^{\text{pol}}$ be a non invertible, non $0$-divisor element such that $R^{\text{pol}}/(\varphi)$ is finite and free as an $A$-module. Consider $R$, its completion along the ideal $(\varphi)$. Then $R$ is a $(\varphi)$-adic ring and if $z \in R^{\text{pol}}$ is a global coordinate then it is also a global coordinate for $R$.
\end{lemma}

\begin{proof}
    Recall that there are natural isomorphisms
    \[
        \hat{R\otimes R} \simeq \widehat{R\otimes R} \simeq \widehat{R^{\text{pol}}\otimes R^{\text{pol}}}
    \]
    where the completion of $R^{\text{pol}}\otimes R^{\text{pol}}$ is taken along the $J = (I\otimes R^{\text{pol}} + R^{\text{pol}}\otimes I)$-adic topology.
    
    By assumption there is an exact sequence
    \[\begin{tikzcd}
	{R^{\text{pol}}\otimes R^{\text{pol}}} && {R^{\text{pol}}\otimes R^{\text{pol}}} && {R^{\text{pol}}} & 0
	\arrow["{\cdot(z-w)}", from=1-1, to=1-3]
	\arrow["\mu", from=1-3, to=1-5]
	\arrow[from=1-5, to=1-6]
\end{tikzcd}\]
    We consider the objects appearing in the exact sequence above as $R^{\text{pol}}\otimes R^{\text{pol}}$ modules in the natural way. Then it is known (see for instance Lemma 10.97.1 \cite[\href{https://stacks.math.columbia.edu/tag/0BNH}{Tag 0BNH}]{stacks-project}) that taking the completion along the ideal $J$ is an exact operation. It follows that
    \[\begin{tikzcd}
	{\hat{R\otimes R}} && {\hat{R\otimes R}} && {R} & 0
	\arrow["{\cdot(z-w)}", from=1-1, to=1-3]
	\arrow["\mu", from=1-3, to=1-5]
	\arrow[from=1-5, to=1-6]
\end{tikzcd}\]
    is exact as well, proving that $z$ is a global coordinate for $R$ as well.
\end{proof}

\begin{lemma}\label{lemmacoordinatefork}
    Let $R$ be a $(\varphi)$-adic ring and let $K = R_\varphi$ be its localization. If $z \in R$ is a global coordinate for $R$ it is also a global coordinate for $K$.
\end{lemma}

\begin{proof}
    Recall that $\hat{K\otimes K} = (\hat{R\otimes R})_{\varphi\otimes \varphi}$. The results easily follows from the exactness of the localization by $\varphi\otimes \varphi$.
\end{proof}

\begin{corollary}\label{corozisacoordinate}
    The function $z \in R_n = A_n[[z]]$ is a global coordinate for $R_n$ and $K_n=A_n[[z]][\varphi_n^{-1}]$.
\end{corollary}

\begin{proof}
    Obviously $z \in R_n^{\text{pol}} = A_n[z]$ is a global coordinate for $R_n^{\text{pol}}$. Hence, by Lemma \ref{lemmacoordinateforr} $z \in R_n$ is a global coordinate and by Lemma \ref{lemmacoordinatefork} it is also a global coordinate for $K_n$.
\end{proof}

We conclude with a remark which will be crucial for us.

\begin{lemma}\label{invertibility}
    Let $K$ be the localization of a $(\varphi)$-adic ring and let $z$ be a global coordinate for $K$. Then the function $z-w \in \hat{K\otimes K}$ is invertible in the completions $\widehat{K\otimes K}^1$ and $\widehat{K\otimes K}^2$.
\end{lemma}

\begin{proof}
    By symmetry it's enough to prove the statement only for $\widehat{K\otimes K}^1$. Since $z$ is a global coordinate for $K$ there exists a function $h \in \hat{K\otimes K}$ such that
    \[
        \varphi\otimes 1 - 1\otimes \varphi = (z-w)h
    \]
    where as always  $z = z\otimes 1$ and $w = 1\otimes z$. Then the following series is convergent in $\widehat{K\otimes K}^1$
    \[
        -h\sum_{n\geq 0} \frac{\varphi^n\otimes 1}{1\otimes \varphi^{n+1}}
    \]
    and it is easily checked to be the inverse of $z-w$.
\end{proof}

\subsection{Module of continuous differentials}

\begin{defi}[Derivations and module of differentials]
    Given a complete topological $J$-adic ring $S$, we may consider the category of complete topological $J$-adic $S$-modules (i.e. $S$-modules such that the natural map $M \to \varprojlim M/J^nM$ is an isomorphism).
    \begin{itemize}
        \item An $A$-linear \emph{derivation} is a homomorphism of $A$-modules $d :S \to M$ such that for any $f,g\in S$
        \[
            d(fg) = fd(g) + gd(f).
        \]
        Note that such derivation $d$ is automatically continuous with respect to the $J$-adic topologies since $d(J^{n+1}) \subset J^nM$. 
        \item We call $\Omega^{1,\text{cont}}_{S/A}$, the module of continuous differentials, a complete $J$-adic $S$-module with a derivation
        \[
            d : S \to \Omega^{1,\text{cont}}_{S/A}
        \]
        which satisfy the following universal property: given a complete topological $J$-adic $S$-module $M$ with a derivation $d_M : R \to M$ there exists a unique continuous $S$-linear map $\Omega^{1,\text{cont}}_{S/A} \to M$ making the following diagram commute
        \[\begin{tikzcd}
    S && {\Omega^{1,\text{cont}}_{S/A}} \\
    \\
    && M
    \arrow["d", from=1-1, to=1-3]
    \arrow["{d_M}"', from=1-1, to=3-3]
    \arrow["{\exists !}", dashed, from=1-3, to=3-3]
\end{tikzcd}\]
    \end{itemize}
\end{defi}

\begin{prop} Let $S$ be a complete topological $J$-adic ring over $A$. The following hold:
    \begin{itemize}
        \item There is a natural isomorphism 
        \[
        \Omega^{1,\text{cont}}_{S/A} = \varprojlim \frac{\Omega^1_{S/A}}{J^n\Omega^1_{S/A}}.
        \]
        \item Suppose that $S$ is isomorphic to the completion of a ring $S^{\text{pol}}$ over $A$, along the topology generated by an ideal $J$ and that $\Omega^1_{S^{\text{pol}}/A}$ is finitely generated. Then $\Omega^{1,\text{cont}}_{S/A}$, the module of continuous differentials is isomorphic to
        \[
            \Omega^{1,\text{cont}}_{S/A} \simeq \varprojlim \frac{\Omega^{1}_{S^{\text{pol}}/A}}{J^n\Omega^{1}_{S^{\text{pol}}/A}}.
        \]
        \item Let $\widehat{S\otimes S}$ be the completion of $S\otimes S$ along the ideal sum topology. Then there is a natural isomorphism
        \[
            \Omega^{1,\text{cont}}_{S/A} = \frac{I}{I^2},\quad \text{where} \quad I = \ker (\widehat{S\otimes S} \to S ).
        \]
    \end{itemize}
\end{prop}

We now restrict our attention to rings equipped with a global coordinate and to (localized) $(\varphi)$-adic rings. 

\begin{prop}  The following hold:
\begin{enumerate}
    \item Let $S$ be a ring over $A$ and $z \in S$ be a global coordinate (i.e. $(z-w) = \ker(S\otimes S \to S)$). Then $\Omega^1_{S/A} = Sdz$;
    \item Let $S$ be a $(\varphi)$-adic ring and let $z \in S$ be a global coordinate (i.e. $(z-w) = \ker(\hat{S\otimes S} \to S)$) then $\Omega^{1,\text{cont}}_{S/A} = Sdz$.
\end{enumerate}

\end{prop}

\begin{proof}
    The statement for non-topological $S$ is trivial since $dz$ is a generator of $\Omega^1_{S/A}$ which is projective. For $S$ a $(\varphi)$-adic ring we use the isomorphism $\hat{S\otimes S} = \widehat{S\otimes S}$. The result then follows from the fact that
    \[
        \Omega^{1,\text{cont}}_{S/A} = \frac{I}{I^2} \quad \text{where}\; I = \ker ( \widehat{S\otimes S} \to S )
    \]
    and by Proposition \ref{propgeneralitiesfadic} (4), $\hat{S\otimes S} = \widehat{S \otimes S}$ so that $I = (z-w)$.
\end{proof}

\begin{defi}
    Let $R$ be a $(\varphi)$-adic ring and let $K = R_{\varphi}$ be its localization. We define the module of continuous differentials for $K$ as
    \[
        \Omega^{1,\text{cont}}_{K/A} \stackrel{\text{def}}{=} \big(\Omega^{1,\text{cont}}_{R/A}\big)_{\varphi}.
    \]
    This is a complete topological module with respect to the topology given by the subspaces
    $\text{Im}\big( \varphi^n\Omega^{1,\text{cont}}_{R/A} \to \Omega^{1,\text{cont}}_{K/A} \big)$. It is equipped with a derivation
    \[
        d : K \to \Omega^{1,\text{cont}}_{K/A}, \qquad \frac{x}{\varphi^n} \mapsto \frac{d(x)}{\varphi^n} - n\frac{xd(\varphi)}{\varphi^{n+1}}
    \]
    which satisfies the following universal property: given a complete topological $(\varphi)$-adic $R$-module $M$ and a derivation $d_M : K \to M_\varphi$ there exists a unique continuous $K$-linear morphism making the following diagram commute
    \[\begin{tikzcd}
    K && {\Omega^{1,\text{cont}}_{K/A}} \\
    \\
    && M_\varphi
    \arrow["d", from=1-1, to=1-3]
    \arrow["{d_M}"', from=1-1, to=3-3]
    \arrow["{\exists !}", dashed, from=1-3, to=3-3]
\end{tikzcd}\]
\end{defi}

\begin{prop}
    Let $K$ be the localization of a $(\varphi)$-adic ring $R$ and let $I_K = \ker (\hat{K\otimes K} \to K)$ be the kernel of the multiplication morphism. Then there is a natural isomorphism
    \[
        \Omega^{1,\text{cont}}_{K/A} = \frac{I_K}{I_K^2}.
    \]
\end{prop}

\begin{proof}
    Let $I_R = \ker ( \hat{R\otimes R} \to R)$ be the kernel of the multiplication morphism of $R$. It follows from the canonical isomorphism $\hat{R\otimes R} = \widehat{R\otimes R}$ that $\Omega^{1,\text{cont}}_{R/A} = I_R/I_R^2$. Now consider the exact sequence
    \[\begin{tikzcd}
	0 && {I_R} && {\hat{R\otimes R}} && R && 0
	\arrow[from=1-1, to=1-3]
	\arrow[from=1-3, to=1-5]
	\arrow[from=1-5, to=1-7]
	\arrow[from=1-7, to=1-9]
\end{tikzcd}\]
    and consider its localization by $\varphi\otimes\varphi$. It follows from exactness of localization and the natural isomorphism $(\hat{R\otimes R})_{\varphi\otimes\varphi} = \hat{K\otimes K}$ that $I_K = (I_R)_{\varphi\otimes\varphi}$. Again, by exactness of the localization it follows that
    \[
        \frac{I_K}{I_K^2} = \bigg(\frac{I_R}{I_R^2}\bigg)_{\varphi\otimes\varphi} = \big( \Omega^{1,\text{cont}}_{R/A} \big)_\varphi.
    \]
\end{proof}

\subsubsection{Taylor expansion}

The goal of this section is to prove the following proposition.

\begin{prop}[Taylor expansion]\label{taylor}
    Let $S$ be a ring (resp. a (localized) $(\varphi)$-adic ring) over $A$, assume that $z \in S$ satisfies $\Omega^{1}_{S/A} = Sdz$ (resp. $\Omega^{1\text{cont}}_{S/A} = Sdz$) and let $\partial = \partial_z$ be the unique derivation such that $\partial \cdot z = 1$. As usual let us write $z = z\otimes 1$ and $w = 1\otimes z$ in $S\otimes S$. Then for every $f \in S$ we have
    \[
        f\otimes 1 = 1 \otimes f + (z-w)1\otimes \partial_z f + \frac{(z-w)^2}{2}1\otimes \partial_z^2f + \dots + \frac{(z-w)^n}{n!}1\otimes \partial_z^nf \quad \text{ mod } (z-w)^{n+1}
    \]
    Or more pictorially 
    \[
        f(z) = f(w) + (z-w)\partial_w f(w) + \frac{(z-w)^2}{2}\partial_w^2f(w) + \dots + \frac{(z-w)^n}{n!}\partial_w^nf(w) \quad \text{ mod } (z-w)^{n+1}
    \]
\end{prop}

We will need the following lemma

\begin{lemma}
    Let $M$ be a $\C[[t]][\partial_t]$-module. Suppose that $M$ is complete as a $\C[[t]]$-module. Then the map
    \begin{equation}\label{eqtaylormorphism}
        M \to \frac{M}{tM}[[\tau]], \qquad m \mapsto \sum_{n\geq 0} \big[\frac{1}{n!}\partial_t^n m\big]\tau^n
    \end{equation}
    is an isomorphism of $\C[[t]][\partial_t]$-modules where $t$ acts on the right as multiplication by $\tau$ while $\partial_t$ acts as $\partial_\tau$.
\end{lemma}

\begin{proof}
    For convenience let us denote $\overline{M} = \frac{M}{tM}[[\tau]]$ and let us identify $\frac{\tau^n\overline{M}}{\tau^{n+1}\overline{M}}$ with $\frac{M}{tM}$ in the obvious way. The map $M \to \overline{M}$ defined above is clearly of $\C[[t]][\partial_t]$-modules.
    The key observation is that $\partial_t$ induces an isomorphism
    \[
        \frac{\partial_t}{n+1} : \frac{t^{n+1}M}{t^{n+2}M} \to \frac{t^{n}M}{t^{n+1}M}.
    \]
    Indeed its inverse is easily checked to be the multiplication by $t$.
    In particular the map
    \[
        \frac{\partial_t^n}{n!} : \frac{t^nM}{t^{n+1}M} \to \frac{M}{tM}
    \]
    is an isomorphism. Now the lemma follows from the following observations. First notice that we just need to prove that the map in \eqref{eqtaylormorphism} is an isomorphism modulo $t^n$. Then for every $n$ we can consider the filtrations
    \[\begin{tikzcd}
    {\frac{M}{t^{n+1}M}} & \supset & {\frac{tM}{t^{n+1}M}} & \supset & \dots & \supset & {\frac{t^nM}{t^{n+1}M}} & \supset & 0 \\
    \\
    {\frac{\overline{M}}{\tau^{n+1}\overline{M}}} & \supset & {\frac{\tau\overline{M}}{\tau^{n+1}\overline{M}}} & \supset & \dots & \supset & {\frac{\tau^n\overline{M}}{\tau^{n+1}\overline{M}}} & \supset & 0
    \arrow[from=1-1, to=3-1]
    \arrow[from=1-3, to=3-3]
    \arrow[from=1-7, to=3-7]
    \arrow[from=1-9, to=3-9]
\end{tikzcd}\]
    To show that the left-most vertical map is an isomorphism it suffice to show that all the maps between the successive quotients are isomorphisms, and these maps are easily checked to correspond to the isomorphisms 
    \[
        \frac{\partial_t^k}{k!} : \frac{t^kM}{t^{k+1}M} \to \frac{M}{tM} = \frac{\tau^k\overline{M}}{\tau^{k+1}\overline{M}}.
    \]
\end{proof}

\begin{corollary}\label{corotaylor}
    If $M$ is a $\C[[t]][\partial_t]$-module which is separated as a $\C[[t]]$-module, then any two elements $m,m' \in M$ for which
    \[
        \partial_t^n m = \partial_t^n m' \quad \text{ mod } tM 
    \]
    for all $n$ (resp. for all $n\leq N$) are equal (resp. are equal modulo $t^{N+1}M$). 
\end{corollary}

We apply this corollary in the case when $M$ is the limit 
\[\varprojlim \frac{\hat{S\otimes S}}{I^{n}}\] where $I$ is the kernel of the multiplication morphism $\hat{S\otimes S} \to S$. We make it into a $\C[[t]][\partial_t]$-module making $t$ act by multiplication of $z-w$ while $\partial_t$ acts like the derivation $\partial_{z-w} = (\partial_z\otimes 1-1\otimes\partial_z)/2$.

\begin{proof}[Proof of Proposition \ref{taylor}]
    First we need to check that the module we are considering is indeed complete as a $\C[[t]]$-module. We claim that the map
    \begin{equation}\label{eqformaldiagonal}
        \frac{S[t]}{(t^{N+1})} \to \frac{\hat{S\otimes S}}{I^{N+1}}, \qquad \sum_{k=0}^{N} s_kt^k \mapsto \sum_{k=0}^{N} (z-w)^ks_k\otimes 1 
    \end{equation}
    is an isomorphism for every $N$. This clearly implies that the limit is complete as a $\C[[t]]$-module. To check that the map in \eqref{eqformaldiagonal} is an isomorphism consider the natural filtrations on both spaces and notice that they a preserved by our morphism:
    \[\begin{tikzcd}
    {\frac{\hat{S\otimes S}}{I^{N+1}}} & \supset & {\frac{I}{I^{n+1}}} & \supset & \dots & \supset & {\frac{I^N}{I^{N+1}}} & \supset & 0 \\
    \\
    {\frac{S[t]}{t^{N+1}S[t]}} & \supset & {\frac{tS[t]}{t^{N+1}S[t]}} & \supset & \dots & \supset & {\frac{t^NS[t]}{t^{N+1}S[t]}} & \supset & 0
    \arrow[from=3-1, to=1-1]
    \arrow[from=3-3, to=1-3]
    \arrow[from=3-7, to=1-7]
    \arrow[from=3-9, to=1-9]
\end{tikzcd}\]
    To show that the left-most vertical map is an isomorphism it is sufficient to show that the induced maps on the successive quotients are all isomorphisms. Therefore our claim boils down to show that for every $k \geq 0$
    \[
        \frac{I^k}{I^{k+1}} = S[(z-w)^k],
    \]
    where $[(z-w)^k]$ is the class of $(z-w)^k\in I^k/I^{k+1}$. The fact that $I^k/I^{k+1}$ is generated by $(z-w)^k$ as an $S$-module is immediate since by hypothesis $(z-w)$ is a free generator of $I/I^2$. Finally one can easily check that the multiplication by $(z-w)$ viewed as a map $I^k/I^{k+1} \to I^{k+1}/I^{k+2}$ is bijective with inverse $\partial_{z-w}$ and this clearly implies that $I^k/I^{k+1}$ is freely generated as an $S$-module by $(z-w)^k$.
    
    To complete the proof we can now apply Corollary \ref{corotaylor}: we need to check that for any $k\leq N$
    \[
        \partial_{z-w}^k f \otimes 1 = \partial_{z-w}^k\big( \sum_{n=0}^N \frac{(z-w)^n}{n!}1\otimes \partial_z^n f \big) \quad \text{ mod } I.
    \]
    This is a straightforward computation in which we have to use the fact that for any $g \in S$
    \[
        g\otimes 1 = 1\otimes g \quad \text{ mod } I,
    \]
    which holds since $z$ is a coordinate in $S$.
\end{proof}

\subsubsection{Chain rule and coordinates}




\begin{prop}[Chain rule]\label{chainrule}
    Let $S$ be a ring (resp. complete topological ring) over $A$ and assume that $z \in S$ is such that
    \[
        \Omega^1_{S/A} = Sdz \qquad \big(\text{ resp. } \Omega^{1,\text{cont}}_{S/A} = Sdz\big)
    \]
    (this holds for instance when $z$ is a global coordinate for $S$). Then for any $A$-homomorphism (resp. continuous homomorphism) $\psi : S \to S$ and every $f \in S$, we have 
    \[
        d(\psi(f)) = \psi(\partial_zf)\partial_z(\psi(z))dz
    \]
\end{prop}

\begin{proof}
    We state the proof only in the non-complete case, the other case being completely analogous. Consider the $S$-module $(\Omega^1_{S/A})_\psi$ which is equal to $\Omega^1_{S/A}$ but where the action of $S$ is given by $f\cdot \omega = \psi(f)\omega$. It is easily seen that the map $d\circ\psi : S \to (\Omega^1_{S/A})_\psi $ is an $A$-derivation. We get therefore an $S$-linear morphism $\psi_* : \Omega^1_{S/A} \to (\Omega^1_{S/A})_\psi$ making the following diagram commute:
    \[\begin{tikzcd}
    && {\Omega^1_{S/A}} \\
    S \\
    && {(\Omega^1_{S/A})_\psi}
    \arrow["d", from=2-1, to=1-3]
    \arrow["{\psi_*}", from=1-3, to=3-3]
    \arrow["d\circ\psi"', from=2-1, to=3-3]
\end{tikzcd}\]
    For which by definition $\psi_*(dz) = d(\psi(z)) = \partial_z\psi(z)dz$. Then, by $S$-linearity of $\psi_*$
    \[
        d(\psi(f)) = \psi_*(df) = \psi_*(\partial_zfdz) = \psi(\partial_zf)\psi_*(dz) = \psi(\partial_z f)\partial_z\psi(z)dz.
    \]
\end{proof}

\begin{corollary}\label{derivativeinvertible}
    In the case when $\psi$ is an automorphism (resp. continuous automorphism), the map $\psi_*$ is an isomorphism with inverse $(\psi^{-1})_*$. In this case $\partial_z\psi(z) \in S^*$.
\end{corollary}

\subsection{Rings of functions on \texorpdfstring{$n$}{n}-pointed disks}

Let $A_n = \C[[a_1,\dots,a_n]]$, $\varphi_n = \prod (z-a_i) \in A_n[z]$ and

\[
    R^{\text{pol}}_{n} \stackrel{\text{def}}{=} A_{n}[z], \qquad
    K^{\text{pol}}_{n} \stackrel{\text{def}}{=}
    A_{n}[z]\bigg[\frac{1}{\varphi_{n}}\bigg], \qquad
    R_{n} \stackrel{\text{def}}{=} A_{n}[[z]], \qquad
    K_{n} \stackrel{\text{def}}{=} A_n[[z]]\bigg[\frac{1}{\varphi_{n}}\bigg].
\]
The complete topological ring $K_n$ should be regarded as the ring of functions on a formal $n$-pointed disk. We already noted that $z$ is a global coordinate for all the rings listed above (c.f. Corollary \ref{corozisacoordinate}), that $R_n$ is a $(\varphi_n)$-adic ring and $K_n$ is its localization by $\varphi_n$.

A crucial feature of the ring $\C((t)) = (K_1)_{a_1 = 0}$ is the residue map $\int : \Omega^{1,\text{cont}}_{\C((t))/\C} \to \C$, which picks the coefficient of $z^{-1}$. We now define a similar residue for our rings $K_n$. Recall that $\Omega^{1,\text{cont}}_{K_n/A}= K_ndz$ so we may regard continuous linear functionals on $\Omega^{1,\text{cont}}_{K_n/A}$ and on $K_n$ as the same object.

\begin{defi}\label{defiresidue}
    We define the $A_n$-linear map
    \[
        \int : \Omega^{1,\text{cont}}_{K_n/A_n} \to A_n
    \]
    as the sum of the residues around the points $a_i$.
    \[
        \int  = \sum_i \int_{a_i}
    \]
    Formally we define it on $K^{\text{pol}}_ndz$ considering its elements as meromorphic $1$-forms on the affine line and then extend it by continuity where, as always, $A_n$ is endowed with the discrete topology. One should check that this map has values in $A_n$, indeed the residue around a point $a_i$ naturally takes values in $Q_n = \text{Frac}(A_n)$ but their sum actually has values in $A_n$, this is not difficult to see.
\end{defi}

We list here some basic properties of the residue:

\begin{enumerate}
    \item After the choice of a global coordinate $z \in R_n$ we may write $\Omega^{1,\text{cont}}_{K_n/A} = K_ndz$ and consider the residue as an $A_n$ linear map $\int : K_n \to A_n$.
    \item The residue is $0$ on regular functions: $\int(R_n) = 0$.
    \item Let $\partial_z$ be the unique continuous  derivation such that $\partial_z z = 1$. Then for every $g \in K_n$
    \[
        \int (\partial_z g) = 0.
    \]
    Equivalently we could say that the composition $\int \circ d$ where $d : K_n \to \Omega^{1,\text{cont}}_{K_n/A_n}$ is the canonical derivation, vanishes.
    \item Given a ring $B$ over $A_n$ there is a natural extension of the residue to the base change $B(K_n)$
    \[
        \int : B(K_n) \to B.
    \]
    \item When we take $B = K_n$ there is a different construction we may perform. Consider the map
    \[
        \int_w : K_n \otimes K_n \to K_n \qquad g\otimes f \mapsto g\int f.
    \] 
    Note that the map above, in addition to being continuous for the right topology, is clearly continuous with respect to the left topology on $K_n\otimes K_n$ and the usual $(\varphi_n)$ topology on $K_n$. Therefore it induces a map
    \[
        \int_w : \widehat{K_n\otimes K_n}^1 \to K_n.
    \]
    \item Similarly, we have $\int_z : \widehat{K_n\otimes K_n}^2 \to K_n$. Both these maps admit natural extensions by base change
    \[
        \int_w : \widehat{B(K_n)\otimes B(K_n)}^1 \to B(K_n), \qquad \int_z : \widehat{B(K_n)\otimes B(K_n)}^2 \to B(K_n).
    \]
\end{enumerate}

\begin{lemma}\label{intformula}
    Let $\int_z$ be the linear map defined above. For $r\in R_n$ and $g \in K_n$ we have
    \[
        \int_z \frac{r\otimes g}{z-w}dz = rg.
    \]
\end{lemma}

\begin{proof}
     The factor $g$ comes out by linearity, so we reduce to the case $g=1$. In addition let $h \in \hat{R_n\otimes R_n}$ such that $(z-w)h = r\otimes 1 - 1\otimes r$, we have
     \[
        \int_z \frac{r\otimes 1}{z-w} = \int_z h + \int_z \frac{1\otimes r}{z-w}
     \]
     Then $\int_z h = 0$ since $h \in \hat{R_n\otimes R_n}$ and again the factor $r$ now comes out by linearity. We reduced ourselves to prove the case $r=1,g=1$. Finally to check that
     \[
        \int_z \frac{1}{z-w} = 1
     \]
     one may perform a direct calculation using the expansion of $(z-w)^{-1}$ described in the proof of Lemma \ref{invertibility}.
\end{proof}

\begin{corollary}\label{intbformula}
     Let $b \in B(R_n)$ and $g \in B(K_n)$, we have
    \[
        \int_z \frac{b\otimes g}{z-w}dz = bg.
    \]
\end{corollary}

\subsubsection{Topological basis}\label{topbasis}

We introduce here a topological basis for $K_n$. As noted before, $R_n/(\varphi_n)$ is free over $A_n$ and any choice of monic polynomials $e_1,\dots,e_n$, of degree $0,\dots,n-1$ respectively, induces an $A_n$-linear basis of $R_n/(\varphi_n)$. We pick a particularly symmetric choice.

\begin{rmk}\label{topbasisexplicit}
    Let $e_i$ the $(i-1)$-th elementary symmetric polynomial in the terms $(z-a_j)$, where $j$ goes from $1$ to $n$. So $e_1 = 1$, $e_2 = \sum_{j=1}^n (z-a_j)$ up to $e_{n+1} = \varphi_n$. Then the polynomials
    \[
        e_{i,k} \stackrel{\text{def}}{=} e_i(e_{n+1})^k, \quad i \in \{ 1\dots n\}, k\in \Z
    \]
    form a topological basis for $K_n$, meaning that each element of $g \in K_n$ can be uniquely expressed as an infinite sum
    \[
        g = \sum_{k\geq -N}\sum_{i=1}^n a_{i,k}e_{i,k} \quad \text{ with } a_{i,k} \in A_n
    \]
    and every sum of that form defines an element of $K_n$.
    
    This extends to the complete topological rings $B(K_n)$: the elements $e_{i,k} = 1\otimes e_{i,k}$ form a topological $B$-basis for $B(K_n)$, namely every element $b \in B(K_n)$ can be expressed in a unique way as an infinite sum
    \[
        b = \sum_{k\geq -N}\sum_{i=1}^n b_{i,k}e_{i,k} \quad \text{ with } b_{i,k} \in B
    \]
    and conversely any infinite sum of this form defines an element in $B(K_n)$.
\end{rmk}

Notice that the residue pairing, after the choice of a coordinate induces a pairing
\[
    K_n\times K_n \to A_n \qquad (f,g) \mapsto \int fgdz.
\]
We would like to describe a dual basis of the $e_{i,k}$ with respect to this pairing.

\begin{prop}
    There exists elements $\lambda_{ij} \in A_n$ such that, letting $\epsilon_{i,k} = \sum_{j = 1}^n \lambda_{ij}e_{j,k}$, the following formula holds
    \[
        \int \epsilon_{j,k}e_{i,l}dz = \delta_{k,-l-1}\delta_{i,j}.
    \]
\end{prop}

\begin{proof}
    It is easy to see, computing the residue as the residue at infinity, that 
    \[
        \int e_{j,k}e_{i,l}dz = 0 \quad \text{if } k \neq -l-1.
    \]
    In addition for any $k$ the elements $\int e_{j,-l-1}e_{i,l}dz$ are all equal to $\int e_ie_j(e_{n+1})^{-1}dz$ as $l$ varies. The claim boils down to prove that the symmetric matrix $S = (s_{ij})_{i,j =1,\dots,n}$ with
    \[
        s_{ij} = \int e_ie_j(e_{n+1})^{-1}dz
    \]
    is invertible in $\text{Mat}_{n\times n}(A_n)$ (the inverse being $(\lambda_{ij})$). In order to do so we calculate the entries $s_{ij}$ by calculating the residue at infinity. One easily sees that
    \[
        a_{i,j} = 0 \quad \text{if } j +i \leq n \qquad  a_{i,j} \in \C^* \quad \text{if } j + i = n + 1
    \]
    by direct computation. The matrix $S = (s_{i,j})$ is therefore invertible, as claimed.
\end{proof}

\subsection{Automorphisms}

We will study continuous automorphisms of an $(\varphi)$-adic ring $R$ which extend to continuous automorphisms of its localization $K$. As always all morphisms are going to be $A$-linear. By continuous automorphism we mean a bijective bicontinuous map.

\begin{prop}\label{propextaut}
    Any continuous automorphism $\psi : R \to R$ extends uniquely to a continuous automorphism $\psi : K \to K$. In particular there is a natural injective morphism of groups
    \begin{equation}\label{eqextensionautomorphism}
        \Aut^{\text{cont}}_A(R) \hookrightarrow \Aut^{\text{cont}}_A(K).
    \end{equation}
\end{prop}


\begin{proof}
    Let $\varphi$ be the generating function of $I$, the ideal which defines the topology on $R$, and $\psi$ a continuous automorphism of $R$. We claim that there exist $a \in R$ and a positive integer $n$ such that
    \[
        \psi(\varphi)a = \varphi^n.
    \]
    Indeed if $\psi^{-1}$ is the inverse of $\psi$, it follows by the continuity of $\psi^{-1}$ that there exist an integer $m$ and an element $b \in R$ such that 
    \[
        \psi^{-1}(\varphi^m) = b\varphi,
    \]
    now it suffice to consider $n=m$ and $a = \psi(b)$.
    
    Hence, the image of $\varphi$ under $\psi$ is invertible in $K$. Therefore $\psi : R \to R$ extends to a continuous automorphism of $K$.
\end{proof}

\begin{rmk}
    The immersion in \eqref{eqextensionautomorphism} does not need to be surjective. An easy counterexample is $A = \C[\epsilon]$, where $\epsilon^2 = 0$, $R = \C[\epsilon][[t]]$ and $K = \C[\epsilon]((t))$: one can check that $ t \mapsto t + \epsilon/t$ induces an automorphism of $K$ which does not restrict to an automorphism of $R$.
\end{rmk}

We may also look at continuous automorphisms of $B(R)$, $B(K)$, the base change of $R$ and $K$ respectively, and collect them in a unique object: the group functors $\underline{\Aut}(R)$ and $\underline{\Aut}(K)$.

\begin{defi}
    The group functors $\underline{\Aut}(R)$ and $\underline{\Aut}(K)$ are defined as follows. Given $B$ an $A$-commutative algebra,
    \[
        \underline{\Aut}(R)(B) = \Aut^{\text{cont}}_B(B(R)), \qquad \underline{\Aut}(K)(B) = \Aut^{\text{cont}}_B(B(K)).
    \]
    Since $B(R)$ is still an $(\varphi)$-adic ring and $B(K)$ is its localization, it follows from Proposition \ref{propextaut} that there is a natural injective morphism of group functors
    \[
        \underline{\Aut}(R) \hookrightarrow \underline{\Aut}(K).
    \]
\end{defi}

Organizing the automorphisms $\Aut^{\text{cont}}_A(R)$ and  $\Aut^{\text{cont}}_A(K)$ in group functors allows us to compute their Lie algebras. By definition, given a group functor $G$ from the category of commutative $A$-algebras, its Lie algebra (more precisely its set of $A$-points) is defined as 
\[
\Lie(G) \stackrel{\text{def}}{=} \ker\big( G(A[\epsilon]) \to G(A) \big).
\]
It has a natural structure of an $A$-module and a natural structure of $A$-linear Lie algebra.

\begin{lemma}
    Let $A = A_n$ and consider the $(\varphi_n)$-adic ring $R = R_n$ and its localization $K_n$. Then the following hold
    \[
        \Lie(\underline{\Aut}(R_n)) = \Der R_n, \qquad \Lie(\underline{\Aut}(R_n)) = \Der K_n.
    \]
    Where $\Der R_n$ and $\Der K_n$ are the Lie algebras of continuous derivations of $R_n$ and $K_n$ respectively, defined in the following section.
\end{lemma}

\subsection{The Lie algebras \texorpdfstring{$\Der R$}{Der R} and \texorpdfstring{$\Der K$}{Der K}}

Given a $(\varphi)$-adic ring $R$ with a fixed global coordinate $z \in R$ we introduce here the Lie algebras $\Der R$ and $\Der K$. Recall that the module of continuous differentials is free with $dz$ as a generator.

\begin{defi}
    We define the Lie algebras $\Der R$ (resp. $\Der K$) to be the Lie algebras of continuous derivations from $R$ to itself (resp. from $K$ to itself).
    \[
        \Der R = Der^{\text{cont}}_A(R), \qquad \Der K = Der^{\text{cont}}_A(K).
    \]
    By the isomorphisms
    \[
        \Omega^{1,\text{cont}}_{R/A} = Rdz, \qquad \Omega^{1,\text{cont}}_{K/A} = Kdz,
    \]
    we get isomorphisms
    \[
        \Der R = R\partial_z, \qquad \Der K = K\partial_z.
    \]
    The bracket is easily checked to be given by the usual formula
    \[
        [f\partial_z,g\partial_z] = \big( f\partial_z g \big)\partial_z - \big( g\partial_z f \big)\partial_z.
    \]
\end{defi}

\subsection{Lie algebras}

We slightly generalize the construction of the affine Kac-Moody algebra $\hat{\g}_k$ to other rings equipped with a residue. As usual in what follows $R$ will denote a $(\varphi)$-adic ring over $A$ and $K = R_\varphi$ its localization.

\begin{defi}
    Let $K$ be a localized $(\varphi)$-adic ring over $A$. Consider $A$ with the discrete topology. A \textbf{residue} on $K$ is a continuous $A$ linear map
    \[
        \int : \Omega^{1,\text{cont}}_{K/A} \to A
    \]
    such that 
    \[
        \int df = 0 \quad \text{ for all } f\in K \quad \text{ and } \quad \int_{|Rdz} = 0.
    \]
\end{defi}

\begin{rmk}
    The residue map we defined for $K_n$ in Definition \ref{defiresidue} is indeed a residue.
\end{rmk}

\begin{defi}\label{kacmoodygeneral}
    Given a finite Lie algebra $\g$ over $A$, an invariant $A$-bilinear symmetric invariant form $\kappa \in Hom_A(\g\otimes_A\g,A)$ and $\int$ a residue on $K$ we define the Lie algebra $\hat{\g}_{K,\kappa}$
    as
    \[
        \hat{\g}_{K,\kappa} = \g\otimes_A K \oplus A\mathbf{1},
    \]
    where $\mathbf{1}$ is a central element and the bracket $[Xf,Yg]$ for $X,Y\in \g$ and $f,g\in R$ is given by
    \[
        [Xf,Yg] = [X,Y]fg + \kappa(X,Y)\mathbf{1}\int gdf.
    \]
\end{defi}

This is easily checked to be a well defined $A$-bilinear bracket. We are often going to consider the previous definition in the case where $\g = \g'\otimes A$ for $\g'$ a finite dimensional Lie algebra over $\C$ and $\kappa$ is the $A$-linear extension of a bilinear symmetric invariant form defined on $\g'$.
The following proposition will be useful to construct a completion of the enveloping algebra.

\begin{prop}\label{envelopingcontinuity}
    Let $I_\alpha$ be a fundamental system of neighborhoods of $0$ for $K$. Then for any $I_\alpha$ and any $f \in K$ there exists an $I_\beta \subset I_\alpha$ such that
    \[
        [\g\otimes I_\beta + A\mathbf{1},\g\otimes f + A\mathbf{1}] \subset \g\otimes I_\alpha.
    \]
\end{prop}

\begin{proof}
    It is enough to find an $I_\beta$ such that
    \[
        I_\beta f \subset I_\alpha \quad \text{ and } \quad fd(I_\beta) \subset \ker\int.
    \]
    This is possible since both the multiplication map and $d : R \to \Omega^{1,\text{cont}}_{R/A}$ are continuous.
\end{proof}

We are going to consider the enveloping algebra $U(\hat{\g}_{K,\kappa})$ over $A$. We would like to treat the central element $\mathbf{1}$ as the unity, so we consider
\[
    U'_{\kappa}(\hat{g}_K) \stackrel{\text{def}}{=} \frac{U(\hat{\g}_{K,\kappa})}{(\mathbf{1}-1)}.
\]

\begin{corollary}\label{envelopingcompletion}
    If $I_\alpha$ is a fundamental system of neighborhoods for $K$ then the product on $U'_\kappa(\hat{\g}_{K})$ is continuous with respect to the topology generated by the left ideals 
    \[
        U'_\kappa(\hat{\g}_K)\big(\g\otimes I_\alpha\big).
    \]
    In particular the completion $\tilde{U}_\kappa(\hat{\mathfrak{g}}_{K})$ for this topology has a natural structure of an associative complete topological algebra over $A$.
\end{corollary}

\begin{rmk}
    The Lie algebra $\Der K$ naturally acts on $\hat{\g}_{K,\kappa}$ by derivations as follows:
    \[
        (f\partial_z)\cdot X\otimes g = X \otimes ( f\partial_z g ), \qquad (f\partial_z)\mathbf{1} = 0.
    \]
    This action is easily checked to be continuous and therefore induces an action by derivations on $\tilde{U}_{\kappa}(\hat{\g}_K)$.
\end{rmk}

\section{Spaces of distributions}

The space of formal series $V[[z^{\pm 1}]]$ plays a pivotal role in the theory of vertex algebras. In this chapter we are going to reformulate the theory of formal series in terms of distributions defined on an arbitrary localized $(\varphi)$-adic ring $K = R_\varphi$ equipped with a global coordinate $z \in R$, the usual case of formal series can be reconstructed considering $K = \C((z))$. Throughout the chapter $R$ will denote a fixed $(\varphi)$-adic ring with a fixed global coordinate $z$ and $K$ its localization $K = R_\varphi$. \newline
We start by elaborating on the following remark.

\begin{rmk}\label{rmkditsseries}
    Let $K^{\text{pol}} = \C[z^{\pm 1}]$ and let $V$ be a vector space over $\C$. Then the $\C[z^{\pm 1}][\partial_z]$ modules
    \[
        \Hom_{\C}(K^{\text{pol}},V) \quad \text{ and } \quad V[[z^{\pm 1}]]
    \]
    are isomorphic via the morphism
    \[
        \varphi \mapsto \sum_{n \in \Z} \varphi(z^n)z^{-n-1}
    \]
    Its inverse is obtained through the residue morphism $\int : V[[z^{\pm 1}]] \to V$ which picks the $z^{-1}$ coefficient.
    \[
        v(z) \mapsto \big( p(z) \mapsto \int (v(z)p(z)) \big).
    \]
    These maps are easily checked to commute with the action of the derivative $\partial_z$, where for $\varphi \in \Hom_{\C}(K^{\text{pol}},V)$ the action of $\partial_z$ is given by the following
    \[
        (\partial_z \cdot \varphi) (g) = \varphi(-\partial_z g).
    \]
    As a final remark notice that through this correspondence the residue map $\int : V[[z^{\pm 1}]] \to V$ corresponds to the restriction morphism $\restr : \Hom_{\C}(K^{\text{pol}},V)$ which evaluate a linear functional on $ 1 \in K^{\text{pol}}$.
\end{rmk}

We are going to study the space $\Hom_{\C}(K^{\text{pol}},V)$ which we call \textbf{space of distributions}. We are going to reformulate the basic theory of vertex algebras with the alternative language of distributions. This will be crucial for our purposes since this point of view allows us to extend the definition of a distribution to the ring $K_n$.

Before turning to the $(\varphi)$-adic case, we start by giving some definitions for a general commutative $A$-algebra $S$. 

\begin{defi}\label{defdistributions}
    Given a commutative $A$-algebra $S$ and an $A$-module $V$ we define the space of $V$-valued distributions as the $S$-module
    \[
        D^1_A(S,V) \stackrel{\text{def}}{=} \Hom_A(S,V).
    \]
    For $n\geq 1$ the space of $V$-valued $n$-distributions on $S$ is
    \[
        D^n_A(S,V) \stackrel{\text{def}}{=} \Hom_A(S^{\otimes n},V).
    \]
    As always tensor products are to be understood over $A$.
    
    If $z \in S$ we denote $z_i = 1 \otimes \dots \otimes 1 \otimes z \otimes 1 \otimes \dots \otimes 1$, where $z$ is placed in $i$-th position. In the case of $n=2$ we often write $z = z_1$ and $w=z_2$.
    
    If $S$ is equipped with a derivation $\partial$ we let $\partial_i = 1\otimes \dots \otimes \partial \otimes 1 \otimes \dots \otimes 1$, where $\partial$ is placed in the $i$-th position, and consider the spaces $D^n_A(S,V)$ as a module over $S^{\otimes n}[\partial_{1},\dots,\partial_{n}]$, where the action of $\partial_{i}$ is given by
    \[
        (\partial_{z_i}\cdot X) (g) = X(-\partial_{i} g).
    \]
\end{defi}

\begin{defi}
    There are restriction morphisms of $S^{\otimes n}[\partial_{1},\dots,\partial_{n}]$ modules
    \[
        \restr_i : D^{n+1}_A(S,V) \to D^{n}_A(S,V),
    \]
    which are dual to the maps
    \[
        \iota_i : S^{\otimes n} \to S^{\otimes n+1}, \qquad s_1\otimes\dots\otimes s_n \mapsto s_1\otimes\dots\otimes 1 \otimes\dots\otimes s_n,
    \]
    where $1$ is placed in the $i$-th position.
\end{defi}

\begin{notation}
    As always $K = R_\varphi$ will denote a fixed $(\varphi)$-adic ring over $A$ with a chosen coordinate $z \in R$. We will write $\partial = \partial_z$ for the unique continuous derivation of $K$ which satisfies $\partial_z z = 1$. From now on we are going to deal with $K$-distributions with values in $U$ (i.e. elements of $D^1_A(K,U)$), which from now on we assume to be a complete topological $A$-algebra with topology generated by left ideals. The main examples for such algebras, for us, are the following.
\end{notation}

\begin{itemize}
    \item Let $V$ be an $A$-module , which we consider as a topological $A$-module with the discrete topology. Consider the $A$-algebra $U = \End_A(V)$ of $A$-linear endomorphisms. We endow it with the topology of pointwise convergence, this is generated by intersections of left ideals of the form $I_v$ with $v \in V$
    \[
        I_v \stackrel{\text{def}}{=} \{ f \in \End_A(V) : f(v) = 0 \}, \quad v \in V.
    \]
    \item The complete topological $A$-algebra $\Tilde{U}_\kappa(\hat{\g}_{K})$ for a complete topological ring $K$ with a residue, introduced in Corollary \ref{envelopingcompletion}.
\end{itemize}

When dealing with topological rings or complete topological rings a natural condition we can impose on distributions is continuity.

\begin{defi}
    We define the space of \textbf{fields} on $K$, with values in a complete topological $A$-algebra $U$ with topology generated by left ideals, as the subspace of continuous distributions:
    \[
        F_A(K,U) \stackrel{\text{def}}{=} \Hom^{\text{cont}}_A(K,U) \subset D^1_A(K,U).
    \]
    Since $\partial_z$ and the multiplication are continuous this is a $K[\partial_z]$-submodule of $D^1_A(K,U)$.
\end{defi}

In the case where there exists another ring $K^{\text{pol}}$ whose completion is $K$ (like in the case $K = K_n$), the restriction map yields an isomorphism
\[
    F_A(K,U) \stackrel{\simeq}{\longrightarrow} F_A(K^{\text{pol}},U).
\]
This is because we are assuming $U$ to be complete.

\begin{example}
    When $A = \C$, $K^{\text{pol}} = \C[t^{\pm 1}]$, $K = \C((t))$ and $U = \End_{\C} V$ for a vector space $V$, the notion of field just introduced coincides with the usual notion of field. Indeed the continuity condition on a distribution $X : \C[t^{\pm 1}] \to \End_{\C} V$ reads exactly as: for any $v \in V$ 
    \[
        X_n = X(t^{n}) \in I_v \text{ for } n \gg 0,
    \]
    Which can be rewritten as
    \[
        X_n(v) = 0 \text{ for } n \gg 0.
    \]
\end{example}

\begin{rmk}
    If $U_\alpha$ is a fundamental system of neighborhoods of $0$ in $U$ consisting of left ideals, it will be often useful to work over $U/U_\alpha$. Notice that there is a natural map
    \[
        D^1_A(K,U) \to D^1_A(K,U/U_\alpha)
    \]
    and that an element $X \in D^1_A(K,U)$ is a field if and only if its image in $D^1_A(K,U/U_\alpha)$ is a field for every $\alpha$, considering $U/U_\alpha$ with the discrete topology.
\end{rmk}

When dealing with distributions with values in an algebra $U$ it makes sense to consider product of fields. In particular there is a natural map
\[
    D^1_A(K,U)\otimes_A D^1_A(K,U) \to D^2_A(K,U) \qquad (X,Y) \mapsto XY
\]
where $XY$ is the $2$-distribution defined by
\[
    XY(f\otimes g) = X(f)Y(g).
\]
In particular when $U$ is not commutative it is interesting to consider the $2$-distribution $[X,Y] \in D^2_A(K,U)$ defined by the bracket of $X$ and $Y$
\[
    [X,Y](f\otimes g) = X(f)Y(g) - Y(g)X(f).
\]
Let $\sigma$ be the automorphism of $K\otimes K$ which permutes the two factors. With the above notation we have
\[
    [X,Y] = XY - YX\circ\sigma.
\]

\begin{rmk}
    Let $X,Y \in F_A(K,U)$ be arbitrary fields. Then the product $XY : K\otimes K \to U$ is continuous for the middle topology. In particular, by the Lemma \ref{continuousextension} it extends to
    \[
        XY : \hat{K\otimes K} \to U.
    \]
\end{rmk}

\begin{defi}
    We define the space of \textbf{$2$-fields} $F^2_A(K,U)$ to be
    \[
        F^2_A(K,U) \stackrel{\text{def}}{=} \Hom_A^{\text{cont}}(\hat{K\otimes K},U).
    \]
    By the remark above there multiplication in $U$ induces a natural bilinear map
    \[
        F_A(K,U)\otimes_A F_A(K,U) \to F^2_A(K,U).
    \]
\end{defi}

\subsection{Local fields, \texorpdfstring{$n$}{n}-products of fields}\label{secnproducts}

We continue here building a theory of vertex algebras for fields from an arbitrary $(\varphi)$-adic ring equipped with a global coordinate.

\begin{defi}
    Two fields $X$ and $Y$ are \textbf{mutually local} if there exists a non negative integer $N$ such that
    \[
        (z-w)^N[X,Y] = 0.
    \]
\end{defi}

\begin{defi}\label{defpositivenprod}
    Let $X$ and $Y$ be two mutually local $K$-fields, and $n \geq 0$ an integer. We define the $n$-product of $X$ and $Y$, denoted $X_{(n)}Y$, as
    \[
        X_{(n)}Y = \restr_1 (z-w)^{n}[X,Y],
    \]
    or equivalently 
    \[
        (X_{(n)}Y)(g) = (z-w)^n[X,Y](1\otimes g).
    \]
    This is obviously continuous and hence a field.
\end{defi}

\begin{prop}\label{taylorfields}
    Let $X$ and $Y$ be two fields such that $(z-w)^{N+1}[X,Y] = 0$. Then for any $f,g \in K$ the following Taylor formula holds
    \[
        [X,Y] (f\otimes g) = \sum_{n=0}^N \frac{1}{n!} (X_{(n)}Y)(g\partial_z^nf).
    \]
\end{prop}

\begin{proof}
    Since $z$ is a coordinate for $K$, by Proposition \ref{taylor}, the following formula holds
    \[
        f\otimes g = \sum_{n\geq 0}^N \frac{1}{n!}(z-w)^n 1\otimes( g\partial_z^n f) \quad \text{ mod } (z-w)^{N+1} \hat{K\otimes K}.
    \]
    Now recall that $[X,Y]$ naturally belongs to $F^2_A(K,U)$ and by hypothesis it is in the image of the map
    \[
        Hom_A\bigg(\frac{\hat{K\otimes K}}{(z-w)^{N+1}\hat{K\otimes K}},U\bigg) \to Hom_A^{\text{cont}}(\hat{K\otimes K},U).
    \]
    The result easily follows.
\end{proof}

\begin{rmk}
    For $K = \C((z))$ we can translate from the language of distributions to the language of formal power series, via the identification described in Remark \ref{rmkditsseries}. In this case the usual $\delta$-function $\delta(z-w) = \sum_{n\in\Z} z^{-n-1}w^n$ corresponds to the $2$-distribution $\delta(f\otimes g) = \int fgdz$, and the product $X(w)\delta(z-w)$ corresponds to the $2$-distribution $(X\delta) (f\otimes g) = X(fg)$. Likewise the formal power series $X(w)\partial^n_w\delta(z-w)$ corresponds to the $2$-distribution $(X\partial^n_w\delta)(f\otimes g) = X(g\partial^nf)$. Hence the Taylor formula above exactly translates the usual OPE
    \[
        [X(z),Y(w)] = \sum_{n\geq 0 } \frac{1}{n!}(X_{(n)}Y)(w)\partial_w^n\delta(z-w).
    \]
\end{rmk}

There is also a converse statement to Proposition \ref{taylorfields}.

\begin{prop}\label{proplocality}
    Let $Z \in F^2_A(K,U)$ be a $2$-field and suppose that there exist fields $\gamma_n$ such that
    \[
        Z(f\otimes g) = \sum_{n=0}^N \frac{1}{n!}\gamma_n(g\partial_z^nf),
    \]
    then $(z-w)^{N+1}Z = 0$ and for every $n \in \{ 0,\dots,N \}$
    \[
    \gamma_n = \restr_1\big( (z-w)^nZ \big).
    \]
\end{prop}

\begin{proof}
    We prove the assertion by induction on $N$.
    Let $N=0$, by hypothesis we have
    \[
        Z(f\otimes g) = \gamma(fg),
    \]
    in particular
    \[
        Z\big( (z-w)f\otimes g\big) = \gamma(zfg) - \gamma(zfg) = 0
    \]
    so that $(z-w)Z=0$. The statement $ \restr_1 Z = \gamma$ is obvious. Now assume that the result is true for all integers up to $N$ and lets prove it for $N + 1$. We have
    \[
        Z(f\otimes g) = \sum_{n=0}^{N+1} \frac{1}{n!}\gamma_n(g\partial_z^nf).
    \]
    It is clear that $\restr_1 Z = Z(1\otimes \_) = \gamma_0$ and that by an easy computation one may prove that, if we let $Z' = (z-w)Z$,
    \[
        Z'(f\otimes g) = \sum_{n=0}^{N} \frac{1}{n!}\gamma_{n+1}(g\partial_z^nf).
    \]
    From this we deduce, by the inductive hypothesis, that $0 = (z-w)^{N+1}Z' = (z-w)^{N+2}Z$ and that for $n\geq 1$, we have $\gamma_n = \restr_1 (z-w)^{n-1}Z' = \restr_1 (z-w)^n Z$.
\end{proof}

Our next goal is to define the $n$ products for negative $n$. Notice that since the topology of $U$ is generated by left ideals the distributions $XY$ and $YX\circ \sigma$ are continuous with respect to the right an left topology respectively. Therefore they induce continuous morphisms
\[
    XY : \widehat{K\otimes K}^{2} \to U \quad \text{ and } \quad YX\circ\sigma : \widehat{K\otimes K}^{1} \to U.
\]

A corollary of Remark \ref{invertibility} is the following.

\begin{corollary}
    For every $n \in \Z$ and fields $X,Y$ we have
    \[
        (z-w)^nXY \in Hom^{\text{cont}}_A(\widehat{K\otimes K}^2,U) \quad \text{and} \quad (z-w)^nYX \circ\sigma \in Hom^{\text{cont}}_A(\widehat{K\otimes K}^1,U).
    \]
    Hence both $(z-w)^nXY$ and $(z-w)^nYX\circ\sigma$, when restricted to $\hat{K\otimes K}$ define elements of $F^2_A(K,U)$ for every $n \in \Z$.
\end{corollary}

\begin{defi}
    Let $X$ and $Y$ be two fields and $n$ a negative integer we define
    \[
        X_{(n)}Y = \restr_1\bigg( (z-w)^nXY - (z-w)^nYX\circ \sigma \bigg).
    \]
    Note that for $n \geq 0$ this formula reduces to the one in Definition \ref{defpositivenprod},while for $n < 0$ we need to write the two terms separately since the multiplication by $(z-w)^n$ is done in two different spaces, namely $\Hom_A^{\text{cont}}(\widehat{K\otimes K}^2,U)$ for the first one and $\Hom_A^{\text{cont}}(\widehat{K\otimes K}^1,U)$ for the second one. 
\end{defi}

\begin{prop}\label{nprodderivative}
    Let $X,Y \in F_A(K,U)$ be fields and let $n \in \Z$. Then
    \begin{align*}
        (\partial X)_{(n)}Y &= -nX_{(n-1)}Y, \\
        \partial \cdot  (X_{(n)}Y) &= (\partial X)_{(n)}Y + X_{(n)}(\partial Y).
    \end{align*}
\end{prop}

\begin{proof}
    We prove the above formulas for $n\geq 0$. For $n<0$ the proof is the same but we need to separate the two terms, and notice that the derivatives $\partial_z,\partial_w$ extend to the completions $\widehat{K\otimes K}^i$.
    
    Let us show the first formula
    \begin{align*}
        \big((\partial X)_{(n)}Y\big)(f) = [\partial X,Y]\big((z-w)^n1\otimes f) &= [X,Y]\big( -\partial_z (z-w)^n1\otimes f\big) \\ &= [X,Y]\big( -n(z-w)^{n-1}1\otimes f) = -n X_{(n-1)}Y(f).
    \end{align*}
    Then the second one follows by a similar argument
    \begin{align*}
        \partial \cdot\big(X_{(n)}Y\big) (f) = X_{(n)}Y(-\partial f) &= [X,Y]\big( (z-w)^n1\otimes (-\partial f) \big) \\
        &= [X,Y]\big( (-\partial_z - \partial_w) ((z-w)^n1\otimes f) \big) \\ 
        &= \big( (\partial_z + \partial_w) \cdot [X,Y] \big)\big((z-w)^n 1\otimes f \big) \\
        &= \big( ([\partial X,Y] + [X,\partial Y] \big)\big((z-w)^n 1\otimes f \big) \\
        &= \big((\partial X)_{(n)}Y\big)(f) + \big(X_{(n)}(\partial Y)\big)(f)
    \end{align*}
\end{proof}

\begin{lemma}\label{continuitynproduct}
    Given two fields $X,Y$ and $n \in \Z$, the $n$-product $X_{(n)}Y$ is again a field.
\end{lemma}

\begin{proof}
    We only need the prove the assertion for negative $n$.
    We will prove that the two $1$-distributions
    \[
        \restr_1 (z-w)^{-1}XY \quad \text{and} \quad \restr_1 (z-w)^{-1}YX\circ \sigma
    \]
    are fields. The assertion for arbitrary $n$ follows from Proposition \ref{nprodderivative}.
    
    For the first one it is enough to notice that it is the composition of the following continuous maps
    \[
        K \stackrel{g \mapsto 1\otimes g}{\longrightarrow} K\otimes K \to \widehat{K\otimes K}^2 \stackrel{(z-w)^n\cdot}{\longrightarrow} \widehat{K\otimes K}^2 \stackrel{XY}{\longrightarrow} U.
    \]
    For the second one consider a left ideal $U_\alpha$ which is a neighborhood of $0$ in $U$. Then there exists $N \geq 0$ such that
    \[
        X(\varphi^n) = 0 \text{ mod } U_{\alpha} \qquad \forall n \geq N +1.
    \]
    As usual $\varphi$ denotes a generator of the ideal $I \subset R$ which gives the topology on $K$. It is then easy to check, using the explicit expression appearing in Lemma \ref{invertibility} that
    \[
        \frac{1}{z-w}YX\circ \sigma (1\otimes g) = \sum_{n=0}^N -hY(g\varphi^{-n-1})X(\varphi^{n}) \in U/U_{\alpha}.
    \]
    It is then pretty clear that for $g \in I^m$, for sufficiently large $m$ this sum is going to be $0$ in the quotient $U/U_\alpha$. This proves continuity.
\end{proof}

\begin{rmk}
    When the commutator $[X,Y] = 0$ we get that both $XY$ and $YX\circ\sigma$ are continuous with respect to the sum topology and the products $X_{(n)}Y$ are obviously $0$. On the other hand, even though $XY = YX\circ\sigma$ is continuous with respect to the naive sum topology, $X_{(n)}Y$ for $n< 0$ is not necessarily $0$, since the completion of $K \otimes K$ with respect to the sum topology is not a ring (see Remark \ref{notaring}.
\end{rmk}

\begin{lemma}\label{normalordercommutative}
    If $[X,Y]= 0$ then $X_{(-1)}Y = Y_{(-1)}X$.
\end{lemma}

\begin{proof}
    Note that $XY =YX\circ\sigma$ is continuous for the sum topology of $K\otimes K$. Since we need to multiply by $(z-w)^{-1}$, to stress if we are calculating its action on $\widehat{K\otimes K}^1$ (resp. $\widehat{K\otimes K}^2$) we are going to denote the multiplications as $E_1(z-w)^{-1}XY$ (resp. $E_2(z-w)^{-1}XY$) in particular
    \[
        X_{(-1)}Y(g) = E_2(z-w)^{-1}XY (1\otimes g) - E_1(z-w)^{-1}YX\circ\sigma (1\otimes g).
    \]
    We are going to prove that, for $g \in K$
    \[
        \big(X_{(-1)}Y\big)(g) = \big(Y_{(-1)}X\big)(g).
    \]
    By the assumption of $z \in K$ being a coordinate it follows that there exists an element $h \in \hat{K\otimes K}$ such that $(z-w)h = g\otimes 1 - 1\otimes g$. Using $XY = YX\circ\sigma$, we easily get
    \[
        E_2(z-w)^{-1}XY (1\otimes g) - E_2(z-w)^{-1}YX (1\otimes g) = -XY(h),
    \]
    and analogously
    \[
        E_1(z-w)^{-1}YX\circ\sigma (1\otimes g) - E_1(z-w)^{-1}XY\circ\sigma (1\otimes g) = -YX\circ\sigma(h).
    \]
    It follows that 
    \[
        \big(X_{(-1)}Y(\big)g) - \big(Y_{(-1)}X\big)(g) = XY(h) - YX\circ\sigma(h) = 0,
    \]
    where the last equality follows from $[X,Y] = 0$.
\end{proof}

Recall that the space of fields $F_A(K,U)$ is naturally a $K$ module.

\begin{lemma}\label{lemmalinearitynormalorder}
    The $(-1)$-product is $K$-linear in the second factor. 
\end{lemma}

\begin{proof}
    Recall that for any $g \in K$ and any field $Y$ the product $gY$ is the field defined by the formula $(gY)(h) = Y(gh)$. We need to prove that given an arbitrary field $X$ the following equality holds.
    \[
        \big(X_{(-1)}(gY)\big)(h) = \big(X_{(-1)}Y\big)(gh) = \big( g(X_{(-1)}Y)\big)(h).
    \]
    This is in some sense obvious. Consider the commutative diagram
    \[\begin{tikzcd}
	{K\otimes K} && {K\otimes K} \\
	\\
	U
	\arrow["{X(gY)}"', from=1-1, to=3-1]
	\arrow["XY", from=1-3, to=3-1]
	\arrow["{(1\otimes g)\cdot}", from=1-1, to=1-3]
\end{tikzcd}\]
    All maps are continuous for the right topology and therefore we get:
    \[
        \bigg(\frac{1}{z-w}X(gY)\bigg)(1\otimes h) = \big(X(gY)\big)\bigg( \frac{1\otimes h}{z-w}\bigg) = (XY)\bigg( \frac{1\otimes gh}{z-w}\bigg) = \bigg(\frac{1}{z-w}XY\bigg)(1\otimes gh).
    \]
    The commutativity of the above diagram is needed for the second equality. Analogously we get
    \[
        \bigg(\frac{1}{z-w}(gY)X\bigg)(h\otimes 1) = \bigg(\frac{1}{z-w}XY\bigg)(hg\otimes 1).
    \]
    Now it's easy to conclude with the definition of the $(-1)$-product.
\end{proof}

In what follows we are going to develop the theory of distributions and fields in a completely analogous way to the classical theory which involves space of series. We start by proving Dong's lemma within this formalism and move forward to prove that a closed family of mutually local fields forms a vertex algebra.

\begin{lemma}[Dong's Lemma for distributions]\label{dong}
    Let $X,Y,Z$ be mutually local fields, then 
    \begin{enumerate}
        \item $\partial_z X$ and $Y$ are again mutually local;
        \item For any $n\in \Z$ the fields $X_{(n)}Y$ and $Z$ are mutually local.
    \end{enumerate}
\end{lemma}

\begin{proof}
    To prove $(1)$ it suffice to notice that
    \[
        (z-w)^n[\partial X,Y] = (z-w)^n\partial_z[X,Y] = -n(z-w)^{n-1}[X,Y] + \partial_z \big((z-w)^n[X,Y]\big),
    \]
    so if $(z-w)^N[X,Y]=0$ we have $(z-w)^{N+1}[\partial X,Y]=0$.
    
    In the proof of $(2)$ we will limit ourselves to the case $n\geq -1$. The other cases then follow by $(1)$ and Proposition \ref{nprodderivative}.
    Consider the following $3$-distributions
    \begin{align*}
        \Phi_1 : f\otimes g\otimes h &\mapsto X(f)Y(g)Z(h) - Z(h)X(f)Y(g) \\
        \Phi_2 : f\otimes g\otimes h &\mapsto Y(g)X(f)Z(h) - Z(h)Y(g)X(f)
    \end{align*}
    Choose $N$ such that $(z-w)^N[X,Y] = (z-w)^N[X,Z] = (z-w)^N[Y,Z] = 0$. We first prove that
    \[
        (z_2 - z_3)^{3N}\Phi_1 = (z_2-z_3)^{3N}\Phi_2.
    \]
    Write
    \[
        (z_2-z_3)^{3N} = \sum_{s=0}^{2N}\binom{2N}{s}(z_2-z_1)^{2N -s}(z_1-z_3)^{s}(z_2-z_3)^N.
    \]
    For $s \in [0,N]$ we have
    \begin{equation}\label{eqboh}   
        (z_2-z_1)^{2N -s}\Phi_1 = (z_2-z_1)^{2N - s} \Phi_2
    \end{equation}
    since we can commute $X$ with $Y$.
    While for $s > N$ we have
    \begin{equation}\label{sgreatern}
        (z_1-z_3)^{s}(z_2-z_3)^N \Phi_1 = 0 = (z_1-z_3)^{s}(z_2-z_3)^N \Phi_2 
    \end{equation}
    since we can commute $X$ with $Z$ and $Y$ with $Z$.
    Fix $n\geq 0$. Since the multiplication by $(z_2 - z_3)^{3N}$ commutes with taking the restriction $\restr_1$ we obtain
    \[
        (z_2-z_3)^{3N}[X_{(n)}Y,Z] = (z_2 - z_3)^{3N} Res_1\big( (z_1-z_2)^n(\Phi_1 - \Phi_2) \big) = 0.
    \]
    To prove the statement for $n = -1$ fix $h$. As in the previous case we work out the cases $s> N$ and $s\leq N$. Consider first the case in which $s > N$. Notice that $\Phi_1$ is continuous with respect to the right topology while $\Phi_2$ is continuous with respect to the left topology. We get that $\frac{1}{z_1-z_2}\Phi_i$ is well defined although in different spaces. The same continuity conditions are true for $(z_1-z_3)^s(z_2-z_3)^N\Phi_i$. In particular we get
    \[
        (z_1-z_3)^s(z_2-z_3)^N\big(\frac{1}{z_1-z_2} \Phi_i \big) = \frac{1}{z_1-z_2}\big((z_1-z_3)^s(z_2-z_3)^N \Phi_i \big) = 0
    \]
    by Equation \eqref{sgreatern}. Now consider the case in which $s  \leq N$. We have
    \[
        (z_1-z_2)^{2N +1-s}\big(\frac{1}{z_1-z_2}\Phi_1\big) = (z_1-z_2)^{2N -s}\Phi_1,
    \]
    which coincides with the same expression with $\Phi_1$ replaced by $\Phi_2$ by equation \eqref{eqboh}.
    Summing everything up and taking the restriction $\restr_1$, we get
    \begin{align*}
        (z_2-z_3)^{3N+1}[X_{(-1)}Y,Z](g\otimes h) = 0
    \end{align*}
    as desired.
\end{proof}

\subsection{Distributions and vertex algebras}

\begin{defi}
    A $\C$-linear subspace $ \mathcal{F} \subset F^1_A(K,U)$ is said to be \textbf{closed} if for any $X,Y \in \mathcal{F}$, $n\in \Z$ the fields $X_{(n)}Y$ and $\partial X$ belong to $\mathcal{F}$. We define the \textbf{closure} $\overline{\mathcal{F}}$ of a subspace $\mathcal{F} \subset F^1_A(K,U)$ to be the smallest closed subspace of $F^1_A(K,U)$ containing $\mathcal{F}$. 
\end{defi}

Dong's lemma has the following fundamental corollary.

\begin{prop}
    Let $\mathcal{F} \subset F^1_A(K,U)$ be a $\C$-linear subspace. If $\F$ consists of mutually local fields then  $\overline{\F}$ consists of mutually local fields as well.
\end{prop} 

\begin{proof}
    Consider the following chain of subspaces. Let $\F_0 = \F$ and define inductively $\F_{n+1}$ to be the $\C$-linear span of all fields in $\F$, all their $n$ products as $n$ varies and all their derivatives. Then it is quite clear that $\overline{\F} = \bigcup_{n\geq 0} \F_{n}$ and, by Lemma \ref{dong} it is clear that each $\F_n$ consists of mutually local fields.
\end{proof}

We are finally able to state and prove the main theorem of this section. This should be thought as an analogue of of \cite[Prop. 3.2]{kac1998vertex}.

\begin{thm}\label{vertexalgebras}
    Let $\F$ be a closed $\C$-linear subspace of $F^1_A(K,U)$ consisting of mutually local fields and let $\mathbf{1} : K \to U$ be a field such that $\mathbf{1}(g) \in U$ is central for every $g \in K$, such that $\partial \mathbf{1} = 0$, and such that  $\mathbf{1}_{(-1)}X = X$ for all $X \in \F$. \newline
    Then $\F + \C\mathbf{1}$ endowed with the $n$-products and with the derivation $T = \partial$ is a vertex algebra with $\mathbf{1}$ as the vacuum vector.
\end{thm}

\begin{proof}
    Since $\mathbf{1}(g) \in U$ is central for every $g \in K$, we have $[X,\mathbf{1}] = 0$. Therefore for all $X \in \F$ and $n\geq 0$ we have
    \[
        X_{(n)}\mathbf{1} = \mathbf{1}_{(n)}X = 0.
    \]
    Since $\partial \mathbf{1} = 0$ all products $\mathbf{1}_{(n)}X$ for $n \leq-2$ vanish, while $\mathbf{1}_{(-1)}X = X_{(-1)}\mathbf{1} = X$ by assumption. Hence the vertex algebra vacuum axioms hold. Moreover the translation covariance axiom hold by Proposition \ref{nprodderivative}. We are left to prove the locality axiom. 
    
    The state-field correspondence $\Y : \mathcal{F} \to \End_{\C}(\mathcal{F})[[x^{\pm 1}]]$ is defined by
    \[
        \Y(X,x)Z \stackrel{\text{def}}{=} \sum_{n\in \Z} X_{(n)}Zx^{-n-1}.
    \]
    Note that $\Y(X,x)$ is indeed a field with respect to $x$ since we are assuming that all the elements in $\F$ are mutually local. We claim that if $(z-w)^N[X,Y] = 0$ then
    \[
        (x-y)^N[\Y(X,x),\Y(Y,y)] = 0.
    \]
    
    When taking expansions we are going to write expressions like
    
    \[
        z_{12}^nXY \quad \text{or} \quad z_{21}^nYX,
    \]
    meaning that we consider respectively the field $f\otimes g \mapsto XY((z_1-z_2)^nf\otimes g)$ where the argument $(z_1-z_2)^n f\otimes g$ is considered as an element of $\widehat{K\otimes K}^2$, and the field $f\otimes g \mapsto YX((z_2-z_1)^ng\otimes f)$ where the argument is considered as an element of $\widehat{K\otimes K}^1$. We remark here that $(z_1 - z_2)z_{12}^nXY = z_{12}^{n+1}XY$ and $(z_1 - z_2)z_{21}^nYX = z_{21}^{n+1}YX$.
    
    Using the above notation we write
    \[
        X_{(n)}(Y_{(m)}Z)(h) = \bigg( z_{13}^nX\big(z_{23}^mYZ -z_{32}^mZY\big) - z_{31}^n\big(z_{23}^mYZ -z_{32}^mZY\big)X \bigg)(1\otimes 1 \otimes h),
    \]
    and similarly for $Y_{(m)}(X_{(n)}Z)$
    \[
        Y_{(m)}(X_{(n)}Z)(h) = \bigg( z_{23}^mY\big(z_{13}^nXZ -z_{31}^nZX\big) - z_{32}^m\big(z_{13}^nXZ -z_{31}^nZX\big)Y \bigg)(1\otimes 1 \otimes h).
    \]
    Now notice that
    \begin{align}\label{eqassociativitydelta}
        z^n_{13}X(z_{32}^mZY) &= z^m_{32}(z_{13}^nXZ)Y \quad \text{ and } \\
        z_{31}^n(z_{23}^mYZ)X &= z_{23}^m Y(z_{31}^nZX).
    \end{align}
    as $3$-distributions on $K$. We give the idea how to show this for the first one, the second one is obtained simply by exchanging $X$ with $Y$ and $n$ with $m$. The argument is makes use of an explicit computations, analogous to the one found in the proof of Lemma \ref{continuitynproduct}. Fix an open neighborhood $U_\alpha$ of $0$ in $U$, we work over $U/U_\alpha$. We have $\varphi_3 - \varphi_2 = 1\otimes 1 \otimes \varphi - 1\otimes\varphi\otimes 1 = (z_3 - z_2)\delta_{32}$ for some $\delta_{32} \in 1 \otimes \hat{K\otimes K}$. We may expand, for example $z_{32}^{-1} = \delta_{32}\sum_{l \geq 0} \varphi_3^{-l-1}\varphi_2^l \in 1 \otimes \widehat{K\otimes K}^2$. Similarly for $z_{32}^m$, $m \in \Z$ and $z_{13}^n$, $n\in \Z$. When computing $\big( z_{13}^nX(z_{32}^mZY) \big)(f\otimes g \otimes h)$ in $U/U_\alpha$ all sums in the expansions of $z_{12}^n$ and $z_{32}^m$ become finite. One then easily checks that the identity in \eqref{eqassociativitydelta} by direct computation.
    
    We obtain that
    \begin{align*}
        X_{(n)}(Y_{(m)}Z)(h) - Y_{(m)}(X_{(n)}Z)(h) &= \\ 
        \bigg( z_{13}^nX\big(z_{23}^mYZ) -&z_{23}^mY\big(z_{13}^nXZ)- z_{32}^m(z_{31}^nZX)Y + z_{31}^n(z_{32}^mZY)X \bigg)(1\otimes 1 \otimes h).
    \end{align*}
    Now consider the right hand side as a $3$ distribution on $K$, and call it $\Phi_{n,m}$.
    We claim that
    \begin{align*}
        (z_1-z_2)^N\big(z_{13}^nX(z_{23}^mYZ) -z_{23}^mY\big(z_{13}^nXZ)\big) &= 0  \quad \text{ and } \\
        (z_1-z_2)^N\big(z_{32}^m\big(z_{31}^nZX)Y - z_{31}^n(z_{32}^mZY)X\big) &= 0
    \end{align*}
    as distributions. All we have to pay attention to is in which completion we are taking our expansions. The first one is immediate since we can take all the expansions in the same space, namely $\widehat{K\otimes K\otimes K}^3$, and then use the fact that $(z_1-z_2)^N[X,Y] = 0$.
    
    To prove the second one we need to consider an open neighborhood $U_\alpha$ of $0$ in $U$ and work in $U/U_\alpha$. Here as always we can write everything as a finite sum of $3$ distributions which are $0$ after being multiplied by $(z_1-z_2)^N$. \newline Putting together these two equations we get
    \begin{equation}\label{eqphinmnull}
        (z_1-z_2)^N\Phi_{n,m} = 0.
    \end{equation}
    In addition it is quite clear that
    \begin{align}\label{eqphinmtranslation1}
        (z_1-z_3)\Phi_{n,m} &= \Phi_{n+1,m} \quad \text{ and }\\
        \label{eqphinmtranslation2}(z_2-z_3)\Phi_{n,m} &= \Phi_{n,m+1}.
    \end{align}
    Consider now the series in $D^3_A(K,U)[[x^{\pm 1},y^{\pm 1}]]$ (which is naturally a $K^{\otimes 3}[x^{\pm 1},y^{\pm 1}]$ module)
    \begin{equation}\label{phitranslation}
        \Phi = \sum_{n,m\in \Z} \Phi_{n,m}x^{-n-1}y^{-m-1}.
    \end{equation}
    By definition
    \[
        [\Y(X,x),\Y(Y,y)]Z = \restr_1\restr_2 \Phi.
    \]
    We are going to show that $(x-y)^N\Phi = 0$ which implies $(x-y)^N[\Y(X,x),\Y(Y,y)]Z = 0$.
    Write
    \[
        (x-y)^N = \bigg( \big(x - (z_1-z_3)\big) - \big(y - (z_2-z_3)\big) + (z_1-z_2)\bigg)^N.
    \]
    Combining equations \eqref{eqphinmtranslation1}, \eqref{eqphinmtranslation2} and \eqref{phitranslation} we have
    \[
        (x - (z_1-z_3))\Phi = (y - (z_2-z_3))\Phi = 0,
    \]
    and therefore
    \[
        (x-y)^N\Phi = (z_1-z_2)^N\Phi = 0.
    \]
\end{proof}

\subsection{Actions by \texorpdfstring{$K$}{K}-fields}

We work as always with an $(\varphi)$-adic ring $R$ and with its localization $K$, with a chosen global coordinate $z \in R$ (c.f. Definition \ref{coordinate}). Recall that by Theorem \ref{vertexalgebras} a closed subspace of $F^1_A(K,U)$ consisting of mutually local fields is essentially a vertex algebra with the $n$-products defined in Section \ref{secnproducts}.

\begin{defi}\label{defiactionkfields}
    Let $V$ be a vertex algebra  over $\C$ and $U$ a complete topological $A$ algebra with topology generated by left ideals. An \textbf{action by $K$-fields} of $V$ on $U$ is a $\C$-linear map
    \[
        Y: V \to F_A(K,U)
    \]
    such that $Y$ commutes with $n$-products and derivatives:
    \[
        Y(A)_{(n)}Y(B) = Y(A_{(n)}B), \qquad Y(TA) = \partial Y(A),
    \] 
    and such that the image of $V$ consists of mutually local $K$-fields with values in $U$.
\end{defi}

In the special case when $A = \C$, $K=\C((t))$ and $U = \End_{\C}(W)$, Definition \ref{defiactionkfields} boils down to the usual definition of a module of the vertex algebra $V$.

\begin{rmk}
    Let $Y$ be an action by $K$-fields of $V$ on $U$. Then the image $Y(V)$ is a $\C$-subspace of $F_A(K,U)$ which is closed an consistst in mutually local fields. Let $\mathbf{1} = Y(\vac)$. Then $\mathbf{1}_{(n)}X = X\delta_{n,-1}$ for any $X \in Y(V)$ and for all $n\geq -1$ so by Theorem \ref{vertexalgebras} the $\C$-vector space $Y(V)$ has the structure of a vertex algebra and the map $Y : V \to Y(V)$ is a morphism of vertex algebras.
\end{rmk}

\begin{defi}\label{defliekv}
    Let $V$ be a vertex algebra over $\C$. Define 
    \[
        \Lie_K(V) \stackrel{\text{def}}{=} \frac{V\otimes_{\C}K}{Im\, T + \partial_z}.
    \]
    This is an $A$-module. Denote by $Xf$ the image of $X\otimes f$ in the quotient and define the following bracket on $\Lie_K(V)$ by
    \[
        [Xf,Yg] \stackrel{\text{def}}{=} \sum_{n\geq 0} \frac{1}{n!}(X_{(n)}Y)g\partial_z^nf.
    \]
\end{defi}

\begin{prop}
    $\Lie_K(V)$ is an $A$-linear Lie algebra with the bracket just defined.
\end{prop}

\begin{proof}
    This is a known statement that holds when $K$ is any commutative algebra over $\C$ with a derivation $\partial$. It can be found for instance in \cite[Prop. 3.2.1]{frenkel2007langlands}.
\end{proof}

\begin{prop}
    Let $Y$ be an action by $K$-fields of $V$ on $U$. Consider the $A$-linear map
    \[
        \Lie(Y) : \Lie_K(V) \to U, \qquad Xf \mapsto Y(X)(f).
    \]
    Then $Lie(Y)$ is a well defined Lie algebra homomorphism, where the bracket on $U$ is given by taking the commutator of its associative product.
\end{prop}

\begin{proof}
    We have $Y(Ta) = \partial Y(a)$ for any $a \in V$ in particular $Y(Ta)(f) = -Y(a)(\partial_z f)$, so the map $\Lie(Y)$ is well defined. Recall that by Proposition \ref{taylorfields} the following formula holds
    \[
        [Y(a)(f),Y(b)(g)] = \sum_{n\geq 0} \frac{1}{n!}\big( Y(a)_{(n)}Y(b)\big)(g\partial_z^nf).
    \]
    Since by definition of an action by $K$-fields of $V$ on $U$, $Y(a)_{(n)}Y(b) = Y(a_{(n)}b)$, the assertion easily follows.
\end{proof}

\section{Spaces of fields on \texorpdfstring{$K_n$}{Kn}}\label{knstarts}

We are now going to restrict our attention to the case in which $K = K_n$, let $z$ be the canonical coordinate.

\begin{lemma}\label{unity}
    Let $U$ be a complete topological $A$-algebra with topology generated by left ideals. Consider the $K_n$-field
    \[
        \mathbf{1} \stackrel{\text{def}}{=} 1_U\int: \qquad f \mapsto 1_U\bigg( \int fdz \bigg) .
    \]
    where $1_U$ is the unity of $U$.
    Then for any $K_n$-field $X$ on $U$ and $n \in \Z$ the following holds:
    \[
        \mathbf{1}_{(n)}X = X\delta_{n,-1},\qquad X_{(-1)}\mathbf{1} = X.
    \]
\end{lemma}

\begin{proof}
    Since $\mathbf{1}$ obviously commutes with every field, we naturally have $\mathbf{1}_{(n)}X = 0$ for any $n\geq 0$. In addition, since the residue satisfies $\int df = 0$, we get $\partial \mathbf{1} = 0$, so, for any $n\leq -2$ the $n$-product is equal to zero: $\mathbf{1}_{(n)}X = 0$. All is left to prove is that $\mathbf{1}_{(-1)}X = X$, since, by Lemma \ref{normalordercommutative}, this also implies $X_{(-1)}\mathbf{1}= X$. Recall that $\mathbf{1}$ is $0$ on $R_n$ (the residue of a regular function is $0$). Hence we have
    \[
        \mathbf{1}_{(-1)}X = \big(E_2(z-w)^{-1}\mathbf{1}X\big)(1\otimes g) - \big(E_1(z-w)^{-1}X\mathbf{1}\circ\sigma\big)(1\otimes g).
    \]
    Note that the second term in the right hand side above vanishes since $E_1(z-w)^{-1}1\otimes g \in \widehat{R_n\otimes K_n}^1$ and $\mathbf{1}(R_n) = 0$. Hence
    \[
        \mathbf{1}_{(-1)}X = (\mathbf{1}X)\bigg( \frac{1\otimes g}{z-w} \bigg) \quad \text{where} \quad \frac{1\otimes g}{z-w} \in \widehat{K_n\otimes K_n}^2.
    \]
    Now notice that, by $A$-linearity, the following diagram is commutative
    \[\begin{tikzcd}
	{K_n \otimes K_n} && K_n && {f\otimes g} && {g\int fdz} \\
	\\
	& U &&&& {X(g)\int fdz}
	\arrow["{\int_z}", from=1-1, to=1-3]
	\arrow["{\mathbf{1}X}"', from=1-1, to=3-2]
	\arrow["X", from=1-3, to=3-2]
	\arrow[maps to, from=1-5, to=3-6]
	\arrow[maps to, from=1-7, to=3-6]
	\arrow[maps to, from=1-5, to=1-7]
\end{tikzcd}\]
    In particular since both $\int_z$ and $\mathbf{1}X$ are continuous for the right topology we get the following formula
    \[  
        (\mathbf{1}X)\bigg(\frac{1\otimes g}{z-w}\bigg) = X\bigg( \int_z \frac{1\otimes g}{z-w}dz \bigg) = X(g),
    \]
    since, by Lemma \ref{intformula}, we have $\int_z \frac{1\otimes g}{z-w}dz = g$.
\end{proof}

Let $\g$ be a simple finite dimensional Lie algebra over $\C$ and let $\kappa$ be an invariant symmetric bilinear form of $\g$. Then $\kappa$ is a scalar multiple of the Killing form $\kappa_{\g}$: $\kappa = k\kappa_{\g}$, $k \in \C$. Now we extend scalars on $\g$ by tensoring with $A$, and extend the bracket and $\kappa$ by $A$-linearity. We can then apply the construction of Definition \ref{kacmoodygeneral}, considering $K = K_n$ with its residue map and obtain a Lie algebra over $A$:
\[
    \hat{\g}_{n,k} \stackrel{\text{def}}{=} \hat{\g}_{K_n,\kappa}.
\]

We denote by $\tilde{U}_k(\hat{\g}_n)$ the completed enveloping algebra appearing in Corollary \ref{envelopingcompletion}.

\begin{defi}
    Fix $X \in \g$ and define the following $K$-distribution on $\tilde{U}_k(\hat{\g}_n)$
    \[
        \hat{X} : K_n \to \tilde{U}_k(\hat{\g}_n), \qquad \hat{X}(f) = X\otimes f.
    \] 
    By definition of the topology on $\tilde{U}_k(\hat{\g}_n)$ this is easily checked to be a field.
\end{defi}

\begin{rmk}
    Let $X, Y \in \g$, then by definition of the Lie bracket in $\hat{\g}_n$ the following formula holds for the associated fields $\hat{X}$ and $\hat{Y}$.
    \[
        [\hat{X},\hat{Y}](f\otimes g) = \widehat{[X,Y]}(fg) + \kappa(X,Y)\int (g\partial_z f)dz.
    \]
    From Corollary \ref{proplocality} we obtain
    \[
        (z-w)^2[\hat{X},\hat{Y}] = 0.
    \]
    In addition we have the following equalities
    \[
        (\hat{X})_{(0)}(\hat{Y}) = \widehat{[X,Y]}, \qquad (\hat{X})_{(1)}(\hat{Y}) = \kappa(X,Y)\mathbf{1}.
    \]
\end{rmk}

\begin{defi}
    Let $\F$ be the $\C$-linear span of the $\hat{X}$ with $X\in \g$, and let $\overline{\F}$ be its closure by $n$-products and derivations. Note that $\mathbf{1} \in \overline{\mathcal{F}}$, hence $V^k_{K_n}(\g) \stackrel{\text{def}}{=} \overline{\mathcal{F}}$ has a natural structure of a vertex algebra with $\mathbf{1}$ as the vacuum vector and $\partial_z$ as the translation operator, by Theorem \ref{vertexalgebras}.
\end{defi}

Consider now the affine vertex algebra $V^k(\g)$ of level $k$.

\begin{prop}\label{propnewvertexisoldvertex}
    There exists an isomorphism of vertex algebras
    \[
        Y_{\text{can}} : V^k(\g) \stackrel{\sim}{\to} V_{K_n}^k(\g) 
    \]
    such that $Y_{\text{can}} (X_{-1}\vac) = \hat{X}$ for all $X$ in $\g$.
\end{prop}

\begin{proof}
    By the universal property satisfied by $V^k(\g)$ it is sufficient to check that the endomorphisms $(\hat{X})_{(n)} \in \End_{\C}\big(V^k_{K_n}(\g)\big)$ satisfy the relations of the affine Kac-Moody algebra of level $k$. Since $V_{K_n}^k(\g)$ is a vertex algebra we have
    \[
        [\hat{X}_{(n)},\hat{Y}_{(m)}] = \big( \hat{X}_{(0)}\hat{Y}\big)_{(n+m)} + n\big( \hat{X}_{(1)}\hat{Y}\big)_{(n+m-1)} = \widehat{[X,Y]}_{(n+m)} + n\kappa(X,Y)\mathbf{1}_{(n+m-1)}.
    \]
    Now by Lemma \ref{unity} we know that $\mathbf{1}_{(n+m-1)} = \delta_{n,-m}\id_{V_{K_n}^k(\g)}$ and the proposition follows. 
    Finally $Y_{\text{can}}$ is clearly surjective, while taking the quotient of the right hand side by $a_i = 0$ gives us back the realization of $V^k(\g)$ as fields on $\tilde{U}_k(\hat{\mathfrak{g}})$ so $Y_{\text{can}}$ is injective as well.
\end{proof}

The morphism $Y_{\text{can}}$ induces an action by $K_n$-fields of $V^k(\g)$ on $\tilde{U}_k(\hat{\g}_n)$ and therefore it induces a morphism of Lie algebras (c.f. Definition \ref{defliekv})
\[
    \Lie_{\text{can}} \stackrel{\text{def}}{=} \Lie(Y_{\text{can}}) : \Lie_{K_n}(V^k(\g)) \to \tilde{U}_k(\hat{\g}_n).
\]

\subsection{Action of derivations}

We are now going to study a natural action of the Lie algebra $\Der K_n$ by derivations on the Lie algebra $\Lie_{K_n}(V)$ for $V$ a quasi-conformal vertex algebra. 

\begin{defi}
    A \textbf{quasi-conformal} vertex algebras is a vertex algebra equipped with an action of the Lie algebra $\Der \C[[t]]$ (which is topologically generated by the elements of the form $L_n = -t^{n+1}\partial_t$) such that
    \begin{itemize}
        \item The operator $L_{-1}$ coincides with $T$;
        \item The action of $L_0 = -t\partial_t$ is semisimple with integer eigenvalues, and in addition each eigenspace is finite dimensional;
        \item The action is continuous, meaning that for any $a \in V$ we have $L_na = 0$ for sufficiently large $n$;
        \item The following formula holds
        \[
            [L_n,a_{(m)}] = \sum_{k\geq -1} \binom{n+1}{k+1}(L_{k}a)_{(n+m-k)}.
        \]
    \end{itemize}
\end{defi}

Recall that a vertex algebra is said to be \textbf{conformal} if there exists a conformal vector $\omega$ such that it generates an action of the Virasoro algebra of some level. With the hypothesis of semisemplicity of the action of $L_{0} = \omega_{(1)}$, any conformal vertex algebra becomes a quasi conformal vertex algebra where $L_n = \omega_{(n+1)}$.

The following proposition is well known and may be found in \cite[Section 6.1.4]{frenkel2007langlands}.
\begin{prop}
    The vertex algebra $V^k(\g)$ is quasi-conformal for any level $k$. In addition for $k\neq k_c$, the critical level, $V^k(\g)$ is conformal, and the quasi-conformal structure at the critical level can be obtained as a limit of the conformal structures away from the critical level.
\end{prop}

\begin{defi}
    Define the following action of $\Der K$ on $\Lie_K(V^k(\g))$
    \[
        (-f\partial_z) \cdot ag = \sum_{l\geq -1} \frac{1}{(l+1)!}(L_la)g\partial_z^{l+1}f.
    \]
    It is easy to check that this is indeed an action of $\Der K_n$ and that $-f\partial_z$ induces a derivation of $\Lie_{K_n}(V^k(\g))$.
\end{defi}

\begin{prop}\label{derequivariance}
    The canonical morphism
    \[
        \Lie_{\text{can}} : \Lie_{K_n}(V^k(\g)) \to \tilde{U}_k(\hat{\g}_n)
    \]
    is $\Der K_n$-equivariant.
\end{prop}

\begin{proof}
    For $k \neq k_c$ the action of $\Der K_n$ (on both spaces) is induced by the action of $(\omega_k) f \in \Lie_{K_n}\big(V^k_{K_n}(\g)\big)$, meaning that for all $x \in \Lie_{K_n}\big(V^k_{K_n}(\g)\big)$ we have $(-f\partial_z)\cdot x = [\omega_kf,x]$, where $\omega_k$ is the conformal vector of $V^k(\g)$. Therefore equivariance follows from the fact that the morphism is a Lie algebra morphism. For $k = k_c$ it suffice to notice that the action of $\Der K_n$ is obtained by taking the limit in both spaces and therefore equivariance follows from the case $k \neq k_c$. A precise proof of this limiting process can be stated using arguments similar to the ones in \cite[Prop. 3.4.3 and Prop. 7.4.1]{casarin2021description}. One can define a vertex algebra over $\C[\mathbf{k}]$ ( $\mathbf{k}$ is now a variable) called $V^{\mathbf{k}}(\g)$ whose specialization to $\mathbf{k} = k$ gives back the usual universal affine vertex algebra of level $k \in \C$. The construction of $\tilde{U}_k(\hat{\g}_n)$ can be done over $\C[\mathbf{k}]$ as well, the action of derivations on both spaces and the map in the statement of the proposition can be made $\C[\mathbf{k}]$-linear. When checking equivariance one has to deal with polynomial equations in $\mathbf{k}$ which, by the discussion above, are $0$ whenever we specialize $\mathbf{k} = k \neq k_c$. This then easily implies that these polynomial equalities vanish also for $k = k_c$.
\end{proof}

\subsection{The complete topological algebra \texorpdfstring{$\tilde{U}_{K_n}(V)$}{U(V) in the case of Kn}}

We are going to construct here a functor $V \rightsquigarrow \tilde{U}_{K_n}(V)$ which sends a vertex algebra $V$ to a complete topological $A$-algebra. The latter will be equipped with a morphism of Lie algebras $\Lie_{K_n}(V) \to\tilde{U}_{K_n}(V) $. In the particular case where $V = V^k(\g)$ we prove that there exists a unique continuous morphism 
\[
    \tilde{U}_{K_n}(V^k(\g)) \to \tilde{U}_k(\hat{\g}_n)
\]
extending the map $\Lie_{K_n}(V^k(\g)) \to \tilde{U}_k(\hat{\g}_n)$.

All vertex algebras that we consider are over $\C$, quasi-conformal and finitely generated by homogeneous (with respect to the action of $L_0$) elements. Let $a^i$ be a set of such generators, then $\deg( a^i_{(-n-1)}) = \deg(a^i) + n$ and every homogeneous component $V_m$ of $V$ is spanned by elements of the form
\[
    a^{i_1}_{(-n_1-1)}\dots a^{i_l}_{(-n_l-1)}\vac \quad \text{with} \quad \sum_{j=1}^l \deg(a^{i_j}) + n_j = m.
\]

We construct $\tilde{U}_{K_n}(V)$ as follows. Let $U^0_{K_n}(V)$ be the universal enveloping algebra over $A$ of $\Lie_{K_n}(V)$ divided by the two sided ideal generated by $(\int (gdz)-\vac g)$ as $g$ varies in $K_{n}$. \newline Now consider its left ideals $\widetilde{I_N}$ which are generated by elements of the form $ag$ with $a\in V_m$ and $g \in R\varphi^n$ such that $n\geq Nm$ and define $U'_{K_n}(V)$ to be the completion
\[
    U'_{K_n}(V) = \varprojlim_N \frac{U^0_{K_n}(V)}{\widetilde{I_N}}.
\]
It can be checked that the product in $U^0_{K_n}(V)$ is continuous with this topology, therefore $U'_{K_n}(V)$ is a complete topological $A$-algebra. 

Given $a \in V$ we consider the $K_n$-distribution $\tilde{a}$ with values in $U'_{K_n}(V)$ defined by setting $\tilde{a}(g)$ to be the image of $ag$ via the map $\Lie_{K_n}(V) \to U^0_{K_n}(V) \to U'_{K_n}(V)$.

By definition of the topology on $U'_{K_n}(V)$ these are actually fields. In addition it can be easily checked that, given $a,b \in V$, if $a_{(k)}b = 0$ for all $k\geq N$ then $(z-w)^N[\tilde{a},\tilde{b}] = 0$, so they are mutually local. We define $\tilde{U}_{K_n}(V)$ to be the quotient of $U'_{K_n}(V)$ by $J$, the closure of the two sided ideal generated by the elements
\[
    J = \big( (a_{(-1)}b)g - \tilde{a}_{(-1)}\tilde{b}(g)\big), \qquad a,b\in V.
\]

\begin{lemma}\label{quantifycontinuity}
    Let $\hat{X}$ be the $K_n$-field associated to $X\in\g$ and $Y$ be an arbitrary field. Let $I_l = R_n\varphi_n^l$ the $A_n$-submodules of $K_n$ defining its topology and $U_N = \tilde{U}_k(\hat{\g}_n)(\g\otimes R_n\varphi_n^N)$ be the fundamental system of open neighborhoods for $\tilde{U}_k(\hat{\g}_n)$. Suppose that
    \[
        Y(I_M) \subset U_N.
    \]
    Then 
    \[
        (\hat{X}_{(-k-1)}Y)(I_{N+M+k}) \subset U_N.
    \]
    In particular let $X^i$ be the fields associated to elements $X^i \in \g$, then
    \[
        \big(\hat{X}^1_{(-n_1-1)}\dots \hat{X}^r_{(-n_r-1)}\mathbf{1}\big)\big( I_{rN + \sum_i n_i}\big) \subset U_N,
    \]
    where the products are done from right to left.
\end{lemma}

\begin{proof}
    This is a straightforward computation using a completely analogous argument to Lemma \ref{continuitynproduct}.
\end{proof}

\begin{prop}\label{envelopinkalgebra}
    Consider the canonical action by $K_n$-fields of the vertex algebra $V^k(\g)$ on the complete topological $A_n$-algebra $\tilde{U}_k(\hat{\g}_n)$. Then there exists a unique continuous morphism of complete topological $A_n$-algebras 
    \begin{equation}\label{eqmaputilde}
    \tilde{U}_{K_n}(V^k(\g)) \to \tilde{U}_k(\hat{\g}_n)
    \end{equation}
    making the following diagram commute.
    \[\begin{tikzcd}
	{\Lie_{K_n}(V^k(\g))} && \tilde{U}_k(\hat{\g}_n) \\
	\\
	& {\tilde{U}_{K_n}(V^k(\g))}
	\arrow[from=1-1, to=1-3]
	\arrow[from=1-1, to=3-2]
	\arrow["{\exists!}"', dashed, from=3-2, to=1-3]
\end{tikzcd}\]
\end{prop}

\begin{proof}
    Let $U = \tilde{U}_k(\hat{\g}_n)$ and $V = V^k(\g)$. Recall that $U^0_{K_n}(V)$ is the enveloping $A_n$-algebra of $\Lie_{K_n}(V)$ quotiented by the ideal generated by the elements $\int gdz - \vac g$. Since the map $\Lie_{\text{can}}$ be definition sends $\vac g \mapsto \mathbf{1}(g) = \int gdz$ we get a unique morphism $U^0_{K_n}(V) \to U$. To check continuity one may use Lemma \ref{quantifycontinuity} to show that $\widetilde{I_{N+1}} \mapsto U_N$. Therefore there is a unique extension to a morphism $U'_{K_n}(V) \to U$. Finally the fact that the ideal $J$ is sent to $0$ in $U$ follows from the fact that $V$ is acting on $U$ by $K_n$-fields and therefore the equality $(a_{(-1)}b)g = (\tilde{a}_{(-1)}\tilde{b})(g)$ in $\tilde{U}_k(\hat{\g}_n)$ follows by construction.
\end{proof}

Finally we prove that the map above is actually an isomorphism.

\begin{prop}\label{propdescrenvelopkalg}
    The complete associative algebra $\tilde{U}_{K_n}(V^k(\g))$ is canonically isomorphic to the algebra $\tilde{U}_k(\hat{\g}_{K_n})$ via the morphism in \eqref{eqmaputilde}. In the simpliest case where $\g$ is abelian and $\kappa = 0$ it is easily seen that
    \[
        \tilde{U}_K(V^\kappa(\g)) = \widetilde{\Sym}_A(\g\otimes_{\C} K).
    \]
\end{prop}

\begin{proof}
    The canonical action by $K_n$-fields of $V^k(\g)$ on $\tilde{U}_k(\hat{\g}_n)$ induces, by Proposition \ref{envelopinkalgebra}, a continuous morphism $\tilde{U}_{K_n}(V^k(\g)) \to \tilde{U}_k(\hat{\g}_n)$. Vice versa we have a natural continuous map $\hat{\g}_{K_n,k} \to \tilde{U}_{K_n}(V^k(\g))$, mapping $X\otimes f$ to the image of $(X_{(-1)}\vac)f$ via the map $\Lie_{K_n}(V^k(\g)) \to \tilde{U}_{K_n}(V^k(\g))$ and $\mathbf{1} \mapsto 1$. This map induces a morphism $\tilde{U}_k(\hat{\g}_n) \to \tilde{U}_{K_n}(V^k(\g))$. The latter is easily seen to be the inverse of the morphism  \eqref{eqmaputilde}, since their composition are the identity on a set of generators.
\end{proof}

\begin{rmk}\label{rmkgradedversion}
    When the Lie algebra $\g$ is already equipped with a $\Z$-grading, there is a slightly different version of this construction. The vertex algebra $V^\kappa(\g)$ obtains a shifted grading if we take into account the original grading on $\g$, the topology we define on $\tilde{U}_{K_n}(V^\kappa(\g))$ is therefore slightly different. We obtain anyway a similar statement. In particular, in the abelian case ($\g$ abelian and $\kappa =0$), if $x_i$ is an homogeneous basis of $\g$ then
    \[
        \tilde{U}_{K_n}(V^\kappa(\g)) = \varprojlim_N \frac{A\big[x_i\epsilon_{j,k_i}\big]}{\big( x_i\epsilon_{j,k_i} : k_i \geq N\deg x_i \big)}.
    \]
\end{rmk}

\begin{corollary}[Of Proposition \ref{derequivariance}]\label{equivariance}
    The isomorphism
    \[
        \tilde{U}_{K_n}(V^\kappa(\g)) \to \tilde{U}_{\kappa}(\hat{\g}_n) 
    \]
    is $\Der K_n$-equivariant, where the action by derivations of $\Der K_n$ on $\tilde{U}_{K_n}(V^k(\g))$ is induced by the action on  $\Lie_{K_n}(V^k(\g))$.
\end{corollary}

\subsection{Central elements}

We are now able to construct, for $k = k_c$, a lot of central elements in the completed enveloping algebra $\tilde{U}_k(\hat{\g}_n)$, namely the ones that come from the center of $V^k(\g)$. \newline
From now on we will work with a fixed simple Lie algebra $\g$ and at the critical level $k = k_c$ only.

\begin{prop}\label{centralelements}
    Let $\zeta(\g)$ be the center of $V^{k_c}(\g)$ at the critical level. By the Feigin-Frenkel theorem \cite{feigin1992affine}, this is isomorphic to the algebra of functions on the space of Opers on the formal disc $\zeta(\g)  = \C[Op_{^L\g}(D)]$. Then for any $z \in \zeta(\g)$ and any $f \in K_n$ the element defined by $Y_{\text{can}}(z)(f) \in \tilde{U}_k(\hat{\g}_n)$ is central:
    \[
        Y_{\text{can}}(z)(f) \in Z(\tilde{U}_{k_c}(\hat{\g}_n)).
    \]
\end{prop}

\begin{proof}
    This is quite easy to see. Let $ x = X_{(-1)}\vac \in V^k(\g)$ for $X\in \g$ and notice that for any $n \geq 0$ we have $z_{(n)}x = 0$ by assumption. In particular the bracket $[zf,xg] = 0$ in $\Lie_{K_n}(V^k(\g))$. Now, since $\Lie_{\text{can}}$ is an homomorphism of Lie algebras and since $\Lie_{\text{can}}(zf) = Y_{\text{can}}(z)(f)$ and $\Lie_{\text{can}}(xg) = \hat{X}(g)$ we obtain that $Y_{\text{can}}(z)(f)$ commutes with every element of the form $Xg$ for $X\in \g$. These topologically generate the completed enveloping algebra and therefore $Y_{\text{can}}(z)(f)$ must be central.
\end{proof}

\begin{rmk}
    For $\g = \mathfrak{sl}_2$ and with $n = 2$ singularities the construction above gives back the central elements in the enveloping algebra constructed in \cite{fortuna2020local}.
\end{rmk}

By Proposition \ref{envelopinkalgebra} and Corollary \ref{equivariance} there is a natural, $\Der K_n$-equivariant morphism

\[
    \tilde{U}_{K_n}(\C[Op_{^LG}(D)]) \to \tilde{U}_{\kappa_c}(\hat{\g}_n),
\]
which, by the previous proposition factors through the center. So there is a natural map
\[
    \tilde{U}_{K_n}(\C[Op_{^LG}(D)]) \to Z(\tilde{U}_{\kappa_c}(\hat{\g}_n)).
\]
In the following sections we are going to study the characterization of $Z(\tilde{U}_{\kappa_c}(\hat{\g}_n))$ as the ring of functions on $Op_{^L\g}(D_n)$ in two steps. First, we are going to prove that the map we constructed above is an isomorphism, obtaining an algebraic description of the center. Second we are going to prove that the algebra $\tilde{U}_{K_n}(\C[Op_{^LG}(D)])$ is canonically isomorphic to $A[Op_{^L\g}(D_n)]$, the algebra of functions on the space of Opers over the $n$-pointed disk $D_n = \spec K_n$:

\[
    A[Op_{^LG}(D_n)] \simeq \tilde{U}_{K_n}(\C[Op_{^LG}(D)]) \simeq  Z(\tilde{U}_{k_c}(\hat{\g}_n)).
\]


\section{Algebraic description of the center}

In this section we are going to prove that the map
\[
    \tilde{U}_{K_n}(\C[Op_{^LG}(D)]) \to Z(\tilde{U}_{\kappa_c}(\hat{\g}_n)) 
\]
is an isomorphism. In order to do that we will study particular graded spaces associated to the center, in complete analogy to with theorem 3.7.7 of \cite{beilinson1991quantization}.

We will need to introduce $R_{n}$-Jet schemes which are a straightforward generalization of the usual notion of Jet schemes.

\subsection{\texorpdfstring{$R$}{R}-Jet Schemes}

We are going to define here analogues of the classical Jet schemes. Recall that given a scheme $X$ of finite type over $\C$ one may consider the jet scheme and the $n$-th jet schemes associated to $X$, which are described as functors of $\C$-algebras:

\[
    JX(B) = X(B[[t]]) \quad \text{ and } \quad J_nX(B) = X(B[t]/(t^n))
\]
for $B \in \Alg_{\C}$. We are going to generalize this construction in the setting of our rings of interest $R_n,K_n$.\newline Many results of this section can be generalized to an arbitrary $(\varphi)$-adic ring $R$ and its localization $K$, we chose to focus on $R_n$ and $K_n$ to be able to use the residue freely.

To make the notation a little bit simpler we will often write $R$ for $R_n$, $K$ for $K_n$, $\varphi$ for $\varphi_n$, $I$ for the ideal $(\varphi)$ and $A$ for $A_n$.

\begin{defi}
    Let $X$ be a scheme over $A$, or a functor of $A$-algebras. We define the functors
    \[
        J^KX(B) \stackrel{\text{def}}{=} X(B(K)), \qquad J^{R}X(B) \stackrel{\text{def}}{=} X(B(R)), \qquad J^{R}_{l}X(B) \stackrel{\text{def}}{=} X(B\otimes (R/I_l)),
    \]
    for $B \in \Alg_A$. Here $I_l$ is the ideal generated by $\varphi^l$ in $R$. We call them the $K$-jet scheme of $X$, the $R$-Jet scheme of $X$ and $l$-th $R$-Jet scheme of $X$.
\end{defi}

\begin{rmk}
    When $A = \C$ and $R = \C[[t]]$ with the standard topology we recover the classical definition of jet schemes.
\end{rmk}

The constructions $J^K,J^R,J^R_l$ are clearly functorial and in addition there are natural maps 
\[
    \pi_n : J^RX \to J^R_nX, \quad \pi_{m,n} : J^R_nX \to J^R_mX, \qquad J^RX \to J^KX
\]
such that $\pi_m = \pi_{m,n}\circ\pi_n$ and $\pi_{n,m}\circ\pi_{m,l} = \pi_{n,l}$. These are functorial in $X$ as well.
We are going to show that if $X$ is affine and of finite type over $A$ then the functor $J^RX$ is representable and affine, while the functors $J^R_nX$ are representable by affine schemes of finite type over $A$.

Recall that in section \ref{topbasis} we introduced a topological basis $e_{i,l}$ with $l\in\Z$ (resp. $l\geq 0$) for $K_n$ (resp. for $R_n$), and the dual basis $\epsilon_{j,k}$ with respect to the residue pairing. Every element $g \in B(R)$ can be expressed uniquely as 
\[
    g = \sum_{i,m\geq 0} b_{i,-m-1}e_{i,m} \quad \text{for some } \quad b_{i,-m-1} \in B.
\]

In addition, every element of $B\otimes (R/I_n)$ can be written in a unique way as a finite sum of the kind above where $m\leq n$.

We start by studying the $R$-Jet schemes of $X =\A^N_A$. By definition the $B$ points of $J^RX$ are given by
\[
    J^R\A^N_A(B) = \{ (b^1,\dots,b^N) \text{ with } b^i \in B(R) \} \subset \{ (b^1,\dots,b^N) \text{ with } b^i \in B(K) = J^K\A^N_A(B).
\]
Consider the functions $x^j_{i,n}$, as $j \in \{ 1,\dots,N\}$ defined on $B$ points as follows
\[
    x^j_{i,m} : J^K\A^N_A \to \A^1_A, \qquad x^j_{i,m}(b^1,\dots,b^N) \stackrel{\text{def}}{=} \int b^j\epsilon_{i,m}dz.
\]
We consider their restriction to $J^RX$ and if we write
\[
    b^j = \sum_{m\geq 0} b^j_{i,-m-1}e_{i,m}
\]
we have $x^j_{i,n}(b^1,\dots,b^n) = b^j_{i,n}$ for $n < 0$ and $x^j_{i,n}(b^1,\dots,b^n) = 0$ for $n\geq 0$. These functions may be defined analogously for $J^R_l\A^N_A$. From this construction the following proposition follows.

\begin{prop}
    Let $X = \A^N_A$, then the functors $J^RX$ and $J^R_lX$ are representable by affine schemes and $J^R_lX$ is of finite type over $A$. In particular
    \[
        J^R\A^N_A  = \spec A\big[x^j_{i,m}\big]_{m < 0}, \qquad J^R_l\A^N_A  = \spec A\big[x^j_{i,m}\big]_{-l \leq m < 0}.
    \]
    In addition the projection maps $\pi_l : J^R\A^N_A \to J^R_l\A^N_A$ are given by the natural inclusions
    \[
        \pi_l^* : A\big[x^j_{i,m}\big]_{-l \leq m < 0} \hookrightarrow A\big[x^j_{i,m}\big]_{m < 0}.
    \]
\end{prop}

We may perform a similar description of the functor $J^K\A^N_A$. Every element of $b \in B(K)$ may be written uniquely as an infinite sum
\[
    b = \sum_{m\geq -N} b_{i,-m-1}e_{i,m} \quad \text{with} \quad b_{i,-m-1} \in B.
\]
and every infinite sum of this form defines an element of $B(K)$.
Notice that we may therefore define analogous functions
\[
    x^j_{i,k} : J^K\A^n_A \to \A^1_A, \qquad x^j_{i,k}(b^1,\dots,b^N) = b^j_{i,k},
\]
for any $i$ and for any integer $k \in \Z$.

\begin{defi}
    Let $X = \A^N_A$ and $n$ a positive integer define the functors
    \[
        J^K_{\geq n}\A^N_A(B) \stackrel{\text{def}}{=} \big\{(b^1,\dots,b^N) \in J^K\A^N_A(B) : \varphi^nb^i \in B(R) \big\}.
    \]
    Here $\varphi$ is as always the function in $K$ which defines the topology. In addition there are natural maps for $n \geq m$
    \[
        \iota_{m,n} : J^K_{\geq m}\A^N_A \to J^K_{\geq n}\A^N_A.
    \]
\end{defi}

\begin{lemma}
    The functors $J^K_{\geq n}\A^N_A$ are representable and isomorphic to $J^R\A^N_A$ via the multiplication
    \[
        \cdot\, \varphi^n : J^K_{\geq n}\A^N_A \to J^R\A^N_A,\qquad (b^1,\dots,b^N) \mapsto (\varphi^nb^1,\dots,\varphi^nb^N).
    \]
    In particular
    \[
        J^K_{\geq n}\A^N_A = \spec A[x^j_{i,k}]_{k<n}.
    \]
\end{lemma}

\begin{proof}
    Obvious.
\end{proof}

\begin{corollary}
    The natural maps $\iota_{m,n} : J^K_{\geq m}\A^N_A \to J^K_{\geq n}\A^N_A$ are closed immersions and 
    \[
        J^K\A^N_A = \bigcup_{n\geq 0} J^K_{\geq n}\A^N_A.
    \]
    It follows that $J^K\A^N_A$ is an ind-scheme and therefore
    \[
        A[J^K\A^N_A] = \varprojlim_n A\big[J^K_{\geq n}\A^N_A\big] =\varprojlim_n \frac{A[x^j_{i,k}]_{k\in\Z}}{(x^j_{i,k} : k\geq n)}.
    \]
\end{corollary}

\subsection{Convolution and normally ordered product}

We now proceed in studying the case where $Z \hookrightarrow \A^N_A$ is a closed subscheme. To study this, we will study the behaviour of functions $p : \A^N_A \to A^1_A$ under the application of the functor $J^R$.

Notice that we may organize the functions $x^j_{i,m}$ on $J^K\A^N_A$ in $K$-distributions with values in $A\big[J^K\A^N_A\big]$ considered as a topological $A$-algebra with the topology induced by the limit topology (where the quotients are discrete). Indeed consider the distributions
\[
    x^j_K : K \to A\big[J^K\A^N_A\big], \qquad x^j_K(g)(b^1,\dots,b^N) = \int b^jgdz.
\]

It is pretty clear that $x^j_K(\epsilon_{i,m}) = x^j_{i,m}$ for all $m$ the continuity easily follows as well. In particular we may consider normally ordered products of the $x^j_K$ and evaluate them on any $g \in K$. 

In addition since $A\big[J^K\A^N_A\big]$ is commutative these fields are all mutually local and their $(-1)$-product does not depend on the order of the factors, by Lemma \ref{normalordercommutative}. It follows that for any polynomial $p \in A[x^j]$ there is a canonical field $p_K$ associated to it.

\begin{defi}\label{deffunctiondistribution}
    We define $p_K$ the $A\big[J^K\A^N_A\big]$ valued $K$ field associated to $p$ as follows. Consider $p$ as an $A$ linear combination of products of the $x^j$ and the constant function $1$, then make the substitution $x^j \mapsto x^j_K$, $1 \mapsto \mathbf{1}$ and replace the ordinary products with the $(-1)$-product of fields. By the remarks above these product do not depend on the order of the $x^j$ so $p_K$ is well defined.
    
    Given $p\in A[x^j]$ we are going to write $p_{j,k} = p_K(\epsilon_{j,k})$.
\end{defi}

Post composing the fields $x^j$ with the natural projections $A[J^K\A^N_A] \to A[J^K_{\geq n}\A^N_A]$ we obtain $A[J^K_{\geq n}\A^N_A]$-valued fields where the latter algebra is considered with the discrete topology. Notice in addition that this construction yields a map of $K[\partial_z]$-modules 
\[
    F_A(K,A[J^K\A^N_A]) \to F_A(K,A[J^K_{\geq n}\A^N_A])
\]
which commutes with $n$-products.

\medskip
We begin studying the fields $p_K$ postcomposing them with the projecting them onto $A[J^R\A^N_A]$, we call this composition $p_R$. It turns out that there is another way to describe the fields $p_R$.

Consider $p$ as a map $\A^N_A \xrightarrow{p} \A^1_A$ and consider the associated morphism, obtained by functoriality
\[
    J^Rp : J^R\A^N_A \to J^R\A^1_A.
\]
For $g \in K$ denote by $J^Rp(g)$ its composition with the map
\[
    \int (\,\cdot\, gdz) : J^R\A^1_A \to \A^1_A \qquad B(R) \ni b \mapsto \int bgdz \in B
\]

\begin{prop}\label{propfunctiondistribution}
    For any $g \in K$ the following equality holds
    \[
        J^Rp(g) = p_R(g).
    \]
\end{prop}

\begin{proof}
    We may assume that $p$ is a monomial, by linearity. We proceed by induction on the degree of $p$, the result being true by definition for $\deg p = 1$. So assume that the statement is true for $\deg p \leq n$ and lets prove it for $\deg p  = n+1$, so that we may write $p = x^jq$ for some monomial $q$ for which we know that
    \[
        q_R(g)\big( b^1,\dots, b^N \big) = \int q(b^1,\dots,b^N)gdz.
    \]
    We make an additional simplification: notice that for any $N$-tuple $(b^1,\dots,b^N)$ of elements in $B(R)$ there is a natural evaluation map
    \[
        ev_{(b^1,\dots,b^N)} : A\big[ J^R\A^N_A \big] \to B.
    \]
    Any two elements of $A\big[ J^R\A^N_A \big]$ which coincide on all evaluations are equal. Therefore we may work over a fixed $N$-tuple of $B(R)$ elements. To keep the notation simple we will be denoting the composition of $x^j_K$ and $q_K$ with the evaluation map $ev_{(b^1,\dots,b^N)}$ by the same symbols $x^j_K$ and $q_K$. We want to prove that
    \[
        p_R(g) = \big( (x^j_R)_{(-1)}q_R \big)(g) = \int b^jq(b^1,\dots,b^N)gdz = J^Rp(g).
    \]
    By the induction hypothesis we have that
    \[
        q_R(g) = \int q(b^1,\dots,b^N)gdz,
    \]
    in addition we may extend $x^j_R,q_R$ to $B\otimes_A K$ by $B$ linearity and furthermore extend it to $B(K)$ by continuity, we denote by $\tilde{x}^j_R,\tilde{q}_R$ these extensions. We are going to prove the equivalent statement
    \[
        \big((\tilde{x}^j_R)_{(-1)}\tilde{q}_R\big)(g) = \int b^jq(b^1,\dots,b^N)gdz.
    \]
    The previous formula for $q_R$ easily implies the following:
    \[
        \forall \tilde{b} \in B(K) \qquad \tilde{q}_R(\tilde{b}) = \int q(b^1,\dots,b^N)\tilde{b}dz.
    \]
    Now consider the commutative diagram
    \[\begin{tikzcd}
	{B(K)\otimes_BB(K)} && {B(K)} \\
	\\
	B
	\arrow["{\tilde{x}^j_R\tilde{q}_R}"', from=1-1, to=3-1]
	\arrow["{\int b^j\otimes 1\cdot\, dz}", from=1-1, to=1-3]
	\arrow["{\tilde{q}_R}", from=1-3, to=3-1]
\end{tikzcd}\]
    where the upper arrow is described by
    \[
        \alpha\otimes\beta \mapsto \int_z (b^j\alpha)\otimes\beta dz.
    \]
    This naturally extend to the completion $\widehat{B(K)\otimes_B B(K)}^2$ and the commutativity of this diagram proves that
    \[
        \big((\tilde{x}^j_R)_{(-1)}\tilde{q}_R\big)(g) = \tilde{q}_R\bigg( \int\frac{b^j\otimes g}{z-w}dz \bigg) = \tilde{q}_R(b^jg) = \int q(b^1,\dots,b^N)b^jgdz.
    \]
    where the second equality follows from Corollary \ref{intbformula} since $b^j \in B(R)$ and the third equality follows from the previous description of $\tilde{q}_R$.
\end{proof}

We are now ready to show that if $X$ is affine and of finite type over $A$ the functor $J^RX$ is representable by a scheme over $A$, the proof that $J^R_lX$ is representable and of finite type is completely analogous and we omit it. Write $X = \spec A[x^j]/(g^l)$ for some $\{g^l\}_l \in A[x^j]$ ($A$ is noetherian). And let
\[
    j : X \hookrightarrow \A^N_A
\]
be the corresponding closed immersion. It is easy to see that the $B$ points of $J^RX$ correspond to $N$-tuples $(b^1,\dots,b^N) \in B(R)^N$ such that $g_l(b^1,\dots,b^N)=0$ for any $g_l$. The latter is an element of $B(R)$. By the description we gave of $B(R)$ it easily follows that 
\[
    g^l(b^1,\dots,b^N) = 0 \iff \int g^l(b^1,\dots,b^N)\epsilon_{i,k}dz = 0 \quad\forall j,\forall k< 0.
\]

But the right hand side is just the function $g^l_{i,k}$ we defined in Definition \ref{deffunctiondistribution} evaluated on $(b^1,\dots,b^N)$. This implies the following proposition.

\begin{prop}\label{representability}
    Let $X$ be an affine scheme of finite type over $A$, write $X = \spec A[x^j]/(g^l)$. Then the functor $J^RX$ (resp. $J^R_lX$) is representable by a scheme (resp. by a scheme of finite type over $A$) and the natural map
    \[
        J^Rj : J^RX \to J^R\A^N_A \qquad \big( \text{resp. } J^R_lj : J^R_lX \to J^R_l\A^N_A \big)
    \]
    is a closed immersion. In particular $J^RX$ may be described as follows
    \[
        J^RX = \spec \frac{A[x^j_{i,k}]_{j,k<0}}{(g^l_{i,k} : k<0)}.
    \]
    From this description it follows that for $X$ an affine scheme of finite type over $A$ we have
    \[
        A[J^RX] = \bigcup_{l \geq 0} \text{Im}(A[J^R_lX] \to A[J^RX])
    \]
\end{prop}

Now that we proved representability we may state the following corollary, which is very useful.

\begin{corollary}\label{corosmoothness}
    Let $X$ be an affine scheme of finite type over $A$. If $X$ is smooth then for any positive integer $l$ the scheme $J^R_lX$ is smooth over $A$, in particular $A[J^R_lX]$ is flat over $A$.
\end{corollary}

\begin{proof}
    Since $X$ is smooth over $A$ it is also formally smooth. This easily imply that the functors $J^R_lX$ are all formally smooth as well. Since they are all representable by schemes of finite type over $A$ it follows that they are also smooth.
\end{proof}

We are going to need one last lemma, which describes how homogeneous polynomials behave under the isomorphism $\cdot\, \varphi^n : J^K_{\geq n}\A^N_A \to J^R\A^N_A$.

\begin{lemma}\label{lemmahompoly}
    Let $p \in A[\A^N_A]$ be an homogeneous polynomial of degree $d$. Then under the isomorphism
    \[
        \big( \cdot\, \varphi^n)^* : A[J^R\A^N_A] \to A[J^K_{\geq n}\A^N_A]
    \]
    the function $p_{j,k}$ maps to the function $p_{j,k+dn}$.
\end{lemma}

\begin{proof}
    This follows from Lemma \ref{lemmalinearitynormalorder}. We prove it by induction on the degree $d$ of $p$  the $A[J^R\A^N_A]$ valued field $p_R$ is sent to the $A\big[J^K_{\geq n}\A^N_A\big]$ valued field $\varphi^{dn}p_K$, which an equivalent statement to the lemma. The assertion is pretty clear for $d = 1$. To show that for $d+1$, knowing the assertion holds for $d$ it is enough to notice that since the $(-1)$-product of commutative fields is commutative, we know by Lemma \ref{lemmalinearitynormalorder} that it is $K$-bilinear.
\end{proof}

\subsection{Open immersions}

We study in this section the behaviour of the functor $J^R$ under open immersions. We first study some natural projection maps from the jet schemes to the base scheme $X$.

\begin{defi}
    There are natural projection maps
    \[
        \ev_i : J^RX \to X, \qquad \ev_i :J^R_{l}X \to X,
    \]
    defined functorially as follows. For any $A$ algebra $B$ consider the $B$-linear maps of rings
    \[
        \widetilde{\ev_i} : B(R) \to B \quad \text{and} \quad B(R)/(\varphi^l), \to B \qquad z \mapsto a_i.
    \]
    We define $\ev_i$ to be the map $J^RX \to X$ (resp. $J^R_{l}X \to X$) obtained by functoriality from $\widetilde{\ev_i}$.
\end{defi}

The projections $\ev_i$ should be thought as setting $z = a_i$. And are analogues of the classical projection $JX \to X$. Notice that they fit in natural commutative diagrams
\[\begin{tikzcd}
	{J^RX} \\
	&& X \\
	{J^R_{l}X}
	\arrow["{\pi_l}"', from=1-1, to=3-1]
	\arrow["{\ev_i}"', from=3-1, to=2-3]
	\arrow["{\ev_i}", from=1-1, to=2-3]
\end{tikzcd}\]

\begin{lemma}\label{lemmaopenimmersions}
    The following hold:
    \begin{enumerate}
        \item  If $U$ is an open subscheme of $X$ then 
        \[
            J^R_{\leq n}U = \bigcap_i \ev_i^{-1}(U).
        \]
        In particular, if $J^R_{l}X$ is representable, then $J^R_{l}U$ is representable as well and it is an open subscheme of $J^R_{l}X$.
        \item If $X$ is an affine scheme and $U = D(g)$ is a principal open subset, then
        \[
            J^RU = \pi_1^{-1}J^R_{\leq 1}U = \bigcap_{i} \ev_i^{-1}(U).
        \]
    \end{enumerate}
\end{lemma}

\begin{proof} We start by proving $(1)$. Since $U \to X$ is an open immersion we may regard $U$ as a subfunctor of $X$, hence we may regard $J^RU$ as a subfunctor of $J^RX$. It is clear that the immersion of functors $J^R_{\leq 1}U \to J^R_{\leq}X$ factors through 
\[
    J^R_{l}U \to \bigcap_i \ev_i^{-1}(U).
\]
Conversely consider a $B$-valued point of $\bigcap_i \ev_i^{-1}(U)$. This consists of a map $\spec B(R)/(\varphi^l) \to X$ such that for each $i$ there exists a commutative diagram
    
    \[\begin{tikzcd}
	{\spec B(R)/(\varphi^l)} && X \\
	\\
	{\spec(B)} && U
	\arrow[from=1-1, to=1-3]
	\arrow[hook, from=3-3, to=1-3]
	\arrow[from=3-1, to=3-3]
	\arrow["{\widetilde{\ev_i}^*}", from=3-1, to=1-1]
\end{tikzcd}\]
    And we need to prove that there exists a factorization
    \[\begin{tikzcd}
	{\spec B(R)/(\varphi^l)} && X \\
	\\
	&& U
	\arrow[from=1-1, to=1-3]
	\arrow[hook, from=3-3, to=1-3]
	\arrow[dashed, from=1-1, to=3-3]
\end{tikzcd}\]
For this, it is sufficient to prove that every prime ideal of $B(R)/(\varphi^l)$ has image in $U$. This easily follows if we show that the maps $\widetilde{\ev_i}^* : \spec B \to \spec B(R)/(\varphi^l)$ are jointly surjective. We reduce ourselves to prove that for any prime ideal $\mathfrak{p} \in B(R)/(\varphi^l)$ there exists an index $i$ such that $z-a_i \in \mathfrak{p}$. This is obvious though since $\prod (z-a_i) = \varphi \in \mathfrak{p}$ for any prime $\mathfrak{p}$, being nilpotent.
    
    \smallskip
    Now we prove point 2. Let $X = \spec S$ be an affine scheme and $U = D(g)$ the principal open subset defined by some $g \in S$. Then the functor of points of $U$ may be described as follows
    \[
        U(B) = \{ \chi : S \to B \; : \chi(g) \in B^* \} \subset X(B).
    \]
    Now recall that by standard results on completions an element $x \in B(R)$ is invertible if and only if its image in $B(R)/(\varphi)$ is invertible as well. From this remark follows that
    \begin{align*}
        J^RU(B) = U(B(R)) &= \{ \chi : S \to B(R)\; : \chi(g) \in B(R)^* \} \\
        &= \{ \chi : S \to B(R)\; : \overline{\chi}(g) \in \big(B(R)/(\varphi)\big)^* \} = \pi_1^{-1}\big(J^R_{\leq 1}U)(B),
    \end{align*}
    where $\overline{\chi}$ is the composition of $\chi$ with the natural projection $B(R) \to B(R)/(\varphi)$. This concludes the proof.
\end{proof}

\subsection{Base change}

In this section we study the behaviour of the $R$-Jet schemes by base change through an arbitrary map $\spec k \to \spec A$.

\begin{rmk}
    Notice that given any field valued point $A \to k$ the elements $a_i$ are naturally partitioned by the equivalence relation $a_i \sim a_j$ induced by equality in $k$. The difference of any couple of elements which are not in the same equivalence class is of course invertible in $k$.
\end{rmk}

\begin{lemma}
    Let $A \to k$ as before, and suppose that $B$ is a $k$-algebra, denote by $J$ the partition of the $a_i$ introduced in the previous remark. Then
    \[
        B(R_n) \simeq \prod_{j \in J} B[[z_j]].
    \]
    Here $z_j$ denotes the element $z - a_i$ for an arbitrary $a_i$ in the $j$-th partition.
\end{lemma}

\begin{proof}
    Let $d_j$ be the number of $a_i$ in the $j$-th partition. It follows from the Chinese remainder theorem that
    \[
        \frac{B[z]}{\varphi_n^N} \simeq \prod_{j\in J} \frac{B[z]}{(z-a_j)^{Nd_j}},
    \]
    where $a_j$ is an arbitrary element in the $j$-th partition. To prove the lemma it suffice to take the limit on both sides.
\end{proof}

The preceding lemma has the following the following consequence.

\begin{prop}\label{propbasechange}
    Let $A \to k$ be an arbitrary map. $J$ the partition of the $a_i$ introduced before and $X$ an arbitrary scheme, or functor of $A$ algebras. Then the base change $(J^R)_k$ 
    \[\begin{tikzcd}
	{(J^RX)_k} && {J^RX} \\
	& \square \\
	{\spec k} && {\spec A}
	\arrow[from=1-1, to=3-1]
	\arrow[from=3-1, to=3-3]
	\arrow[from=1-3, to=3-3]
	\arrow[from=1-1, to=1-3]
\end{tikzcd}\]
    is canonically isomorphic to
    \[
        (J^RX)_k = \prod_{j \in J} (JX)_k
    \]
\end{prop}

\begin{proof}
    As a functor $(J^RX)_k$ is described as the restriction of $J^RX$ to the subcategory $\Alg_k \to \Alg_A$. It takes as input any $A$-algebra $B$ whose structure comes from a composition $A \to k \to B$. By the preceding proposition
    \[
        (J^RX)_k(B) = X(B(R)) = X\big( \prod_{j\in J} B[[z_j]]\big) = \prod_{j\in J} X\big(B[[z_j]]\big) = \prod_{j\in J} JX(B)
    \]
\end{proof}

\subsection{Invariants}

By functoriality of $J^R$ if we are given a group scheme $G$ over $A$, the $R$-Jet $J^RG$ is naturally a group functor. In particular if $G$ acts on $X$ there is a natural action of $J^RG$ on $J^RX$ and in particular we obtain an action of $J^RG(A)$, the $A$ points of $J^RG$ on the ring of functions $A[J^RX]$. Analogously if $\g$ is the Lie algebra of $G$ there is an induced action by derivations of $J^R\g(A)$ on $A[J^RX]$. 

We are going to need the following remark about these actions.

\begin{rmk}\label{rmkcontinuousaction}
    Let $G$ be a group scheme which acts on a scheme $X$, let $\mu : X \times G \to X$ the morphism defining this action. For any positive integer $n$ we have commutative diagrams
    \[\begin{tikzcd}
	{J^RX\times J^RG} && {J^RX} \\
	\\
	{J^R_lX\times J^R_lG} && {J^R_lX}
	\arrow[from=1-1, to=3-1]
	\arrow["{J^R_l\mu}", from=3-1, to=3-3]
	\arrow["{J^R\mu}", from=1-1, to=1-3]
	\arrow[from=1-3, to=3-3]
\end{tikzcd}\]
    From this diagram we deduce that the map
    \[
        A[J^R_lX] \to A[J^RX]
    \]
    is $J^RG$ equivariant, where $J^RG$ acts on $A[J^R_lX]$ via the projection $J^RG \to J^R_lG$. In particular this shows that the action of $J^RG$ on the image $\text{Im} (A[J^R_lX] \to A[J^RX])$ factors through an action of $J^R_lG$. An analogous statement follows from the action of the $A$ points $J^RG(A)$ and for the Lie algebra $J^R\g$, as well as for its $A$ points $J^R\g(A)$. In particular if $X$ is affine and of finite type over $A$, the action of $J^R\g(A)$ is continuous with respect with the natural topology, and $J^R\g(A)$ acts on $\text{Im}(A[J^R_lX] \to A[J^RX])$ through its quotient $J^R_l\g(A)$.
\end{rmk}

\medskip

We are going to apply the functor $J^R$ mainly to base changes through $\spec A \to \spec \C$ of schemes defined over $\C$. If $X$ is a scheme over $\C$, we denote by $J^RX = J^R(X_A)$ the $R$-Jet functor applied to base change of $X$.

We are interested in the following setting. Let $\g$ be a simple Lie algebra over $\C$, coming from a linear algebraic group $G$. There is a canonical action of $G$ on the dual $\g^*$: the coadjoint action. The following theorem is due to Kostant \cite{kostant1963lie}.

\begin{thm}
    There exist algebraically independent homogeneous polynomials $P^i \in \C[\g^*]$ with $i \in \{ 1\dots l \}$ where $l = \rank \g$ such that
    \[
        \C[\g^*]^G =\C[g^*]^{G(\C)} = \C[\g^*]^{\g(\C)} = \C[P^i].
    \]
\end{thm}

We denote by $\Pg = \spec \C[P_i]$ it is equipped with a natural $G$-invariant morphism $\pi : \g^* \to \Pg$, which, by a theorem of Kostant admits a section $\sigma : \Pg \to \g^*$.

It is well known that this result extend to Jet schemes. The following theorem is due A. Beilinson and V. Drinfeld \cite{beilinson1991quantization}.

\begin{thm}
    The algebra of invariant functions $\C[J\g^*]^{J\g(\C)}$ is equal to $\C[J\Pg]$ via the inclusion map $J\pi^* : \C[J\Pg] \to \C[J\g^*]$.
\end{thm}

The goal of this section is to extend the previous theorem in our case of $R$-Jet schemes. Consider the previous setting
\[
    \pi : \g^* \to \Pg,
\]
and consider the base change of this map with the ring $A$. 
\[
    \pi_A : \g^*_A \to \Pg_A.
\]
We may now apply the functor $J^R$ to obtain a $J^RG$ equivariant map
\[
    J^R\pi : J^R\g^* \to J^R\Pg.
\]
Here we omit the subscripts to make the notation lighter. Since $\pi$ has a section $\sigma$ it easily follows by functoriality that $J^R\pi$ admits a section $J^R\sigma$.

\begin{prop}\label{propinvariants}
    The map $J^R\pi^* : A[J^R\Pg] \to A[J^R\g^*]$ is injective and induces an isomorphism
    \[
        A[J^R\Pg] = A[J^R\g^*]^{J^R\g(A)}.
    \]
\end{prop}

To prove this we will need first various lemmas regarding finite jet schemes.

\begin{lemma}\label{lemmagroupinvariants}
    Let $l$ be a positive integer. Consider the action of $J^R_{l}G$ on $J^R_{l}\g^*$. The following holds
    \[
        A[J^R_{l}\Pg] = A[J^R_{l}\g^*]^{J^R_{l}G}.
    \]
\end{lemma}
\begin{proof}
    Notice that by Lemma \ref{lemmaopenimmersions} $J^R_{l}\g^*_{\text{reg}}$ is an open subscheme of $J^R_{l}\g$, and the latter, since $\g$ is isomorphic to $\A^N$, is an integral scheme. In particular the map $A[J^R_{l}\g^*] \to A[J^R_{l}\g^*_{\text{reg}}]$ is injective.
    Since the action of $G$ naturally restricts to an action on $\g^*_{\text{reg}}$ we may reduce to show that
    \[
        A[J^R_{l}\Pg] = A[J^R_{l}\g^*_{\text{reg}}]^{J^R_{l}G}.
    \]
    This is because by the discussion above we already have natural inclusions
    \[
        A[J^R_{l}\Pg] \subset A[J^R_{l}\g^*]^{J^R_{l}G} \subset A[J^R_{l}\g^*_{\text{reg}}]^{J^R_{l}G}.
    \]
    We are going to use proposition 0.2 of \cite{mumford1994geometric}. Which proves that given two normal noetherian $S$ schemes $X$ and $Y$, a group scheme $H$, universally open over $S$ and of finite type, which act on $X$ and a $H$ invariant map $\pi : X \to Y$ such that every geometric fiber consists at most in a unique $H$ orbit then $\mathcal{O}_Y(Y)  = \mathcal{O}_X(X)^{H}$. 
    \smallskip
    
    We apply this proposition where $S = \spec A$, $X = J^R_{l}\g^*_{\text{reg}}$, $Y = J^R_{l}\Pg$ and $H = J^R_{l}G$.
    
    $J^R_{l}\Pg$ is certainly normal and noetherian since its isomorphic to $\A^N_A$, and the same goes for $J^R_{l}\g^*_{\text{reg}}$ since it is an open subscheme of a normal noetherian scheme. We showed before that $J^R_{l}G$ is of finite type over $A$, while the property of being universal open follows from the fact that $J^R_{l}G$ is smooth over $A$ (Corollary \ref{corosmoothness}). All that remains to show is that the geometric fibers are at most single $J^R_{l}G$-orbits.
    
    \smallskip
    In order to do this we are going to use the description of base change and the following key fact (which can be found for instance in \cite{frenkel2007langlands}, theorem 3.4.2): the classical map of Jet schemes $J_l\g^*_{\text{reg}} \to J_l\Pg$ has geometric fibers which consist in at most single $J_lG$ orbits.
    
    \smallskip
    Now consider a geometric point $s : \spec k \to J^R_{l}\Pg$, naturally $k$ is an $A$-algebra. To give a $k$ point of $J^R_l \Pg$ is equivalent to give a $k$ point of $\big( J^R_l \Pg\big)_k$, which, by Proposition \ref{propbasechange}, is naturally isomorphic to $\prod_i(J\Pg)_k$. We have the following diagram
    \[\begin{tikzcd}
	{\big(J^R_l\g^*_{\text{reg}}\big)_s} && {\prod_j(J_l\g^*_{\text{reg}})_k} && {J^R_l\g^*_{\text{reg}}} \\
	& \square && \square \\
	{\spec k} && {\prod(J_l\Pg)_k} && {J^R_{l}\Pg}
	\arrow["s", from=3-1, to=3-3]
	\arrow[from=1-3, to=3-3]
	\arrow[from=3-3, to=3-5]
	\arrow["s", bend right = 20, from=3-1, to=3-5]
	\arrow[from=1-3, to=1-5]
	\arrow[from=1-5, to=3-5]
	\arrow[from=1-1, to=3-1]
	\arrow[from=1-1, to=1-3]
\end{tikzcd}\]
    which shows that the geometric fiber $\big( J^R_l\g^*_{\text{reg}}\big)_s$ is isomorphic to the fiber of the left square. From this it follows that $\big( J^R_l\g^*_{\text{reg}}\big)_s$ is also a geometric fiber of the product map $\prod (J\g^*_{\text{reg}})_k \to \prod (J\Pg)_k$ and by the classical case the fibers of this map are at most single $(J_l G)_k$ orbits.
\end{proof}

\begin{proof}[Proof of Proposition \ref{propinvariants}]
    Injectivity follows from the existence of the section $J^R\sigma$. In addition since the map $J^R\pi$ is $J^RG$ invariant by construction we obtain a natural inclusion
    \[
        A[J^R\Pg] \subset A[J^R\g^*]^{J^R\g(A)}.
    \]
    To show surjectivity pick an element $ x \in A[J^R\g^*]^{J^R\g(A)}$, from the fact that $A[J^R\g^*] = \bigcup A[J^R_l\g^*]$ it follows that there exists $l$ such that $x \in A[J^R_l\g^*]$. In addition it follows from Remark \ref{rmkcontinuousaction} that 
    \[x \in A[J^R_l\g^*]^{J^R_l\g(A)}.\]
    
    It is now enough to prove that any $J^R_l\g(A)$ invariant element of $A[J^R_L\g^*]$ is also $J^R_LG$ invariant. This proves the proposition since, by Lemma \ref{lemmagroupinvariants}, we obtain $x \in A[J^R_l\Pg] \subset A[J^R\Pg]$.
    \smallskip
    
    Let $Q$ be the fraction field of $A$ and $\overline{Q}$ be its algebraic closure. $A[J^R_l\g^*]$ is free, hence its $A$ submodule $A[J^R_l\g^*]^{J^R_l\g(A)}$ is torsionless and therefore it embeds in
    \[
        A[J^R_l\g^*]^{J^R_l\g(A)} \subset \big(A[J^R_l\g^*]\big)_Q^{J^R_l\g(A)}.
    \]
    where $\big(A[J^R_l\g^*]\big)_Q$ is the tensor product $\big(A[J^R_l\g^*]\big)\otimes_AQ$. This, always by Proposition \ref{propbasechange}, is naturally isomorphic to $Q[\prod (J_l\g^*)_Q]$, in addition, since $J^R_L\g(A)$ is finite over $A$, we get
    \[
        \big(A[J^R_l\g^*]\big)_Q^{J^R_l\g(A)} = \big(A[J^R_l\g^*]\big)_Q^{(J^R_l\g(A))\otimes Q} = \big(A[J^R_l\g^*]\big)_Q^{\prod(J_l\g)(Q)} = Q\big[\prod (J_l\g^*)_Q\big]^{\prod(J_l\g)(Q)}.
    \]
    where the second equality follows from the equality $J^R_l\g(A)\otimes Q = \prod(J_l\g)(Q)$ with a base change argument. Now, by $Q$ linearity, we may pass from $Q$ to its algebraic closure $\overline{Q}$ obtaining an immersion
    \[
        A[J^R_l\g^*]^{J^R_l\g(A)} \subset \overline{Q}\big[\prod(J_l\g^*)_{\overline{Q}}\big]^{\prod(J_l\g)(\overline{Q})}.
    \]
    The latter, by standard theorems on smooth connected algebraic groups over an algebraically closed field of characteristic $0$ is equal to
    \[
        \overline{Q}\big[\prod(J_l\g^*)_{\overline{Q}}\big]^{\prod(J_l\g)(\overline{Q})} = \overline{Q}\big[\prod(J_l\g^*)_{\overline{Q}}\big]^{\prod(J_lG)(\overline{Q})} = \overline{Q}\big[\prod(J_l\g^*)_{\overline{Q}}\big]^{\prod(J_lG)_{\overline{Q}}}.
    \]
    The action of $J^R_lG$ on $A[J^R_l\g^*]$ may be written as a coaction of the Hopf algebra $A[J^R_lG]$, and the same goes for $\prod (J_lG)_{\overline{Q}}$ acting on $\overline{Q}\big[\prod(J_l\g^*)_{\overline{Q}}\big]$. The invariant elements are exactly those $x$ sent to $x\otimes 1$ via the coaction map. We have the following diagrams of coactions and inclusions
    \[\begin{tikzcd}
	{A[J^R_l\g^*]} && {A[J^R_l\g^*]\otimes_A A[J^R_lG]} \\
	\\
	{\overline{Q}\big[\prod(J_l\g^*)_{\overline{Q}}\big]} && {\overline{Q}\big[\prod(J_l\g^*)_{\overline{Q}}\big]\otimes_{\overline{Q}}\overline{Q}\big[\prod(J_lG)_{\overline{Q}}\big]}
	\arrow[hook, from=1-1, to=3-1]
	\arrow[from=3-1, to=3-3]
	\arrow[from=1-1, to=1-3]
	\arrow[hook, from=1-3, to=3-3]
\end{tikzcd}\]
    The first vertical arrow is injective by the discussion above, while the second vertical arrow is injective because $A[J^R_l\g^*]$ is free as an $A$ module while the smothness (hence flatness over $A$ hence torsionlessness) of $J^R_lG$ implies $A[J^R_lG] \subset \overline{Q}\big[\prod (J_lG)_{\overline{Q}}]$.
    Consider $x \in A[J^R_l\g^*]^{J^R_l\g(A)}$, we proved that the image of $x$ in $\overline{Q}\big[\prod(J_l\g^*)_{\overline{Q}}\big]$ is invariant for the action of the group, hence it must be sent, via the lower horizontal map, to $x\otimes 1$. Then by injectivity of the second vertical arrow we obtain that $x$ must be sent, via the upper horizontal map, to $x\otimes 1$, proving that $x$ is $J^R_lG$ invariant. 
\end{proof}

Now recall that in Definition \ref{deffunctiondistribution} we defined an $A[J^R\g^*]$-valued $K$ field $p_R$ associated to any polynomial $p \in A[\g^*]$ in particular to the various $P^i$ we may associate the fields $P^i_K$. We write $P^i_{j,k} = P^i_R(\epsilon_{j,k})$, these are elements in $A[J^R\g^*]$.

\begin{corollary}
    The elements $P^i_{j,k}$ with $k< 0$ are algebraically independent and following equality holds
    \[
        A[J^R\g^*]^{J^R\g(A)} = A[P^i_{j,k}]_{k<0}.
    \]
\end{corollary}

\begin{proof}
    Combine Proposition  \ref{propfunctiondistribution} with Proposition \ref{propinvariants}.
\end{proof}

Finally notice that the isomorphism
\[
    \big( \cdot\, \varphi^n)^* : A[J^R\A^N_A] \to A[J^K_{\geq n}\A^N_A]
\]
is $J^R\g(A)$ equivariant. Therefore we may relate the rings of invariants on both sides. Lemma \ref{lemmahompoly} as the following implication.

\begin{corollary}\label{coroinvariantsjk}[Of Lemma \ref{lemmahompoly}]
    The algebra of invariants $A[J^K_{\geq n}\A^N_A]^{J^R\g(A)}$ is isomorphic to the free polynomial algebra
    \[
        A[J^K_{\geq n}\A^N_A]^{J^R\g(A)} = A[P^i_{j,k_i}]_{k_i < d_in},
    \]
    where $d_i$ is the degree of $P^i$.
\end{corollary}

\subsection{Algebraic description of the center}

We are now ready to prove that the map we constructed
\[
    \tilde{U}_{K_n}\big(\C[Op_{^LG}(D)]\big) \to Z(\tilde{U}_{\kappa_c}(\hat{\g}_n)) 
\]
is an isomorphism. To make the notation lighter from now on we will call $Z_n(\hat{\mathfrak{g}}) \stackrel{\text{def}}{=} Z(\tilde{U}_{\kappa_c}(\hat{\g}_n))$

We will need the following result on filtered commutative algebras.

\begin{lemma}\label{lemmapolynomial}
    Let $S$ be a commutative algebra over a commutative ring $A$ which has an $A$-linear filtration. Suppose that there exist elements $x_i \in S$ such that their symbols $\overline{x_i} \in \gr S$ in the associated graded algebra satisfy
    \[
        \gr S = A[\overline{x_i}].
    \]
    Then the elements $x_i \in S$ are algebraically independent and generate $S$ as an $A$ algebra:
    \[
        S = A[x_i].
    \]
\end{lemma}

Let $U_N$ be the left ideals in $\tilde{U}_{\kappa_c}(\hat{\g}_n)$ generated by $\g\otimes (R_nf^N)$, these are the ideals which define the topology on the completed enveloping algebra. Denote by $I_N = Z_n(\hat{\mathfrak{g}})\cap U_N$, this is a two sided ideal of the center. It is easy to see that the center is a closed subalgebra of $\tilde{U}_{\kappa_c}(\hat{\g}_n)$ and therefore, it is complete. In particular we have
\[
    Z_n(\hat{\mathfrak{g}}) = \varprojlim_N \frac{Z_n(\hat{\mathfrak{g}})}{I_N}
\]
We therefore reduce ourselves to study the quotients $Z_n({\g})/I_N$

\begin{rmk}
    The space $\tilde{U}_{\kappa_c}(\hat{\g}_n)/U_N$ has a natural filtration induced by the PBW filtration. It is easy to see that there is a $J^R\g(A)$-equivariant natural isomorphism
    \[
        \gr (\tilde{U}_{\kappa_c}(\hat{\g}_n)/U_N) = A[J^K_{\geq N}\g^*].
    \]
    In particular we have the following diagram
    \[\begin{tikzcd}
	{\gr (Z_n(\hat{\mathfrak{g}})/I_N)} && {A[J^K_{\geq N}\g^*]^{J^R\g(A)}} & {} & {A[P^i_{j,k_i}]_{k_i<Nd_i}} \\
	\\
	{\gr(\tilde{U}_{\kappa_c}(\hat{\g}_n)/U_N)} && {A[J^K_{\geq N}\g^*]}
	\arrow[hook, from=1-1, to=3-1]
	\arrow[hook, from=1-1, to=1-3]
	\arrow["{=}", no head, from=3-1, to=3-3]
	\arrow[hook, from=1-3, to=3-3]
	\arrow["{\simeq}"', no head, from=1-3, to=1-5]
\end{tikzcd}\]
\end{rmk}

\begin{prop}\label{propgradedspaces}
    The map described above
    \[
        \gr (Z_n(\hat{\mathfrak{g}})/I_N) \hookrightarrow A[J^K_{\geq N}\g^*]^{J^R\g(A)} = A[P^i_{j,k_i}]_{k_i < Nd_i}
    \]
    is an isomorphism.
\end{prop}

\begin{proof}
    It is enough to prove that the map is surjective and hence that the various $P^i_{j,k_i}$ belong to the image of the center, or, equivalently, that there exists elements $\overline{P^i_{j,k_i}} \in Z_n(\hat{\mathfrak{g}})$ such that their symbols in $\gr (Z_n(\hat{\mathfrak{g}})/I_N)$ are equal to $P^i_{j,k_i}$.
   
    We claim that the elements $Y_{\text{can}}(P^i_{-1})(\epsilon_{j,k}) \in Z_n(\hat{\mathfrak{g}})$, which are the image of the elements $P^i_{-1}\epsilon_{j,k} \in \tilde{U}_{K_n}(\C[Op_{^L\g}(D)])$, have symbol equal to $P^i_{j,k}$. Recall that the elements $P^i_{-1}$ are by construction elements in $V^{\kappa_c}(\g)$ such that their associated symbols in $\gr V^{\kappa_c}(\g) = \C[J\g^*]$ are exactly the polynomials $P^i$. 
    
    To make the statement more precise, after a choice of a basis $x_a$ of $\g$ we can write
    
    $P^i = \sum_{\alpha} \lambda_{a} x_{a_1}\dots x_{a_d}$ where $d = d_i$ is the degree of the homogeneous polynomial $P^i$, $a = (a_1,\dots,a_d)$ is a $d$-uple of indexes. Then $P^i_{-1}$ may be expressed, up to an ordering and up to terms of lower PBW degree as 
    \[
        P^i_{-1} \sim \sum_{\alpha} \lambda_{\alpha} x_{a_1,-1}\dots x_{a_d,-1}\vac.
    \]
    
    By construction $Y_{\text{can}}(P^i_{-1})$ is obtained by taking the expression of $P^i_{-1}$, making the substitution $x_{a} \mapsto \overline{x_a}$ (where $\overline{x_a}$ is the $\tilde{U}_{\kappa_c}(\hat{\g}_n)$ valued $K_n$ field associated to $x_a \in \g$) and then applying the $(-1)$-product between $K_n$ fields where the ordinary $(-1)$-product occur in the vertex algebra $V^{k_c}(\g)$. One can see by direct computation, using the explicit description of the $(-1)$-product in $\tilde{U}_{\kappa_c}(\hat{\g}_n)/U_N$ as a finite sum of products of evaluations of our fields (see for instance the proof of Lemma \ref{continuitynproduct}), that the $(-1)$-product commutes with taking the associated graded space. From this follows that the image of $P^i_{-1}\epsilon_{j,k}$ in $\gr (\tilde{U}_{\kappa_c}(\hat{\g}_n)/U_N) = A[J^K_{\geq N}\g^*]$ equals to $P^i_K(\epsilon_{j,k})$, where $P^i_K$ is the field define in Definition \ref{deffunctiondistribution}. The claim that the symbol of $Y_{\text{can}}(P^i_{-1})(\epsilon_{j,k})$ is equal to $P^i_{j,k}$ now follows from Proposition \ref{propfunctiondistribution}.
\end{proof}

\begin{corollary}
    Let $\overline{P}^i_{j,k} = Y_{\text{can}}(P^i_{-1})(\epsilon_{j,k})$. Then for every $N\geq 0$
    \[
        \frac{Z_n(\hat{\mathfrak{g}})}{I_N} = A\big[\overline{P}^i_{j,k_i}\big]_{k_i < d_iN}.
    \]
    In particular
    \[
        Z_n(\hat{\mathfrak{g}}) = \varprojlim_N \frac{A\big[\overline{P}^i_{j,k}\big]_{k\in \Z}}{\big( \overline{P}^i_{j,k_i} : k_i \geq d_iN \big)}.
    \]
\end{corollary}

\begin{proof}
    This directly follows from the fact that $Z_n(\hat{\mathfrak{g}}) = \varprojlim_N Z_n(\hat{\mathfrak{g}})/I_N$, and the combination of Proposition \ref{propgradedspaces} and Lemma \ref{lemmapolynomial}.
\end{proof}

\begin{corollary}\label{coroisomorphism}
    The map we constructed
    \[
        \tilde{U}_{K_n}(\C[Op_{^L\g}(D)]) \to Z_n(\hat{\mathfrak{g}})
    \]
    is an isomorphism.
\end{corollary}

\begin{proof}
    This easily follows from the fact that by Proposition \ref{propdescrenvelopkalg} the algebra $\tilde{U}_{K_n}(\C[Op_{^L\g}(D)]$ has the same description of $Z_n(\hat{\mathfrak{g}})$. Indeed $\C[Op_{^L\g}](D)$ may be interpreted as an abelian affine vertex algebra, generated by the abelian finite dimensional Lie algebra $P = \spn (P^i_{-1})$. It follows from Proposition \ref{propdescrenvelopkalg} that
    \[
        \tilde{U}_{K_n}(\C[Op_{^L\g}(D)] = \varprojlim_N \frac{A\big[P^i\epsilon_{j,k}\big]_{k\in \Z}}{\big( P^i\epsilon_{j,k_i} : k_i \geq d_iN \big)}
    \]
    Here the shift in the grading is due to the fact that actually the Lie algebra $P$ is graded with $\deg P^i_{-1} = d_i$ and this grading is involved in the construction of the algebra $\tilde{U}_{K_n}(\C[Op_{^L\g}(D)])$ (see Remark \ref{rmkgradedversion}). 
    Finally, by Proposition \ref{propgradedspaces}, the elements $P^i\epsilon_{j,k}$ are sent exactly to $\overline{P}^i_{j,k}$ and this concludes the proof.
\end{proof}

\section{Identification with the algebra of functions on the space of Opers}

\subsection{Generalities on Opers}

This section is devoted to the identification of the complete $A$-algebra $\tilde{U}_{K_n}(\C[Op_{\g}(D)])$ with the algebra of functions on the space of $\g$-Opers over the formal $n$-pointed disc $D_n = \spec K_n$.

\medskip
From now on $\g$ is a fixed simple Lie algebra over $\C$ and $G$ is its associated group of adjoint type. We fix also a Cartan subalgebra and a Borel subalgebra $\h \subset \bg \subset \g$ and call respectively $H$ and $B$ the associated subgroups of $G$. The choice of $\h$ induces a set of roots $\Phi$ for $\g$, and the choice of $\bg$ induces a preferred basis of simple roots $\Delta \subset \Phi$. We are going to write $\Delta = \{ \alpha_i \}$, and by $e_i,h_i,f_i$ a standard set of generators for $\g$. We are going to call $p_{-1} = \sum f_i$, and $\rho^\vee$ for the element in $\h$ which induces the principal gradation (i.e. $(\ad \rho^\vee)_{|\g_{\alpha}} = ht(\alpha)$ for any root $\alpha$). It is a well known fact that after the choice of $p_{-1}$ and $\rho^\vee \in \h$ there exists a unique $\mathfrak{sl}_2$ triple $\{ p_{-1},2\rho^\vee,p_{1} \}$ inside of $\g$. 

We are going to denote by $V^{\text{can}} = \g^{p_1}$, and pick a basis of homogeneous elements $v_i \in V^{\text{can}}_i$, ordered by weight (so that $v_1$ has degree $1$), and we will denote by $d_i$ the degree of $v_i$. It is known that the natural numbers $d_i+1$ are exactly those appearing as the homogeneous degrees of the polynomials $P^i$ generating the center of the classical enveloping algebra.

\medskip

We start by recalling the definition of $\g$-Oper, following mainly \cite{beilinson2005opers}, with respect to our space of interest $D_n  = \spec K_n$.
Since we want to give to the space of Oper a geometric structure, and hence describe its functor of points, we are going to introduce immediately the notion of an $S$ family of Opers over $D_n$, for an arbitrary $A$-algebra $S$.

\begin{defi}
    An $S$ family of $\g$-Opers over $D_n$ is the datum of $P$, a trivial $G$-bundle over $(D_n)_S = \spec S(K_n)$ with a trivial $B$-reduction $P_B$. Equipped with a connection $\nabla$ such that under any isomorphism $P_B \simeq (D_n)_S\times B \subset (D_n)_S\times G \simeq P$ it takes the form
    \[
        \nabla = d + \bigg(\sum_i f_i\otimes\psi_i + v(t)\bigg)dt \quad \text{with} \quad \psi_i \in (S(K_n))^* \; \text{and} \; v(t) \in \bg\otimes_{\C}. S(K_n).
    \]
\end{defi}

There is an obvious notion of isomorphism of Opers. In more concrete terms, when two Opers $\nabla,\nabla'$ are written in the above form they are isomorphic if and only if there is a section $b \in B(S(K_n))$ such that
\[
    b\cdot \nabla = \nabla',
\]
where the action of $B$ is given by the so called \emph{action by Gauge transformations}
\[
    b\cdot \big( d + A(t)dt\big) = d + \big( \text{Ad}_{b}(A(t)) \big)dt - db\cdot b^{-1}.
\]
The following lemma is fundamental and can be found in its original version in \cite{drinfeld1984lie}.

\begin{lemma}\label{lemmacanonicalrepresentatives}
    Given $\nabla$, an $S$ family of $\g$-Opers there exists a unique element $b \in B(S(K_n))$ such that
    \[
        b \cdot \nabla = d + \big( p_1 + c(t)\big)dt, \quad \text{with} \quad c(t) \in V^{\text{can}}\otimes_{\C} S(K_n).
    \]
    
    This will be called the \textbf{canonical form} of the Oper $\nabla$.
\end{lemma}

\begin{proof}
    First one can take a representative of $\nabla$ of the form
    \[
        \nabla = d + \bigg(\sum_i f_i\otimes\psi_i + v(t)\bigg)dt.
    \]
    It is easy to check, that since $G$ is of adjoint type, there exists a unique $h \in H(S(K_n))$ and a unique $v_{1}(t) \in \bg\otimes S(K_n)$ such that $h \cdot \nabla$ has the following form:
    \[
        h\cdot \nabla = d + \big( p_{-1} + v_{1}(t) \big)dt.
    \]
    Now we need to show that there exists a unique element $U \in \mathfrak{n}\otimes K_n$ such that
    \[
        \exp(U)\cdot \bigg( d + \big( p_{-1} + v_{1}(t) \big)dt \bigg) = d + \big( p_{-1} + v_{2}(t) \big)dt,
    \]
    with $v_{2}(t) \in V^{\text{can}}\otimes K_n$. This proof is completely analogous to the one of Lemma 4.2.2 of Frenkel's book \cite{frenkel2007langlands}.
\end{proof}

The same argument may be carried out for $R_n$ instead of $K_n$. Or in general for anything with trivial rank 1 module of differentials.
We can define the functor of Opers as follows:

\begin{defi}
    We define the functor of $\g$-Opers with $n$ singularities $Op_{\g,n}$ as a functor on $A$-algebras
    \[
        Op_{\g,n}(S) = \big\{ \text{isomorphism classes of } S \text{-families of } \g\text{-Opers over } D_n \big\}.
    \]
\end{defi}

\begin{corollary}[Of Lemma \ref{lemmacanonicalrepresentatives}]\label{corooperindscheme}
    The construction of canonical representatives induces the following description for functor of $A$-algebras $Op_{\g,n}$
    \[
        Op_{\g,n}(S) = \big\{ d + (p_1 + v(t))dt : v(t) \in V^{\text{can}}\otimes S(K_n)\big\}.
    \]
    In particular $Op_{\g,n}$ is isomorphic to the functor $J^K\A^l$ where $l$ is the rank of $\g$, which is the same as the dimension of the space $V^{\text{can}}$.
\end{corollary}

Analogously to the treatment of Jet schemes, after the choice of the basis $v_i$, we may associate to an index $i$ and to a function $g \in K_n$ a function on $Op_{\g,n}$ which we call $v_i^*(g)$ defined as

\[
    v_i^*(g) : Op_{\g,n} \to \A^1, \qquad d + \big( p_{-1} + v(t)\big)dt \mapsto \int v_i^*(v(t))gdt,
\]
where $v_i^*$ denotes the dual basis of the basis $v_i$ of $V^{\text{can}}$. By Corollary \ref{corooperindscheme} we have
\[
    A[Op_{\g,n}] \simeq \widetilde{\Sym}_A(V^{\text{can}}\otimes K_n) = \varprojlim_N \frac{A\big[v_i^*(\epsilon_{j,k})\big]}{ (v_i^*(\epsilon_{j,k}): k \geq N)}.
\]

\subsection{Action of automorphisms and Lie algebras}

Every automorphism of the ring $K_n$ changes the chosen trivialization of the cotangent bundle. Namely, if $\psi : K_n \to K_n$ is an automorphism, the new coordinate which trivializes $\Omega^{1,\text{cont}}_{K_n/A}$ will be $d(\psi(t))$. This construction induces a natural action on the space of connections and hence on the space of Opers. \newline
This is realized as follows. Consider a continuous $S$ automorphism $\varphi : \spec S(K_n) \to \spec S(K_n)$ and let $\psi : S(K_n) \to S(K_n)$ the associated continuous automorphism of the ring $S(K_n)$ let $s = \psi(t)$ consider an $S$ family of Opers $\nabla$ in the canonical form with respect to the coordinate $s$
\[
    \nabla = d + \big( p_{-1} + v(s)\big)ds, \quad \text{with} \quad v(s) \in V^{\text{can}}\otimes_{\C} S(K_n).
\]

Then the connection $\nabla$, is read under the isomorphism $\varphi$ as
\[
    \nabla = d + \big( p_{-1} + v(\psi(t))\big)d(\psi(t)) = d + \big( p_{-1} + v(\psi(t))\big)\partial_t(\psi(t))dt.
\]
Here, with a slight abuse of notation when we write $v(\psi(t))$ we mean $\psi\big(v(s))$ where $\psi$ acts on the second factor of $V^{\text{can}}\otimes S(K_n)$. To bring back the Oper in its canonical form it is enough to apply the gauge action of $g = \exp\big( \frac{1}{2}\frac{\partial_t^2\psi(t)}{\partial_t\psi(t)}p_1\big)\rho^\vee(\partial_t\psi(t))$ as is shown in \cite{frenkel2007langlands}, formulas 4.2-5/4.2-8. Notice that this makes sense since by Corollary \ref{derivativeinvertible} the function $\partial_t\psi(t)$ in invertible.

Everything done so far is functorial, and considering $\underline{\Aut}(K_n)$ the group valued functor of $A$-algebras $S \mapsto Aut_S^{\text{cont}}(S(K_n))$ we obtain the following proposition.

\begin{prop}
    There is a natural action of the group functor $\underline{\Aut}(K_n)$ on the space $Op_{\g,n}$ which, given $\psi \in \underline{\Aut}(K_n)(S)$ and $\nabla = d + (p_1 + v(t))dt \in Op_{\g,n}(S)$, may be described as follows
    \begin{align}\label{formulaschangecoordinate}
        \psi  \cdot v_1(s) &= v_1(\psi(s))(\partial_t\psi(t))^2 - \frac{1}{2}\bigg(\frac{\partial_t^3\psi(t)}{\partial_t\psi(t)} - \frac{3}{2}\bigg( \frac{\partial_t^2\psi(t)}{\partial_t\psi(t)} \bigg)^2 \bigg), \\
        \psi \cdot v_j(s) &= v_j(\psi(s))(\partial_t\psi(t))^{d_j+1} \qquad \text{ for } j>1.
    \end{align}
\end{prop}


From the action of the group functor $\underline{\Aut}(K_n)$ there is a naturally attached action of its Lie algebra $\Der K_n = K_n\partial_t$.

\begin{corollary}
    Let $h\partial_t \in \Der K_n$ and $v_i^*(g) \in A[Op_{\g,n}]$ the function defined at the end of the previous section. Then the Lie algebra $\Der K_n$ acts on the generators $v_i^*(g)$ of $A[Op_{\g,n}]$ as follows 
    \begin{align}
        h\partial_t \cdot v_1^*(g) &=  -v_1^*(g\partial_t h) + v_1^*(h\partial_t g) - \frac{1}{2}\int (\partial_t^3h)gdt,\\
        h\partial_t \cdot v_j^*(g) &=  -d_jv_j^*(g\partial_th) + v_j^*(h\partial_tg) \qquad \text{ for } j>1.
    \end{align}
\end{corollary}
We can apply everything done so far with $R_n$, instead of $K_n$. Notice that in the case where $n = 1$, and $a_1 = 0$ so that $A = \C$ and $R_n = \C[[t]]$ the action is the usual one on $\C[Op_{\g}(D)]$. Denote by $v_{i,m}^* = v_{i}^*(t^m)$ (in the case of $\C[[t]]$ we consider only $m \leq -1$) and by $L_n  = - t^{n+1}\partial_t$ (in the case of $\C[[t]]$ we consider only $n \geq -1$) then
\begin{align}
    L_n \cdot v^*_{1,-m-1} &=  (n+m+2)v^*_{i,n-m-1} + \frac{1}{2}(n+1)n(n-1)\delta_{n-m,2}\vac,\\
    L_n \cdot v^*_{j,-m-1} &=  (d_j(n+1)+m+1)v^*_{j,n-m-1} \qquad \text{ for } j>1.
\end{align}

In particular
\begin{align}
    L_n v_{1,-1}^* &= 
    \begin{cases}
        v_{1,-2}^* = T(v_{1,-1}^*) & \text{if } n=-1, \\
        2v_{1,-1}^* & \text{if } n= 0, \\
        3\vac & \text{if } n=2, \\
        0 & \text{otherwise.}
    \end{cases} \\
    L_n v_{j,-1}^* &= 
    \begin{cases}
        v_{j,-2}^* = T(v_{j,-1}^*) & \text{if } n=-1, \\
        (d_j+1)v_{j,-1}^* & \text{if } n= 0, \\
        0 & \text{otherwise.}
    \end{cases}
\end{align}

\subsection{A Feigin-Frenkel isomorphism in the \texorpdfstring{$n$}{n} singularities setting}

We are now ready to construct a natural $\Der K_n$-equivariant isomorphism between $\tilde{U}_{K_n}(\C[Op_{\g}(D)])$ and $A[Op_{\g,n}]$. We are going to prove in the last section that this isomorphism satisfies the natural factorization properties.

\smallskip

Recall that $\C[Op_{\g}(D)] = \C[v_{i,n}^*]_{n<0}$, where the $v_{i,n}^*$ are the functions on $Op_{\g}(D)$ defined in the previous section. Call $(\tilde{V}^{\text{can}})^*$ the linear span of the $v_{i,-1}^*$. By Proposition \ref{propdescrenvelopkalg} we have
\[
    \tilde{U}_{K_n}(\C[Op_{\g}(D)]) \simeq \varprojlim_N \frac{A[v_{i,-1}^*\epsilon_{j,k}]}{(v_{i,-1}^*\epsilon_{j,k_i} : k_i \geq Nd_i)} \simeq \widetilde{\Sym}_A(\tilde{V}^{\text{can}}\otimes_{\C} K_n).
\]

The algebras we are studying have therefore the same description.

\begin{prop}\label{propisooper}
    The natural isomorphism
    \[
        \tilde{U}_{K_n}(\C[Op_{\g}(D)]) \to A[Op_{\g,n}],
    \]
    which on generators sends $\tilde{U}_{K_n}(\C[Op_{\g}(D)]) \ni v_{i,-1}^*g \mapsto v_i^*(g) \in A[Op_{\g,n}]$ for any $g \in K_n$ is $\Der K_n$-equivariant.
\end{prop}

\begin{proof}
    Since both spaces have the same description it is quite clear that the map above is an isomorphism. We only need to check the $\Der K_n$ equivariance. Recall that the action of $\Der K_n$ on $\tilde{U}_{K_n}(\C[Op_{\g}(D)])$ is defined as follows
    \[
        (h\partial_t)\cdot (v_{i,-1}^*g) = -\sum_{k\geq -1} \frac{1}{(k+1)!}(L_kv_{i,-1}^*)(g\partial_t^{k+1}h).
    \]
    In particular by the formulas of the previous section
    \begin{align*}
        (h\partial_t)\cdot (v_{1,-1}^*g) &= -T(v_{1,-1}^*)(gh) - 2v_{1,-1}^*(g\partial_t h) - \frac{1}{2}\vac(g\partial_t^3h) \\ &= -v_{1,-1}^*(g\partial_t h) + v_{1,-1}^*(h\partial_t g) - \frac{1}{2}\vac(g\partial_t^3 h), \\
        (h\partial_t)\cdot (v_{j,-1}^*g) &= -T(v_{j,-1}^*)(gh) - (d_j+1)v_{j,-1}^*(g\partial_t h) = -d_jv_{j,-1}^*(g\partial_t h) + v_{j,-1}^*(h\partial_t g).
    \end{align*}
    The images of the right hand sides of this equation are exactly
    \begin{align*}
         -v_1^*(g\partial_t h) + v_1^*(h\partial_t g) - \frac{1}{2}\int (\partial_t^3h)gdt &= (h\partial_t) \cdot v_1^*(g),\\
        -d_jv_j^*(g\partial_th) + v_j^*(h\partial_tg) &= (h\partial_t) \cdot v_j^*(g) \qquad \text{ for } j>1.
    \end{align*}
    This shows equivariance on the generators and therefore that the hole morphism is equivariant.
\end{proof}
    Combining Proposition \ref{propisooper} and Corollary \ref{coroisomorphism} we get the following theorem.
\begin{thm}\label{theoremmain}
    There is a natural $\Der K_n$-equivariant isomorphism
    \[
        Z_n(\hat{\mathfrak{g}}) = A[Op_{^L\g,n}]
    \]
\end{thm}

\section{Factorization properties}

In this last section we are going to show that the $A_n$-commutative algebras $Z_n(\hat{\mathfrak{g}})$ and $A[Op_{^L\g,n}]$ are factorization algebras in the sense of \cite{beilinson2004chiral} and that the isomorphism of Theorem \ref{theoremmain} preserves the factorization structures. Since the algebras we are considering are complete topological we shall use several completions of tensor products. 

\begin{notation}
    In this section and only for the tensor product of enveloping algebras or the algebras of functions on the space of Opers, the symbol $\hat{\otimes}$ will denote the completion of the tensor product along the naive sum topology (defined in the discussion before Lemma \ref{continuousextension}). Since the algebras we are dealing with have topologies generated by two sided or left ideals it can be checked that the completions along the naive sum topology have natural structures of algebras.
\end{notation}

In order to deal with factorization properties we introduce a slight change of notation, given a finite set $I$ we will denote by $A_I = \C[D^I] = \C[[a_i]]_{i\in I}$ and define the complete $(\varphi_I)$-adic ring $R_I$ in the same way of $R_n$, and consider its localization $K_I = (R_I)_{\varphi_I}$. 


\subsection{Factorization properties of \texorpdfstring{$K_n$}{Kn} and \texorpdfstring{$\hat{\g}_n$}{of the affine algebra with n singularities}}

Given an $A_I$ algebra $S$ recall the construction of $S(K_I)$, defined in the discussion before Lemma \ref{continuousextension}, as the natural completion of the tensor product $S\otimes_{A_I}K_I$ along the right topology. 

Given a surjection $I \twoheadrightarrow J$ there is a natural map $A_I \to A_J$ in particular $A_J$ is an $A_I$-algebra and we may consider the complete topological ring $A_J(K_I)$ 

On the other hand, for a given decomposition $I = I_1 \coprod I_2$ denote by $A_{I_1,I_2}$ the localization of the ring $A_I$ by the elements $a_i - a_j$ as $i$ runs over $I_1$ and $j$ runs over $I_2$. This is naturally an $A_I$ algebra and hence an $A_{I_1},A_{I_2}$ algebra in a natural way. 

\begin{prop}\label{propfactrk} The following hold:

\begin{itemize}
    \item For any given surjection $f : I \twoheadrightarrow J$ there is a natural isomorphism
    \[
        A_J(K_I) = K_J.
    \]
    Attached to this isomorphism there is a natural map $\restrp_{I\to J} : K_I \to A_J(K_I) = K_J$ which we call the \textbf{restriction map}.
    \item For any given decomposition $I = I_1 \coprod I_2$ there is a natural isomorphism
        \[
            A_{I_1,I_2}(K_I) = \big( A_{I_1,I_2}(K_{I_1})\times A_{I_1,I_2}(K_{I_2})\big).
        \]
    Attached to this isomorphism there is a natural map $e_{I_1,I_2}^I : K_I \to A_{I_1,I_2}(K_I)$ and two natural maps $e_{I_1}^I : K_I \to A_{I_1,I_2}(K_I) \stackrel{\pi_1}{\longrightarrow} A_{I_1,I_2}(K_{I_1})$ and $e_{I_2}^I : K_I \to A_{I_1,I_2}(K_I) \stackrel{\pi_2}{\longrightarrow} A_{I_1,I_2}(K_{I_2})$ which we call the \textbf{expansions maps}.
    \item The isomorphisms above are compatible with residues and derivatives, where the residue on the product is defined as the sum of the residues on the factors.
\end{itemize}
\end{prop}

\begin{proof}

One can show this using the isomorphism $S(K_I) = S(R_I)_{\varphi_I}$, which is valid for any $A_I$ algebra and hence reducing ourselves the analogous statement with $R_I$ instead of $K_I$.

The first point follows from the fact that the tensor product $A_J\otimes_{A_I} K_I$ is isomorphic to the quotient of $K_I$ by the ideal generated by $a_i - a_j$ for those couples of indices such that $f(i) = f(j)$. This is a closed ideal hence $A_J\otimes_{A_I} K_I$ is already complete and the equality above easily follows from showing that $A_J(R_I) = R_J$. Indeed we have 
\[
    A_J(R_I) = \varprojlim_n \frac{A_J[t]}{\varphi_{I \to J}^n},
\]
where $\varphi_{I \to J} = \prod_{j \in J} (t-a_{j})^{|f^{-1}(j)|}$. It is easy to check that the topology induced by $\varphi_{I\to J}$ is equivalent to the one generated by $\varphi_J$. This proves $A_J(R_I) = R_J$.

As before, we first show that $A_{I_1,I_2}(R_I) = A_{I_1,I_2}(R_{I_1}) \oplus A_{I_1,I_2}(R_{I_2})$. Indeed we have
\[
    A_{I_1,I_2}(R_I) = \varprojlim_n \frac{A_{I_1,I_2}[t]}{\varphi_I^n}.
\]
By definition $\varphi_I = \varphi_{I_1}\varphi_{I_2}$ and by definition of $A_{I_1,I_2}$ the resultant of $\varphi_{I_1}$ and $\varphi_{I_2}$ is invertible in the ring $A_{I_1,I_2}$. Hence, by known results on the resultant, we have $(\varphi_{I_1},\varphi_{I_2}) = A_{I_1,I_2}[t]$. Finally, by the Chinese remainder theorem, it follows that
\[
    A_{I_1,I_2}(R_I) = \varprojlim_n \frac{A_{I_1,I_2}[t]}{\varphi_I^n} = \varprojlim \bigg( \frac{A_{I_1,I_2}[t]}{\varphi_{I_1}^n} \times \frac{A_{I_1,I_2}[t]}{\varphi_{I_2}^n}\bigg) = A_{I_1,I_2}(R_{I_1})\times A_{I_1,I_2}(R_{I_2}).
\]
Now notice that $\varphi_{I_2}$ is invertible in $A_{I_1,I_2}(R_{I_1})$: it is indeed invertible in $A_{I_1,I_2}[t]/(\varphi_{I_1})$ since, as noted before, $(\varphi_{I_1},\varphi_{I_2}) = A_{I_1,I_2}[t]$. Analogously $\varphi_{I_1}$ is invertible in $A_{I_1,I_2}(R_{I_2})$. Therefore we have equalities $A_{I_1,I_2}(K_{I_j}) = A_{I_1,I_2}(R_{I_j})_{\varphi_I}$. From this follows that the isomorphism
\[
    A_{I_1,I_2}(R_I) = A_{I_1,I_2}(R_{I_1})\times A_{I_1,I_2}(R_{I_2})
\]
induces an isomorphism
\[
    A_{I_1,I_2}(K_I) = A_{I_1,I_2}(R_{I_1})_{\varphi_I}\times A_{I_1,I_2}(R_{I_2})_{\varphi_I} = A_{I_1,I_2}(K_{I_1})\times A_{I_1,I_2}(K_{I_2})
\]

The fact that this isomorphisms are compatible with derivatives is obvious. While compatibility with the residues easily follows from our definition of the  residue on $K_I$ as the sum of the residues on the points $\{a_i\}_{i \in I}$.
\end{proof}

The factorization properties of the ring $K_I$ induce factorization properties for $A_I[Op_{^L\g,I}]$, the ring of functions on the space of Opers. To make notation lighter for an $A_I$-algebra $B$ we will write $B[Op_{^L\g,I}]$ for the base change $B\big( A_I[Op_{^L\g,I}] \big)$.

\begin{prop}
    The collection $I \mapsto A[Op_{^L\g,I}]$ is naturally a factorization algebra. Meaning that
    \begin{itemize}
        \item For each surjection $I \twoheadrightarrow J$ there is a natural isomorphism
        \[
            A_J[Op_{^L\g,I}] = A_J[Op_{^L\g,J}].
        \]
        \item For each decomposition $I = I_1 \coprod I_2$ there is a natural isomorphism
        \[
            A_{I_1,I_2}\big(A_I[Op_{^L\g,I}]\big) = A_{I_1,I_2}[Op_{^L\g,I_1}\times Op_{^L\g,I_2}] = A_{I_1,I_2}[Op_{^L\g,I_1}]\hat{\otimes} A_{I_1,I_2}[Op_{^L\g,I_2}].
        \]
    \end{itemize}
\end{prop}

\begin{prop}\label{propfactalg}
The factorization properties of the complete topological ring $K_I$ transfer to the Lie algebra $\hat{\g}_{K_I,\kappa}$. There are natural isomorphisms of Lie algebras
\begin{align*}
    A_{J}(\hat{\g}_{K_I,\kappa}) &= \hat{\g}_{K_J,\kappa}, \\
    A_{I_1,I_2}(\hat{\g}_{K_I,\kappa}) &= \big(A_{I_1,I_2}(\hat{\g}_{K_{I_1},\kappa}) \oplus A_{I_1,I_2}(\hat{\g}_{K_{I_2},\kappa})\big)/\big( \mathbf{1}_1 - \mathbf{1}_2 \big).
\end{align*}
These extend to the completion of the enveloping algebras, yielding isomorphisms
\begin{align*}
    A_J\big( \tilde{U}_{\kappa}(\hat{\g}_{K_I}) \big) &= \tilde{U}_{\kappa}(\hat{\g}_{K_J}), \\
    A_{I_1,I_2}\big( \tilde{U}_{\kappa}(\hat{\g}_{K_I}) \big) &= A_{I_1,I_2}\big(\tilde{U}_{\kappa}(\hat{\g}_{K_{I_1}})\big)\hat{\otimes}A_{I_1,I_2}\big(\tilde{U}_{\kappa}(\hat{\g}_{K_{I_2}})\big).
\end{align*}

\end{prop}

\begin{proof}
    The isomorphisms
    \begin{align*}
        A_{J}(K_I) = K_J \qquad
        A_{I_1,I_2}(K_I) = A_{I_1,I_2}(K_{I_1})\times A_{I_1,I_2}(K_{I_2})
    \end{align*}
    induce natural isomorphisms
    \begin{align*}
    A_{J}(\hat{\g}_{K_I,\kappa}) &= \hat{\g}_{K_J,\kappa}, \\
    A_{I_1,I_2}(\hat{\g}_{K_I,\kappa}) &= \big(A_{I_1,I_2}(\hat{\g}_{K_{I_1},\kappa}) \oplus A_{I_1,I_2}(\hat{\g}_{K_{I_2},\kappa})\big)/\big( \mathbf{1}_1 - \mathbf{1}_2 \big).
    \end{align*}
    The fact that these isomorphisms are Lie algebra isomorphisms follows from the fact that the isomorphisms introduced in Proposition \ref{propfactrk} commute with the residues, so that the bracket is preserved.
\end{proof}

\subsection{Factorization of fields}

We focus on the following setting: consider $X$ an $U$-valued $K_I$ field, where $U$ is as always a complete topological $A_I$ algebra with topology defined by left ideals. We consider two parallel cases

\begin{itemize}
    \item Suppose that we have a surjection $f : I \twoheadrightarrow J$ and that the action of $A_I$ on $U$ factors through $A_I \to A_J$. Then the field $X : K_I \to U$ extends linearly in a unique way to a field $X : A_J(K_I) \to U$. By the previous proposition we have $A_J(K_I) = K_J$. In particular for an arbitrary $U$ we may consider the natural completion of the tensor product $A_J(U)$. There is a natural map $U \to A_J(U)$ and given a field $X : K_I \to U$ we may post compose it with the natural map $U \to A_J(U)$.  This construction yield an $A_I$ linear map between the following spaces of fields
    \[
        \restr_{I\to J} : F_{A_I}(K_I,U) \to F_{A_J}(K_J,A_J(U)),
    \]
    which we call the \textbf{restriction map}. Given a field $X \in F_{A_I}(K_I,U)$ the field $\restr_{I \to J} X$ is the only field making the following diagram commute:
    \[\begin{tikzcd}
	{K_I} && U \\
	\\
	{K_J} && {A_J(U)}
	\arrow["X", from=1-1, to=1-3]
	\arrow["{\restrp_{I\to J}}"', from=1-1, to=3-1]
	\arrow[from=1-3, to=3-3]
	\arrow["{\restr_{I \to J}X}"', from=3-1, to=3-3]
\end{tikzcd}\]
    \item Suppose that we are given a decomposition $I = I_1 \coprod I_2$, and consider the $A_I$-algebra $A_{I_1,I_2}$. Suppose that the action of $A_I$ on $U$ extends to an action of $A_{I_1,I_2}$. Then $X$ extends $A_{I_1,I_2}$ linearly to $A_{I_1,I_2}\otimes K_I$ in an unique way, and by continuity it extends in a unique way to a $U$ valued $A_{I_1,I_2}(K_I)$ field. 
    As before, given an arbitrary $U$ we may consider $A_{I_1,I_2}(U)$ and post compose the field $X : K_I \to U$ with the natural map $U \to A_{I_1,I_2}(U)$. This yields an $A_I$ linear map between the following spaces of fields
    \[
        E_{I_1,I_2}^I : F_{A_I}(K_I,U) \to F_{A_{I_1,I_2}}(A_{I_1,I_2}(K_I),A_{I_1,I_2}(U)),
    \]
    which we call the \textbf{expansion map}. Given $X \in F_{A_I}(K_I,U)$, the field $E_{I_1,I_2}^I X$ is the only one making the following diagram commute:

    \[\begin{tikzcd}
	{K_I} && U \\
	\\
	{A_{I_1,I_2}(K_{I_1})\times A_{I_1,I_2}(K_{I_2})} && {A_{I_1,I_2}(U)}
	\arrow["X", from=1-1, to=1-3]
	\arrow["{e_{I_1,I_2}^I}"', from=1-1, to=3-1]
	\arrow[from=1-3, to=3-3]
	\arrow["{E_{I_1,I_2}^IX}"', from=3-1, to=3-3]
\end{tikzcd}\]
    
    Using the isomorphism $A_{I_1,I_2}(K_I) = A_{I_1,I_2}(K_{I_1})\oplus A_{I_1,I_2}(K_{I_2})$ the codomain of $E_{I_1,I_2}^I$ may be rewritten as
    \[
        F_{A_{I_1,I_2}}(A_{I_1,I_2}(K_{I_1}),A_{I_1,I_2}(U))\oplus F_{A_{I_1,I_2}}(A_{I_1,I_2}(K_{I_2}),A_{I_1,I_2}(U)).
    \]
    We denote by $E_{I_1}^I$ and $E_{I_2}^I$ the compositions of $E_{I_1,I_2}^I$ with the natural projections on the first and second factor with respect to this decomposition respectively.
\end{itemize}

\begin{prop}\label{propfactfields} The following hold:
    \begin{itemize}
        \item The map constructed above $\restr_{I\to J}$ commutes with all $n$-products of fields.
        \item If we are given two fields $X,Y$ such that $[E_{I_1}^IX(g_1),E_{I_2}^IY(g_2)] = 0$, for every $g_1 \in A_{I_1,I_2}(K_{I_1})$ and every $g_2 \in A_{I_1,I_2}(K_{I_2})$, then the morphisms $E_{I_1}^I$ and $E_{I_2}^I$ commute with all $n$-products.
    \end{itemize}
\end{prop}

The proof of the above proposition will require several remarks on the rings we are dealing with. We are going to consider continuous maps of localizations of $(\varphi)$-adic rings with a global coordinate $\chi : K_1 \to K_2$ which sends the coordinate of $K_1$ to the coordinate of $K_2$ and call them coordinate-preserving morphisms. We require $K_1$ to be a $(\varphi_1)$-adic ring over a base ring $A_1$ and $K_2$ to be a $(\varphi_2)$-adic ring over a base ring $A_2$, we do not require though that $A_1 = A_2$. We only require that $A_2$ is an $A_1$-algebra and that the morphism $\chi$ is $A_1$-linear. Coordinates are to be understood relative to the corresponding base rings.

To make the notation lighter we are going to fix a decomposition $I= I_1 \coprod I_2$ and denote by $\tilde{K}_{I_j} = A_{I_1,I_2}(K_{I_j})$ and by $\tilde{K}_I = A_{I_1,I_2}(K_{I}) = \tilde{K}_{I_1} \times \tilde{K}_{I_2}$. 

\begin{rmk}\label{rmkfattiutili} The following hold
    \begin{enumerate}
        \item Given fields $X,Y : K \to U$ and a morphism $\chi : K' \to K$ we may pullback them to $K'$ fields $\chi^*X,\chi^*Y$. Since the morphism $K' \to K$ preserves the chosen coordinates, taking the pullback commutes with all $n$-products. Indeed it is quite clear that as $2$ distributions $\chi^*(XY)  = (\chi^*X)(\chi^*Y)$. There is also a natural map
        \[
            \chi\otimes \chi : \widehat{K'\otimes K'}^2 \to \widehat{K\otimes K}^2.
        \]
        The equality $\chi^*(XY)  = (\chi^*X)(\chi^*Y)$ also hold on these completions. By the assumption that $\chi$ preserves the chosen coordinates it follows that $\chi\otimes\chi$ sends $(z'-w')^n$ (coordinates in $K'$) to $(z-w)^n$ (coordinates in $K$) and therefore $\chi^*$ commutes with taking the $n$-products . Notice that it would be enough that $\chi\otimes\chi$ preserves the difference of the coordinates $z-w$;
        \item Recall that the product of two fields $XY$ extends naturally to $\hat{K\otimes K}$, the completion of $K\otimes K$ along the middle topology. The same holds for the product $\sigma \circ YX$.
        \item Now consider the case of $\tilde{K}_I = \tilde{K}_{I_1}\times \tilde{K}_{I_2}$. The completed tensor product $\tilde{K}_I\hat{\otimes}\tilde{K}_I$ splits into four direct summands of the form $K_{ij} \stackrel{\text{def}}{=} \tilde{K}_{I_i}\hat{\otimes} \tilde{K}_{I_j}$. We claim that if $i\neq j$ the function $z-w \in \tilde{K}_I\hat{\otimes} \tilde{K}_I$ is already invertible when restricted to $K_{ij}$. Indeed we are going to show that $z-w$ is invertible in $\tilde{R}_{I_1}\hat{\otimes}\tilde{R}_{I_2}$, this is sufficient to prove our claim since there is a natural map of rings $\tilde{R}_{I_1}\hat{\otimes} \tilde{R}_{I_2} \to K_{ij}$. 
        
        As stated in Proposition \ref{propgeneralitiesfadic}[4] there is an isomorphism $\tilde{R}_{I_1}\hat{\otimes}\tilde{R}_{I_2} = \widehat{\tilde{R}^{\text{pol}}_{I_1}\otimes \tilde{R}^{\text{pol}}_{I_2}}$, where $\tilde{R}^{\text{pol}}_{I_1} = A_{I_1,I_2}[z]$ with the topology defined by $\varphi_{I_1}$ and $\tilde{R}^{\text{pol}}_{I_2} = A_{I_1,I_2}[w]$ with the topology defined by $\varphi_{I_2}$. To show that $z-w$ is invertible in the completion is enough to show that it is invertible in the first  quotient $A_{I_1,I_2}[z,w]/(\varphi_{I_1}(z),\varphi_{I_2}(w))$. This is equivalent to show that if we quotient again by $z-w$ we obtain the $0$ ring. This second quotient is naturally isomorphic to $A_{I_1,I_2}[z]/(\varphi_{I_1}(z),\varphi_{I_2}(z))$. By construction of $A_{I_1,I_2},\varphi_{I_1},\varphi_{I_2}$ we know that the resultant of $\varphi_{I_1}(z)$ and $\varphi_{I_2}(z)$ is invertible, hence the ideal generated by them is the whole ring. We conclude that
        \[
            \frac{A_{I_1,I_2}[z,w]/(\varphi_{I_1}(z),\varphi_{I_2}(w))}{(z-w)} = 0,
        \]
        and therefore $z-w$ must be invertible;
    \end{enumerate}
\end{rmk}

\begin{proof}[Proof of Proposition \ref{propfactfields}]
    The assertion regarding the restriction morphism $\restr_{I\to J}$ follows from the first point of the previous remark.
    
    To show the compatibility between the expansion and the $n$ products first notice that we may decompose our fields $X,Y$ as follows. To shorten notation let $\tilde{X}_j = E_{I_j}^IX$ be the restriction to $\tilde{K}_{I_j}$ of the expansion $E_{I_1,I_2}^IX$, there are natural maps $E_{I_j}^I : \tilde{K}_I \to \tilde{K}_{I_j}$ and the following equality follows immediately from the definitions
    \[
        X = (E_{I_1}^I)^*\tilde{X}_1 + (E_{I_2}^I)^*\tilde{X}_2.
    \]
    As noted before, this equation uniquely determines the $\tilde{X}_i$, meaning that given two other fields $\chi_i : \tilde{K}_{I_i} \to A_{I_1,I_2}(U)$ such that $X = (E_{I_1}^I)^*\chi_1 + (E_{I_2}^I)^*\chi_2$ then $\chi_i = X_i$. This is because the image of $A_{I_1,I_2}\otimes K_I \to K'_{I_1}\times K'_{I_2}$ is dense.
    
    By bilinearity of the $n$ products we only need to show that 
    \[
        \big((E_{I_i}^I)^*\tilde{X}_i\big)_{(n)}\big((E_{I_j}^I)^*\tilde{Y}_j\big) = 0
    \]
    if $i \neq j$. Indeed from this follows that
    \[
        X_{(n)}Y = \big((E_{I_1}^I)^*\tilde{X}_1\big)_{(n)}\big((E_{I_1}^I)^*\tilde{Y}_1\big) + \big((E_{I_2}^I)^*\tilde{X}_2\big)_{(n)}\big((E_{I_2}^I)^*\tilde{Y}_2\big),
    \]
    which shows $E_{I_j}^I(X_{(n)}Y) = \big(E_{I_j}^IX\big)_{(n)}\big(E_{I_j}^IY\big)$ as desired.
    
    Fix $i\neq j$, by hypothesis we have 
    \[
        \big[ (E_{I_i}^I)^*\tilde{X}_i, (E_{I_j}^I)^*\tilde{Y}_j \big] = 0,
    \]
    so the $n$ products, for $n \geq 0$, are automatically $0$. To show that the other products are $0$ as well we use point 4 of Remark \ref{rmkfattiutili}. Indeed extend the products $\tilde{X}_i\tilde{Y}_j$ and $\sigma\circ\tilde{Y}_i\tilde{X}_j$ to the completion $K_{ij} = \widehat{\tilde{K}_{I_i}\otimes \tilde{K}_{I_j}}$ defined in point 2 of Remark \ref{rmkfattiutili}, by the hypothesis of commutativity, this extensions coincide. In particular, since $z-w$ is already invertible in this completed tensor product we have that for any $n$ and any $g_i \in \tilde{K}_{I_i}, g_j \in \tilde{K}_{I_j}$ we have
    \[
        \tilde{X}_i\tilde{Y}_j\bigg( \frac{g_i \otimes g_j}{(z-w)^n} \bigg) = \sigma\circ\tilde{Y}_i\tilde{X}_j\bigg( \frac{g_i \otimes g_j}{(z-w)^n} \bigg).
    \]
    This easily implies that 
    \[
        \big((E_{I_i}^I)^*\tilde{X}_i\big)_{(n)}\big((E_{I_j}^I)^*\tilde{Y}_j\big) = 0.
    \]
    also for negative $n$'s.
\end{proof}

\begin{corollary}\label{cororesexpfields}
    Given an element $x \in V^k(\g)$ consider the fields $Y^{\text{can}}_I(x) \in F_{A_I}(K_I,\tilde{U}_{\kappa}(\hat{\g}_{K_I}))$. Then for every surjection $I \twoheadrightarrow J$ and for every decomposition $I = I_1 \coprod I_2$ we have
    \begin{align*}
        \res_{I\to J} Y^{\text{can}}_I(x) &= Y^{\text{can}}_J(x), \\
        E^I_{I_j} Y^{\text{can}}_I(x) &= (Y^{\text{can}}_{I_j}(x))_{A_{I_1,I_2}}.
    \end{align*}
    where $(Y^{\text{can}}_{I_j}(x))_{A_{I_1,I_2}}$ is the image of $Y^{\text{can}}_{I_j}(x)$ in $F_{A_{I_1,I_2}}(A_{I_1,I_2}(K_{I_j}),A_{I_1,I_2}(\tilde{U}_{\kappa}(\hat{\g}_{I_j}))$.
\end{corollary}

\begin{proof}
    The first equality directly follows from the first point of Proposition \ref{propfactfields}. For the second equality it is enough to notice that for $x = X_{-1}\vac$ and $y = Y_{-1}\vac$ with $X,Y \in \g$ the fields $Y^{\text{can}}_I(x)$ and $Y^{\text{can}}_I(y)$ commute with respect to any decomposition $I = I_1 \coprod I_2$. The equality easily follows from multiple repetitions of the second point of Proposition \ref{propfactfields} and from the fact that 
    \[
    A_{I_1,I_2}(\tilde{U}_k(\hat{\g}_I)) = A_{I_1,I_2}(\tilde{U}_k(\hat{\g}_{I_1}))\hat{\otimes} A_{I_1,I_2}(\tilde{U}_k(\hat{\g}_{I_2})).
    \]
\end{proof}

\subsection{Factorization maps}

When considering the center at the critical level $Z_I(\g) = Z\big( \tilde{U}_{\kappa_c}(\hat{\g}_{K_I}) \big)$ there are natural homomorphisms
\begin{align*}
    A_J(Z_I(\g)) &\to Z_J(\g), \\
    A_{I_1,I_2}(Z_I(\g)) &\to A_{I_1,I_2}(Z_{I_1}(\g))\hat{\otimes}A_{I_1,I_2}(Z_{I_2}(\g)).
\end{align*}
induced by the corresponding maps of the completed enveloping algebras, introduced in Proposition \ref{propfactalg}
We are going to show that the above maps are actually isomorphisms, and we are going to do that by noticing that the isomorphism with the ring of functions on the space of Opers commutes with these restriction and expansion maps.

\begin{thm}\label{thmassociativity}
    The isomorphism constructed in Theorem \ref{theoremmain} makes the following diagrams commute
    \[\begin{tikzcd}
	{A_I[Op_{^L\g,I}]} && {Z_I(\g)} \\
	\\
	{A_J[Op_{^L\g,J}]} && {Z_J(\g)}
	\arrow[from=1-1, to=3-1]
	\arrow[from=1-3, to=3-3]
	\arrow["\simeq"', from=3-1, to=3-3]
	\arrow["\simeq", from=1-1, to=1-3]
\end{tikzcd}\]
    \[\begin{tikzcd}
	{A_I[Op_{^L\g,I}]} && {Z_I(\g)} \\
	\\
	{A_{I_1,I_2}[Op_{^L\g,I_1}\times Op_{^L\g,I_2}]} && {A_{I_1,I_2}(Z_{I_1}(\g))\hat{\otimes} A_{I_1,I_2}(Z_{I_2}(\g))}
	\arrow[from=1-1, to=3-1]
	\arrow[from=1-3, to=3-3]
	\arrow["\simeq"', from=3-1, to=3-3]
	\arrow["\simeq", from=1-1, to=1-3]
\end{tikzcd}\]
    In particular since the collection of the algebras of functions on the space of Opers is a factorization algebra we obtain that the family $Z_I(\g)$ is a factorization algebra as well and that the isomorphism of Theorem \ref{theoremmain} is an isomorphism of factorization algebras.
\end{thm}

\begin{proof}
    Let's focus on the first diagram. It is easy to see that the functions $v_i^*(g)$ we constructed on $Op_{^L\g,I}$, for $v_i$ an homogeneous basis of $V^{\text{can}}$ and $g \in K_I$ are sent via the left vertical map to the functions $v_i^*(\res_{I \to J}g)$. In addition by Proposition \ref{propfactfields} we have the following diagram
    \[\begin{tikzcd}
	{v_i^*(g)} && {Y^{\text{can}}_I(v_{i,-1}^*)(g)} \\
	\\
	{v_i^*(\res_{I\to J}g)} && {Y^{\text{can}}_J(v_{i,-1}^*)(\res_{I \to J}g)}
	\arrow[maps to, from=1-1, to=3-1]
	\arrow[maps to, from=1-1, to=1-3]
	\arrow[maps to, from=1-3, to=3-3]
	\arrow[maps to, from=3-1, to=3-3]
\end{tikzcd}\]
    This shows that the first square is commutative on a set of topological generators for $A_I[Op_{^L\g,I}]$, since all maps are continuous homomorphisms of complete topological algebras we get the commutativity we wanted.
    
    A similar argument can be carried out for the second square. It is pretty clear that the functions $v_i^*(g)$ are sent via the left vertical arrow to the functions $v_i^*(E_{I_1}^Ig) + v_i^*(E_{I_2}^Ig)$, by Proposition \ref{propfactfields} we have that the elements $Y^{\text{can}}_I(v_{i,-1}^*)(g)$ are sent to $Y^{\text{can}}_{I_1}(v_{i,-1}^*)(E_{I_1}^Ig) + Y^{\text{can}}_{I_1}(v_{i,-1}^*)(E_{I_1}^Ig)$, following the diagram we get the following square 
    \[\begin{tikzcd}
	{v_i^*(g)} && {Y^{\text{can}}_I(v_{i,-1}^*)(g)} \\
	\\
	{v_i^*(E_{I_1}^Ig) + v_i^*(E_{I_2}^Ig)} && {Y^{\text{can}}_{I_1}(v_{i,-1}^*)(E_{I_1}^Ig) + Y^{\text{can}}_{I_2}(v_{i,-1}^*)(E_{I_2}^Ig)}
	\arrow[maps to, from=1-1, to=3-1]
	\arrow[maps to, from=1-1, to=1-3]
	\arrow[maps to, from=1-3, to=3-3]
	\arrow[maps to, from=3-1, to=3-3]
\end{tikzcd}\]
    and again, this shows that the diagram is commutative on a set of generators for $A_I[Op_{^L\g,I}]$, as before, this proves that the hole square is commutative.
\end{proof}

\bibliographystyle{apalike}
\bibliography{nsingularities}{}

\end{document}